\documentclass[preprint,11pt]{article}
\usepackage{amssymb}
\usepackage{amsfonts}
\usepackage{mathrsfs}
\usepackage{amssymb,amsfonts,amsmath}
\usepackage{graphicx,graphics,epstopdf}
\usepackage{url}
\graphicspath{{./figures/}}
\usepackage{latexsym}
\usepackage[ruled,vlined,linesnumbered,noresetcount]{algorithm2e}
\usepackage{algpseudocode,algorithmicx}
\algdef{SE}[DOWHILE]{Do}{doWhile}{\algorithmicdo}[1]{\algorithmicwhile\ #1}%
\usepackage{esint}
\usepackage{multirow}
\usepackage{caption}
\usepackage{subcaption}
\usepackage{ctable} 
\usepackage{times}
\usepackage{multicol,enumerate}
\usepackage{epstopdf}
\usepackage{color,xcolor}
\usepackage{fullpage}
\usepackage{setspace}
\usepackage[english]{babel}
\usepackage[version=3]{mhchem}
\usepackage{amsmath, amssymb, amsthm}
\usepackage{verbatim, latexsym}
\usepackage{colortbl}
\usepackage{multirow}
\usepackage[OT2,OT1]{fontenc}
\usepackage{float}
\definecolor{black}{rgb}{0,0,0}

\definecolor{red}{rgb}{1,0,0}
\newcommand\red[1]{\textcolor{red}{#1}}
\definecolor{blue}{rgb}{0,0,1}

\newcommand\cyr
{
	\renewcommand\rmdefault{wncyr}
	\renewcommand\sfdefault{wncyss}
	\renewcommand\encodingdefault{OT2}
	\normalfont
	\selectfont
}
\DeclareTextFontCommand{\textcyr}{\cyr}

\newcommand{\vertiii}[1]{{\left\vert\kern-0.25ex\left\vert\kern-0.25ex\left\vert #1
    \right\vert\kern-0.25ex\right\vert\kern-0.25ex\right\vert}}

\def\XXint#1#2#3{{\setbox0=\hbox{$#1{#2#3}{\int}$ }
		\vcenter{\hbox{$#2#3$ }}\kern-.6\wd0}}

\def\Dbe{{\partial^{\alpha}_t}}
\def\Dbel{{\partial^\alpha_{\mathcal{N}}}}

\def\Dbela{{\bar{\partial}^\alpha_{\tau_c,\mathcal{N}}}}
\def\Dbehaxx{{\bar{\partial}^\alpha_{\tau_f,\mathcal{F}}}}
\def\Dbelaxx{{\bar{\partial}^\alpha_{\tau_f,\mathcal{N}}}}
\def\Dbeh{{\partial^\alpha_{\mathcal{F}}}}
\def\Dap{{\bar{\partial}^\alpha_{\tau_c}}}

\newcommand{\nexp}{N_{\mathrm{exp}}}
\newcommand{\etamaxmin}[1]{\eta_{\text{#1}}}

\newcommand{\dx}{\,\mathrm{d}x}

\newcommand{\prnt}[1]{\left( #1 \right)}

\newcommand{\norm}[1]{\left\|#1\right\|}

\newcommand{\normE}[2]{\norm{#1}_{H^{1}_{\kappa}\prnt{#2}}}

\newcommand{\normL}[2]{\norm{#1}_{#2}}

\newcommand{\normLii}[2]{\norm{#1}_{L^2_{\widetilde{\kappa}^{-1}}\prnt{#2}}}

\newcommand{\normI}[2]{\norm{#1}_{L^\infty\prnt{#2}}}

\newcommand{\Const}[1]{{\rm C}_{\mathrm{#1}}}

\newtheorem{theorem}{Theorem}[section]
\newtheorem{corollary}{Corollary}[section]
\newtheorem{assumption}{Assumption}[section]
\newtheorem{remark}{Remark}[section]
\newtheorem{lemma}{Lemma}[section]
\newtheorem{definition}{Definition}[section]

\numberwithin{equation}{section}
\usepackage{soul}
\setstcolor{red}

\title{Wavelet-based Edge Multiscale Parareal Algorithm for subdiffusion equations with heterogeneous coefficients in a large time domain}
\author{Guanglian Li\thanks{Department of Mathematics, The University of Hong Kong, Pokfulam Road, Hong Kong.  (\texttt{lotusli@maths.hku.hk}, \texttt{lotusli@hku.hk}). G.L. acknowledges support from Newton International Fellowships Alumni following-on funding awarded by The Royal Society, Young Scientists fund (Project number: 12101520) by NSFC, China, and Early Career Scheme (Project number: 27301921), RGC, Hong Kong.}
}
\begin{document}
	\maketitle
	\begin{abstract}
We present the Wavelet-based Edge Multiscale Parareal (WEMP) Algorithm, recently proposed in [Li and Hu, {\it J. Comput. Phys.}, 2021], for efficiently solving subdiffusion equations with heterogeneous coefficients in long time. This algorithm combines the benefits of multiscale methods, which can handle heterogeneity in the spatial domain, and the strength of parareal algorithms for speeding up time evolution problems when sufficient processors are available. Our algorithm overcomes the challenge posed by the nonlocality of the fractional derivative in previous parabolic problem work by constructing an auxiliary problem on each coarse temporal subdomain to completely uncouple the temporal variable. We prove the approximation properties of the correction operator and derive a new summation of exponential to generate a single-step time stepping scheme, with the number of terms of $\mathcal{O}(|\log{\tau_f}|^2)$ independent of the final time, where $\tau_f$ is the fine-scale time step size. We establish the convergence rate of our algorithm in terms of the mesh size in the spatial domain, the level parameter used in the multiscale method, the coarse-scale time step size, and the fine-scale time step size. Finally, we present several numerical tests that demonstrate the effectiveness of our algorithm and validate our theoretical results.	
\end{abstract}
\noindent{\bf Keywords:}
multiscale, heterogeneous, long time, edge, wavelets, parareal, time-fractional, diffusion
	\section{Introduction}
	We consider in this paper a new efficient multiscale parareal algorithm for time-fractional diffusion problems with heterogeneous coefficients, which reads to find $u\in V:=H^1_0(D)$, satisfying
\begin{equation}\label{eqn:pde}
\left\{
\begin{aligned}
\Dbe u & =\nabla\cdot \left(\kappa\nabla u \right)+f&& \text{in } D,
\quad t\in \left( 0,T\right] \\
u& =0 && \text{on }\partial D, \;t\in \left( 0,T\right] \\
u(\cdot,0) & =u_0  && \text{in } D.
\end{aligned}
\right .
\end{equation}
Here, $\Dbe u$ ($\alpha\in (0,1)$) refers to the left-sided Caputo fractional derivative of order $\alpha$ of the function $u(\cdot, t)$, defined by (see, e.g. \cite[p. 91, (2.4.1)]{KilbasSrivastavaTrujillo:2006} or \cite[p. 78]{Podlubnybook}),
\begin{equation}\label{eq:frac}
  \Dbe u (\cdot,t) = \frac{1}{\Gamma(1-\alpha)}\int_0^t \frac{1}{(t-s)^\alpha}\frac{\partial u}{\partial s}(\cdot,s)\, \mathrm{ds}
\end{equation}
with $\Gamma(\cdot)$ being Euler's Gamma function: $\Gamma(x)=\int_{0}^{\infty}s^{x-1}\mathrm{e}^{-s}\mathrm{d}s$ for $x>0$.

Furthermore, $D\subset\mathbb{R}^d$ for $d=2,3$ is bounded with certain smoothness and $\kappa$ is piecewise smooth as described in Assumption \ref{ass:coeff}. The force term $f\in W^{2,\infty}([0,T]; L^2(D))$, the initial data $u_0\in \dot{H}^{2\sigma}(D)\cap H^1_0(D)$ with $\sigma\in (0,1]$. The subspace $\dot{H}^{2\sigma}(D)\subset L^2(D)$ is a Hilbert space to be defined in Section \ref{section:problem setting}. To simplify the notation, let $I:=[0,T]$. The main challenges for solving Problem \eqref{eqn:pde} numerically lie in: 1) the existence of multiple scales in the coefficient $\kappa$, since resolving the problem to the finest scale as mandatory for classical numerical methods would incur huge number of degrees of freedom in the spatial domain. 2) the nonlocality of the fractional derivative $\Dbe$ \eqref{eq:frac} makes the numerical computation and storage extremely costly especially for the case of long time simulation with a large final time $T$.

On the one hand, the multiscale nature of the solution in the spatial domain as a consequence of the heterogeneous coefficient $\kappa$ makes the computation challenging and deteriorates the storage cost. Nevertheless, the accurate description of many important applications, e.g., composite materials, porous media, ground water modeling and reservoir simulation, inevitably involves many heterogeneous parameters with multiple scales. They can have both multiple inseparable scales and high-contrast in order to reflect the nature of the practical applications. Because of this, the classical numerical treatment becomes prohibitively expensive especially for the unsteady problem \eqref{eqn:pde}. Consequently, various multiscale model reduction techniques have been proposed in the literature over the past few decades. Among them are multiscale finite element methods (MsFEMs), heterogeneous multiscale methods (HMMs), variational multiscale methods, flux norm approach, generalized multiscale finite element methods (GMsFEMs) and localized orthogonal decomposition (LOD) \cite{ altmann2021numerical,MR2721592,MR1979846, egh12,MR1455261,MR1660141, li2017error, MR3246801}. Wavelet-based Edge Multiscale Finite Element Methods (WEMsFEMs) was proposed recently by the author, c.f. Algorithm \ref{algorithm:wavelet}, which facilitates deriving a rigorous convergence rate with merely mild assumptions \cite{fu2018,fu2019wavelet, li2019convergence}. This method utilizes wavelets as the basis functions over the coarse edges to define local multiscale basis functions, which facilitates transforming the approximate rate over the edges to the convergence rate in each local region. Then the Partition of Unity Method (PUM) \cite{Babuska2} is applied to derive the global convergence rate. The motivation for using wavelets as the ansatz space is that the existence of heterogeneity causes the solution has a low regularity, and wavelets are known to be efficient in approximating functions with such unfavorable property.
	
One the other hand, the nonlocality of the fractional derivative makes long-time simulation prohibitive, which are essential in many applications such as the rare event sampling and molecular dynamics simulation. The L1 scheme is one of the well known time stepping schemes that uses piecewise linear approximation on each time step. High order schemes have been designed and analyzed in the literature to obtain a higher order convergence rate when the solution has sufficient temporal regularity. Furthermore, these time stepping schemes can be combined with approximation by exponential sums or the proper orthogonal decomposition method in the history part to reduce the storage costs \cite{MR3601002,WU2018135,zhu2019fast}. The parareal algorithm is one of the most popular and success algorithms for long time simulation, and it has been applied to many problems including the financial mathematics and quantum chemistry \cite{Bal_Mayday_AmericanPut,parareal_ChemicalKinetics2010, time_decomposed_parallel2003, time_parallel2006, parareal_NS_mayday,  parareal_ocean_model2008,parareal_ModelReduction_mayday2007,  parareal_QuantumSystem_mayday2007}. The parareal algorithm facilitates speeding up the numerical simulation of time evolution problems when sufficient processors are available \cite{bal2005convergence} by using an iterative solver based on a cheap inaccurate sequential coarse-scale time stepping scheme and expensive accurate fine-scale time stepping scheme that can be performed in parallel. This algorithm was introduced by Lions, Mayday and Turinici \cite{lions_mayday_2001_parareal}. Coupling of parareal algorithm and some other techniques has been developed in many literature, see \cite{2016PararealMultiscale,  parareal_integrator_2010,parareal_multigrid_2014, 2013Micro_Macro_parareal}. Among them is  the coupling of parareal algorithm and model reduction techniques \cite{2013Micro_Macro_parareal}, wherein a micro-macro parareal algorithm for the time-parallel integration of multiscale-in-time systems is introduced to solve singularly perturbed ordinary differential equations. A new coupling strategy for the parareal algorithm with multiscale integrators is introduced in \cite{2016PararealMultiscale}.
	
As a continuation of \cite{li2020wavelet}, we incorporate the parareal algorithm into WEMsFEM to numerically calculate the time evolution problems efficiently. The main contribution of this paper is three-fold. Firstly, we derive a new model reduction in the temporal domain by means of summation of exponentials to handle long time situation with the number of terms independent of the final time. Secondly, the stability of the resulting time stepping scheme is rigorously analyzed, with a very mild requirement on the accuracy of the approximation by summation of exponentials. Thirdly, together with the multiscale model reduction in the spatial domain, we obtain a stable iterative scheme with a very fast convergence rate. Furthermore, the convergence analysis is rigorously justified.

Our main algorithm is as follows. We first construct a multiscale space $V_{\text{ms},\ell}$ based on WEMsFEM with $\ell$ as the wavelets level parameter. Then we design a new parareal algorithm that corrects the current solution only and treats the history information as a parameter.
This parameter is updated once by a formula after obtaining a corrected current solution.  In the meanwhile, we will use $V_{\text{ms},\ell}$ as the ansatz space to save computation and storage costs. Note that two parareal algorithms were proposed in \cite{WU2018135}, both of which involved correction operator for the history information. On the one hand, our numerical experiments demonstrate much faster convergence for our new parareal algorithm. On the other hand, our new algorithm saves more computational costs, and we can derive an explicit convergence rate.
The convergence rate of this algorithm is presented in Theorem \ref{prop:wavelet-based}. We proved for any $\sigma\in (0,1)$,
	\begin{equation*}
		\begin{aligned}
		\normL{u(\cdot,T^n)-U_{k}^{n}}{D}
		&\lesssim \left({\tau_f}^{\sigma\alpha}+\eta(H,\ell)|\log \eta(H,\ell)|\right)\|f\|_{W^{2,\infty}(0,T;L^2(D))}\\
&+\left({\tau_f}^{\sigma\alpha}+\eta(H,\ell)|\log \eta(H,\ell)|(T^n)^{-\alpha(1-\sigma)}+\Pi_{j=0}^{k}\Big(\frac{\tau_c}{T^{n-j}}\Big)^{\alpha\sigma} \right)|u_0|_{2\sigma}.
		\end{aligned}
		\end{equation*}
where $u(\cdot,T^n)$ and $U_{k}^{n}$ are the exact solution at $T^n=n\times\tau_c$ for $n=2,\cdots$ and our proposed WEMP algorithm at iteration $k$. The notations $\tau_c$ and $\tau_f$ represent for the coarse time step size and fine time step size, respectively. $H$, $\ell$ and $k$ are the space domain mesh size, the level parameter and iteration number. The parameter $\eta(H,\ell)$ is defined in \eqref{eq:etaHl} that reflects the approximation property of the multiscale ansatz space $V_{\text{ms},\ell}$. This implies that taking $\ell=\mathcal{O}(|\log H|)$ and $k=\mathcal{O}(|\log\tau_f/\log\tau_c|)$, we obtain $\mathcal{O}(H^2+\tau_f^{\sigma\alpha})$ error.

To demonstrate the performance of our proposed algorithm we present several numerical tests to first investigate the influence of approximation by exponential sums on the accuracy of our method. Then a series of simulations are made to show the convergence of our method. Finally, we present several numerical results with large final time $T$.
	
Our paper is structured as follows. In Section \ref{section:problem setting}, we introduce the Galerkin-L1-scheme to Problem \eqref{eqn:pde}, and provide an overview of multiscale model reduction in the spatial domain and the parareal algorithm. In Section \ref{sec:multiscale}, we discuss the construction of the multiscale ansatz space $V_{\text{ms},\ell}$ by WEMsFEMs based upon \cite{fu2018,fu2019wavelet, li2019convergence} and recap its approximation properties. Section \ref{sec:history} is devoted to deriving an efficient numerical scheme for the history information in the fractional derivative based on the Sinc approximation, which facilitates designing a new multiscale-SOE-scheme to solve Problem \eqref{eqn:pde}. The main ingredient of our proposed algorithm is presented in Section \ref{sec:WEMP} and its convergence rate is derived in Section \ref{sec:convergence}. Extensive numerical tests are presented in Section \ref{sec:num}. Finally, we complete our paper with concluding remarks in Section \ref{sec:conclusion}.
	
\section{Problem setting} \label{section:problem setting}
We first present in Section \ref{subsec:gl1} the full discretization of problem \eqref{eqn:pde} by the standard conforming Galerkin piecewise linear element in the spatial domain under a tiny fitted mesh $\mathcal{T}_h$ of mesh size $h\ll 1$, and L1 scheme in the temporal domain using a uniform discretization of mesh size $\tau_f\ll 1$. Then we propose, on the one hand, to replace this expensive fitted mesh $\mathcal{T}_h$ with a much cheaper unfitted mesh $\mathcal{T}_H$ with mesh size $H\gg h$ by construction of multiscale ansatz space in Section \ref{subsec:multiscale}, and on the other hand, to further reduce the computational complexity by building iterative time stepping scheme over the coarse subintervals based on the parareal algorithm in Section \ref{subsec:parareal}.

To start with, we define the Hilbert space $\dot{H}^s(D)$, which is analogous to \cite[Chapter 3]{thomee1984galerkin}. Let $\{(\lambda_m,\phi_m)\}_{m=1}^{\infty}$ be the eigenpairs of the following eigenvalue problems with the eigenvalues arranged in a nondecreasing order,
\begin{align*}
\mathcal{L}\phi_m&:=-\nabla\cdot(\kappa\nabla\phi_m)=\lambda_m \phi_m&& \text{ in } D\\
\phi_m&=0 &&\text{ on }\partial D.
\end{align*}
Note that the eigenfunctions $\{\phi_m\}_{m=1}^{\infty}$ form an orthonormal basis in $L^2(D)$, and consequently, each $v\in L^2(D)$ admits the representation
$v=\sum_{m=1}^{\infty}(v,\phi_m)_D\phi_m$ with $(\cdot,\cdot)_D$ being the inner product in $L^2(D)$.
The associated norm in $L^2(D)$ is denoted as $\|\cdot\|_D$.

The Hilbert space $\dot{H}^s(D)\subset L^2(D)$ is defined by
\begin{align}\label{def:dotH}
\dot{H}^s(D)=\{v\in L^2(D): \sum_{m=1}^{\infty}\lambda_m^s|(v,\phi_m)_D|^2<\infty\}
\end{align}
with its associated norm being $|v|_s:=\Big( \sum_{m=1}^{\infty}\lambda_m^s|(v,\phi_m)_D|^2\Big)^{1/2}$ for all $v\in \dot{H}^s(D)$.
\subsection{Galerkin-L1-scheme}\label{subsec:gl1}
To discretize problem \eqref{eqn:pde}, we first introduce fine and coarse discretization in the spatial domain $D$. Let $\mathcal{T}_H$ be a regular partition of the domain $D$ into finite elements (triangles, quadrilaterals, tetrahedral, etc.) with a mesh size $H$, which is determined by the desired accuracy of our numerical method. We refer to this partition and its elements as coarse mesh and the coarse elements. Then each coarse element is further partitioned into a union of connected fitted fine elements. This fine-scale partition is denoted by $\mathcal{T}_h$ with $h$ being its mesh size. Let $\mathcal{F}_h$ (or $\mathcal{F}_H$) be the collection of all edges in $\mathcal{T}_h$ (or $\mathcal{T}_H$). Over the fine mesh $\mathcal{T}_h$, let $V_h$ be the conforming piecewise linear finite element space:
\[
V_h:=\{v\in V: V|_{E}\in \mathcal{P}_{1}(E) \text{ for all } E\in \mathcal{T}_h\},
\]
where $\mathcal{P}_1(E)$ denotes the space of linear polynomials on the fine element $E\in\mathcal{T}_h$.

The time interval $I:=[0,T]$ is decomposed into a sequence of uniform coarse subintervals $[T^n,T^{n+1}]$ for $n=0,1,\cdots, M_{c}$ of size $\tau_c$ with $\tau_c:=T/M_{c}$ for some $M_{c}\in \mathbb{N}_{+}$ and $T^0:=0$. The total number of coarse steps $M_c$ is governed by the available number of processors and the aimed speedup. Each coarse time interval $[T^{n}, T^{n+1}]$ is further uniformly discretized with a fine time step $\tau_f$ controlled by the expected accuracy in the temporal discretization.
Let $t_n=n\times\tau_f$ for $n=0,1,\cdots,M_{f}$ with $M_{f}:=T\times\tau_f^{-1}$. Note that we assume $\tau_c\gg \tau_f$ throughout of this paper.

The $L1$ scheme \cite{LIN20071533, SUN2006193} is one of the most popular and simplest discretization schemes for \eqref{eq:frac} the temporal domain $I$, which is based on the idea of using piecewise constant approximation to approximate the first derivative. The scheme reads
  \begin{align}\label{eq:l1scheme}
  \tilde{\partial}_{\tau_f}^{\alpha}v(x,t_{n+1})
  &=\sum_{j=0}^{n}b_j\frac{v(x,t_{n+1-j})-v(x,t_{n-j})}{\tau_f^{\alpha}c_{\alpha}}
  \end{align}
with $c_{\alpha}=\Gamma(2-\alpha)$ and $b_j=(j+1)^{1-\alpha}-j^{1-\alpha}$ for $j=0,1,\cdots,n$.

The next lemma provides a lower bound of the coefficient for the L1 scheme:
\begin{lemma}
The coefficient from \eqref{eq:l1scheme} has the following lower bound,
\begin{align*}
b_{j-1}-b_j\geq \alpha(1-\alpha)(n+1)^{-\alpha-1} \text{ for all } j=1,\cdots,n.
\end{align*}
\end{lemma}
\begin{proof}
Let $r(x):=(x+1)^{1-\alpha}-x^{1-\alpha}$ for $x>0$. Then $r(x)$ is decreasing and convex. Consequently,
\begin{align*}
b_{j-1}-b_{j}\geq b_{n-1}-b_{n}\geq -r'(n)=(1-\alpha)\left(n^{-\alpha}-(n+1)^{-\alpha}\right).
\end{align*}
Note that $r_1(x)=x^{-\alpha}$ for $x>0$ is decreasing and convex, this yields, 
 \begin{align*}
b_{j-1}-b_{j}\geq 
(1-\alpha)\left(n^{-\alpha}-(n+1)^{-\alpha}\right)
\geq\alpha(1-\alpha)(n+1)^{-\alpha-1}.
\end{align*}
We have proved the desired assertion.
\end{proof}
 We use conforming Galerkin method for the discretization in the spatial domain $D$ throughout this paper. Then the Galerkin-L1-scheme reads to find $u_h^{n}\in V_h$ for $n=1,2,\cdots,M_{f}$, satisfying
	\begin{equation}\label{eqn:weakform_h}
	\left\{
	\begin{aligned}
	(\tilde{\partial}_{\tau_f}^{\alpha}u^n_h,v_h)_D+a(u_h^{n},v_h)&=(f(t_n),v_h)_{D} \quad \text{ for all } v_h\in V_h\\
	u_h^0&=I_h u_0.
	\end{aligned}
	\right.
	\end{equation}
	Here, the bilinear form $a(\cdot,\cdot)$ on $V\times V$ is defined by
	\[
	a(v_1,v_2):=\int_{D}\kappa \nabla v_1\cdot \nabla v_2\dx \text{ for all } v_1,v_2\in V
	\]
with $I_h$ being the $L^2(D)$-projection from $V$ to $V_h$. Furthermore, we define the energy norm $\|v\|_{H^1_{\kappa}(D)}:=\sqrt{a(v,v)}$ for all $v\in V$.
	
The fine-scale solution $u_h^{n}$ will serve as a reference solution in Section \ref{sec:num}. Note that due to the presence of multiple scales in the coefficient $\kappa$, the fine-scale mesh size $h$ should be commensurate with the smallest scale and thus it can be very small in order to obtain an accurate solution for each time step. In addition, the well-known limited spatial smoothness property of Problem \eqref{eqn:pde} makes high-order methods in the spatial domain infeasible \cite{Sakamoto-Yamamoto11}. These result in huge computational complexity, and consequently, more efficient methods are in great demand. 
Moreover, the storage cost of order $\mathcal{O}(M_{f} h^{-d})$ associated with the L1 scheme \eqref{eq:l1scheme} becomes tremendous under such scenario. These make a direct solver for $u_h^{n}$ prohibitive.
	
\subsection{Multiscale-L1-scheme}\label{subsec:multiscale}		
Multiscale model reduction methods aim at replacing the fine-scale mesh $\mathcal{T}_h$ with a coarse-scale mesh $\mathcal{T}_{H}$ when solving Problem \eqref{eqn:weakform_h} numerically, which, meanwhile, maintains a certain accuracy. To describe it, we need a few notations. The vertices of $\mathcal{T}_H$ are denoted by $\{O_i\}_{i=1}^{N}$, with $N$ being the total number of coarse nodes. The coarse neighborhood associated with the node $O_i$ is denoted by
\begin{equation} \label{neighborhood}
\omega_i:=\bigcup\{ K\in\mathcal{T}_H: ~ O_i\in \overline{K}\}.
\end{equation}
	We refer to Figure~\ref{schematic} for an illustration of neighborhoods and elements subordinated to the coarse
	discretization $\mathcal{T}_H$. Throughout, we use $\omega_i$ to denote a coarse neighborhood. Furthermore, Let $\mathcal{F}_h$ (or $\mathcal{F}_H$) be the collection of all edges in $\mathcal{T}_h$ (or $\mathcal{T}_H$) and let $\mathcal{F}_h(\partial \omega_i)$ (or $\mathcal{F}_H(\partial \omega_i)$) be the restriction of $\mathcal{F}_h$ on $\partial\omega_i$ (or $\mathcal{F}_H$ on $\partial\omega_i$).
	\begin{figure}[htb]
		\centering
		\includegraphics[width=0.65\textwidth]{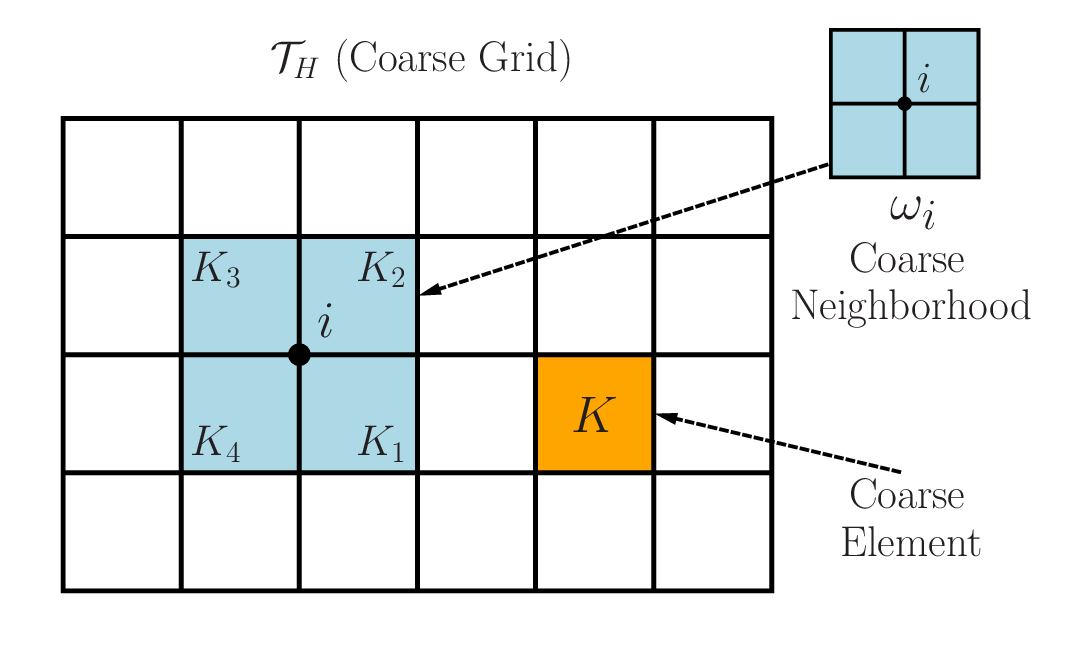}
		\caption{Illustration of a coarse neighborhood and coarse element.}
		\label{schematic}
	\end{figure}

Next, the multiscale ansatz space $V_{\text{ms}}\subset V_h$ is constructed that are locally supported on certain coarse neighborhood $\omega_i$. Then the multiscale-L1-scheme aims at seeking $u_{\text{ms}}^{n}\in V_{\text{ms}}$ for $n=1,\cdots,M_{f}$, s.t.,
\begin{equation}\label{eq:multiscale}
\left\{
\begin{aligned} (\tilde{\partial}_{\tau_f}^{\alpha}u_{\text{ms}}^{n},v_{\text{ms}})_D+a(u_{\text{ms}}^{n},v_{\text{ms}})&=(f(\cdot, t_n),v_{\text{ms}})_{D}
\quad \text{for all} \, \, v_{\text{ms}}\in
V_{\text{ms}}\\
u_{\text{ms}}^0&=I_{\text{ms}} u_0.
\end{aligned}
\right.
\end{equation}
Here, $V_{\text{ms}}$ denotes the multiscale space spanned by these multiscale basis functions and $I_{\text{ms}}$ is a projection operator from $V$ to $V_{\text{ms}}$.

The construction of $V_{\text{ms}}$ will be presented in Section \ref{sec:multiscale}. Note that we need a very tiny fine-scale time step size $\tau_f$ to guarantee a reasonable approximation property of  $u_{\text{ms}}^{n}$ to $u(\cdot,t_n)$ for $n=1,\cdots,M_{f}$ due to, e.g., the singularity of the solution $u(\cdot,t)$ at $t=0$ when the source term $f$ fails to belong to certain smooth functional space. Consequently, the computational complexity and the storage requirement of the multiscale-L1-scheme \eqref{eq:multiscale} can be still extremely expensive. For this reason, we introduce in Section \ref{subsec:parareal} the parareal algorithm to significantly improve its efficiency.

\subsection{Parareal algorithm}\label{subsec:parareal}
The main goal of the parareal algorithm is to obtain a high-order approximation of the solution at the coarse time step $u(\cdot,T^n)$ for $n=1,\cdots,M_c$ iteratively under the coarse time stepping scheme, which is corrected by parallel local fine time stepping scheme over each coarse subinterval. To demonstrate this algorithm, we need to introduce the one-step coarse propagator and fine propagator
\begin{align*}
\mathcal{G}(T^n, U^n) \text{ and }\mathcal{F}(T^n, U^n)
\end{align*}
for given initial data $U^n$ over the coarse interval $[T^n, T^{n+1}]$. The outcome of these propagators is an approximation to $u(\cdot,T^{n+1})$ for $n=0,1,\cdots, M_{c}-1$.

\begin{algorithm}[H]
\caption{Parareal algorithm}
\label{algorithm:parareal}
\KwData{Initial data $u_0$; tolerance $\delta$ and initial solution $U_0^n$ for $n=0,\cdots, M_c$.}
		
\KwResult{$U^n_k$ for $n=0,\cdots, M_c$.}

\begin{algorithmic}[1]
	\State Set $k=1$.
    \Do
\State  Set the initial data $U_k^0=I_{\text{ms}} u_0$.
\State Parallel compute $U^n_k$ on each local time subinterval $[T^{n-1}, T^n]$ for $n=1,\cdots,M_c$,
\begin{align*}
U^n_k=\mathcal{G}(T^{n-1}, U^{n-1}_{k-1}).
\end{align*}
	\State Parallel compute $u^n_k$ on each local time subinterval $[T^{n-1}, T^n]$ for $n=1,\cdots,M_c$,
\begin{align*}
u^n_k=\mathcal{F}(T^{n-1}, U^{n-1}_{k-1}).
\end{align*}
    \State Parallel compute the solution jumps for $n=1,\cdots,M_c$,
\begin{align*}
\mathcal{S}(T^{n-1}, U_{k-1}^{n-1})=u^n_k-U^n_k.
\end{align*}
\State Sequentially compute the corrected coarse solutions by propagation of the jumps
\begin{align*}
U_k^n=\mathcal{S}(T^{n-1}, U_{k-1}^{n-1})+\mathcal{G}(T^{n-1}, U^{n-1}_{k}).
\end{align*}
\State Estimate the error
			\[{\rm err}:=1/{M_c}\sum\limits_{n=1}^{M_{c}}\|U_{k}^{n}-U_{k-1}^{n}\|_{\ell_2}.\]
\State $k\gets k+1$.
\doWhile{$\text{err}>\delta$}
\end{algorithmic}
\end{algorithm}
We present in Algorithm \ref{algorithm:parareal} the main ingredient of this method, which includes a parallel computation of the local jumps in Step 6 and the sequential propagation of the local jumps in Step 7.  One can observe from \cite[Theorem 2.1]{Sakamoto-Yamamoto11} that the solution operator to problem \eqref{eqn:pde} is not a semigroup. Consequently, the fine propagator $\tilde{\partial}_{\tau_f}^{\alpha}$ defined by the L1 scheme \eqref{eq:l1scheme} requires all the information from the previous time steps. This results in a huge computational and storage cost.  To reduce this cost, we will on the one hand, decrease the number of terms in the history part by means of approximation by exponential sums. One the other hand, the length of each term will be reduced by multiscale approximation.

The goal of this paper is to design an efficient numerical algorithm to approximate $u(\cdot,t)$ such that it can be solved
on the coarse spatial mesh $\mathcal{T}_{H}$, without resolving the microscale information in the coefficient $\kappa$ using a rather coarse time step $\tau_c\gg \tau_f$ that can be computational inexpensive. In the meanwhile, the storage cost can be reduced hugely. This is achieved by a combination of the multsicale method and parareal algorithm as proposed in \cite{li2020wavelet} for parabolic problems with heterogeneous coefficients. Compared with parabolic problems, the main difficulty encountered here is that the solution operator of the time fractional diffusion problem is not a semigroup.
This will be overcome in Section \ref{sec:history} by a new approximation by exponential sums.

\section{Multiscale ansatz space}\label{sec:multiscale}
This section is concerned with the construction of the multiscale space by means of the Wavelet-based Edge Multiscale Finite Element Methods (WEMsFEMs) \cite{fu2018,fu2019wavelet, li2019convergence}. The main idea is to use wavelets as the boundary data to construct local multiscale basis functions on each coarse neighborhood $\omega_i$. Then the multiscale ansatz space $V_{\text{ms},\ell}$ is constructed by multiplying these local multiscale basis functions with the partition of unity functions defined over the coarse patches $\{\omega_i\}_{i=1}^N$.

Let $\{ \chi_i \}_{i=1}^{N}$ be the partition of unity functions, which are the standard multiscale basis functions on each coarse element $K\in \mathcal{T}_{H}$ defined by
\begin{alignat}{2} \label{pou}
-\nabla\cdot(\kappa\nabla\chi_i) &= 0  &&\quad\text{ in }\;\;K \\
\chi_i &= g_i &&\quad\text{ on }\partial K. \nonumber
\end{alignat}
Here, $g_i$ is affine over $\partial K$ with $g_i(O_j)=\delta_{ij}$ for all $i,j=1,\cdots, N$. Recall that $\{O_j\}_{j=1}^{N}$ are the set of coarse nodes on $\mathcal{T}_{H}$.	

The construction of the multiscale test space $V_{\text{ms},\ell}$ is presented in Algorithm \ref{algorithm:wavelet}. Given the level parameter $\ell\in \mathbb{N}$, and the type of wavelets on each edge of the coarse neighborhood $\omega_i$, one can obtain the local multiscale space $V_{i,\ell}$ on $\omega_i$ by solving $2^{\ell+2}$  local problems in parallel in Step 2. One local multiscale basis function that is vanishing over $\partial\omega_i$ is calculated in parallel. In Step 4, we can use these local multiscale space to build the global multiscale space $V_{\text{ms},\ell}$ by multiplying the partition of unity functions $\chi_i$. 

\begin{algorithm}[H]
\caption{Wavelet-based Edge Multiscale Finte Element Method (WEMsFEM)}
\label{algorithm:wavelet}
\KwData{ The level parameter $\ell\in \mathbb{N}$; coarse neighborhood $\omega_i$ and its four coarse edges $\Gamma_{i,k}$ with $k=1,2,3,4$, i.e., $\cup_{k=1}^{4}\Gamma_{i,k}=\partial\omega_i$; the subspace $V_{\ell,k}\subset L^2(\Gamma_{i,k})$ up to level $\ell$ on each coarse edge $\Gamma_{i,k}$.}
		
\KwResult{Multiscale space $V_{\text{ms},\ell}$.}
\begin{algorithmic}[1]
\State Define the local edge basis functions. Let $V^{i}_{\ell}:=\oplus_{k=1}^{4}V_{\ell,k}:=\text{span}\{v_k: k=1,\cdots, 2^{\ell+2}\}$. Then the number of basis functions in $V^{i}_{\ell}$ is $4\times 2^{\ell}=2^{\ell+2}$.
\State Calculate local multiscale basis $\mathcal{L}^{-1}_i (v_k)$ for all $k=1,\cdots,2^{\ell+2}$.
Here, $\mathcal{L}^{-1}_i (v_k):=v$ satisfies
\begin{equation*}
\left\{\begin{aligned}
\mathcal{L}_i v&:=-\nabla\cdot(\kappa\nabla v)=0&& \mbox{in }\omega_i\\
v&=v_k&& \mbox{on }\partial\omega_i.
\end{aligned}
\right.
\end{equation*}
\State Solve one local multiscale basis function,
\begin{equation*}
\left\{
\begin{aligned}
-\nabla\cdot(\kappa\nabla v^{i})&=\frac{\widetilde{\kappa}}{\int_{\omega_i}\widetilde{\kappa}\dx} \quad&&\text{ in } \omega_i\\
-\kappa\frac{\partial v^{i}}{\partial n}&=|\partial\omega_i|^{-1}\quad&&\text{ on }\partial \omega_i.
\end{aligned}
\right.
\end{equation*}
\State Build global multiscale space:
$
V_{\text{ms},\ell}  := \text{span} \{\chi_i\mathcal{L}^{-1}_i(v_k),\chi_i v^{i} : \,  \, 1 \leq i \leq N,\,\,\, 1 \leq k \leq  2^{\ell+2}\}.
$
\end{algorithmic}
\end{algorithm}
To analyze this algorithm, we define the weighted coefficient:
\begin{equation}\label{defn:tildeKappa}
\widetilde{\kappa} =H^2 \kappa \sum_{i=1}^{N}  | \nabla \chi_i |^2.
\end{equation}

Furthermore, let $\widetilde{\kappa}^{-1}$ be defined by
\begin{equation}\label{eq:inv-tildeKappa}
\widetilde{\kappa}^{-1}(x):=
\left\{
\begin{aligned}
&\widetilde{\kappa}^{-1}, \quad &&\text{ when } \widetilde{\kappa}(x)\ne 0,\\
&1, \quad &&\text{ otherwise }.
\end{aligned}
\right.
\end{equation}
Next, we recap a few results on the approximation properties of the multiscale ansatz space $V_{\text{ms},\ell}$ which play a key role in the later convergence analysis (Section \ref{subsec:xxxx}). The following assumption on the coefficient $\kappa$ implies that $\kappa$ is piecewise smooth.
\begin{assumption}[Structure of $D$ and $\kappa$]\label{ass:coeff}
		Let $D$ be a domain with a $C^{1,\rho}$ $(0<\rho<1)$ boundary $\partial D$,
		and $\{D_i\}_{i=1}^m\subset D$ be $m$ pairwise disjoint strictly convex open subsets, {each with a $C^{1,\rho}$ boundary
			$\Gamma_i:=\partial D_i$,} and denote $D_0=D\backslash \overline{\cup_{i=1}^{m} D_i}$. 
		Let the permeability coefficient $\kappa$ be piecewise regular function defined by
		\begin{equation}
		\kappa=\left\{
		\begin{aligned}
		&\eta_{i}(x) &\text{ in } D_{i},\\
		&1 &\text{ in }D_0.
		\end{aligned}
		\right.
		\end{equation}
		Here $\eta_i\in C^{\mu}(\bar{D_i})$ with $\mu\in (0,1)$ for $i=1,\cdots,m$. Denote $\etamaxmin{min}:=\min\limits_{i}\{\min\limits_{x\in D_i}\{\eta_i(x)\}\}\geq 1$ and $\etamaxmin{max}:=\max\limits_{i}\{\|\eta_i\|_{C_0(D_i)}\}$. Furthermore, we assume the
diameter of $D$ is 1.
	\end{assumption}
Assumption \ref{ass:coeff}, first proposed in \cite{li2019convergence} then further utilized in \cite{fu2018}, is needed to make the constant appearing in the approximation properties independent of the possible contrast in the coefficient $\kappa$.
	
	In the following, we define the $L^2(\partial\omega_i)$-orthogonal projection onto the local multiscale space up to level $\ell$ by $\mathcal{P}_{i,\ell}: L^2(\partial\omega_i)\to \mathcal{L}_{i}^{-1}(V^{i}_{\ell})$ satisfying
	\begin{align}\label{eq:projectionEDGE}
	\mathcal{P}_{i,\ell}(v):=\sum\limits_{j=1}^{2^{\ell+2}}(v,\psi_{j})_{\partial\omega_i}\mathcal{L}_{i}^{-1}(\psi_{j}) \quad \text{ for all }v\in L^2(\partial\omega_i).
	\end{align}
	Here, we denote $\psi_{j}$ for $j=1,\cdots, 2^{\ell+2}$ as the hierarchical bases defined on the four edges of $\omega_i$ of level $\ell$ and the local operator $\mathcal{L}_i$ is defined as in Algorithm \ref{algorithm:wavelet}.
	
Let $\mathcal{P}_{\ell}$ be the projection onto the multiscale space $V_{\text{ms},\ell}$ given by
	\begin{align}\label{eq:projectionglo}
	\mathcal{P}_{\ell}(v):=\sum_{i=1}^{N}\chi_i\mathcal{P}_{i,\ell}(v) \text{ for all }v\in V.
	\end{align}
Let $\mathcal{L}:=-\nabla\cdot(\kappa\nabla \cdot)$ be the elliptic operator defined on $V$, and let its discrete operator $\mathcal{L}_{\ell}: V_{\text{ms},\ell}\to H^{-1}(D)$ be
\begin{align*}
(\mathcal{L}_{\ell}w_{\ell},v_{\ell} ):=(\mathcal{L}w_{\ell},v_{\ell} )
=a(w_{\ell},v_{\ell} )\text{ for all }v_{\ell}\text{ and }w_{\ell} \in V_{\text{ms},\ell}.
\end{align*}
Then the inverse operator $\mathcal{L}_{\ell}^{-1}$ exists, which is self-adjoint, positive semi-definite on $L^2(D)$, and positive definite on $V_{\text{ms},\ell}$.
Further, let $\mathcal{R}_{\ell}$ be the Riesz operator associated to $\mathcal{L}_{\ell}$ in the multiscale space $V_{\text{ms},\ell}$, i.e.,
\[
\forall v\in V \text{ and }w_{\text{ms}}\in V_{\text{ms},\ell}:a(v-\mathcal{R}_{\ell}v,w_{\text{ms}})=0,
\]
then the following relation holds
\begin{align*}
\mathcal{L}_{\ell}^{-1}=\mathcal{R}_{\ell}\mathcal{L}^{-1}.
\end{align*}

The approximation property of $\mathcal{L}_{\ell}^{-1}$ in energy norm is derived in \cite[Proposition 5.2]{fu2018}. For any $v\in L^2(D)$, it holds
\begin{equation}\label{eq:glo-energy1}
\begin{aligned}
\normE{\mathcal{L}^{-1}v-\mathcal{L}_{\ell}^{-1}v}{D}&\lesssim\eta(H,\ell)
\|v\|_{L^2(D)}
\end{aligned}
\end{equation}
with
\begin{align}\label{eq:etaHl}
\eta(H, \ell):=H\normI{\widetilde{\kappa}}{D}^{1/2}+2^{-\ell/2}\|\kappa\|_{L^{\infty}(\mathcal{F}_H)}.
\end{align}
This, together with Assumption \ref{ass:coeff} and orthogonality of the Riesz operator $\mathcal{R}_{\ell}$, leads to
\begin{equation}\label{eq:glo-energy2}
\begin{aligned}
\normL{\mathcal{L}^{-1}v-\mathcal{L}_{\ell}^{-1}v}{D}&\lesssim \eta(H,\ell)
\normE{\mathcal{L}^{-1}v}{D}.
\end{aligned}
\end{equation}

\section{Model reduction in the history part}\label{sec:history}
Recall that the time fractional operator $\Dbe$ can be split into a summation of the local component and the history component
\begin{align}
\Dbe v(x,t_{n+1})&=\frac{1}{\Gamma(1-\alpha)}\int_0^{t_{n+1}} \frac{1}{(t_{n+1}-s)^\alpha}v'(x,s)\, \mathrm{ds}\nonumber\\
&=\frac{1}{\Gamma(1-\alpha)}\int_{t_n}^{t_{n+1}} \frac{1}{(t_{n+1}-s)^\alpha}v'(x,s)\, \mathrm{ds}
+\frac{1}{\Gamma(1-\alpha)}\int_0^{t_{n}} \frac{1}{(t_{n+1}-s)^\alpha}v'(x,s)\, \mathrm{ds}\nonumber\\
&=: \Dbel v(x,t_{n+1}) +\Dbeh v(x,t_{n+1}). \label{eqn:cont-decom}
\end{align}
We will approximate the local component $ \Dbel v(x,t_{n+1}) $ by the L1 scheme and the history component $ \Dbeh v(x,t_{n+1}) $ by a new exponential sums inspired by sinc approximation \cite{braess2005approximation,MR1226236} in Section \ref{subsec:expSum}. Based upon this, a new local time integrator is presented in Section \ref{subsec:locInt} that can handle long time simulation more efficiently.

\subsection{Approximation by exponential sums}\label{subsec:expSum}
In order to deal with long time situation, we will construct a new approximation by exponential sums inspired by
\cite{braess2005approximation}, which is essentially based upon Sinc quadrature to analytic, improperly integrable functions in $\mathbb{R}$, see \cite{MR1226236}. Let $F(s)$ be an analytic, improperly integrable functions defined on $\mathbb{R}$ and let $\tau>0$ be the step size and $N\in\mathbb{N}_{+}$, then we obtain by integration after Sinc interpolation to $F(s)$ the following quadrature formula
\begin{align*}
\int_{-\infty}^{\infty}F(s)\mathrm{d}s&\approx \tau\sum\limits_{j=-\infty}^{\infty}F(j\tau)\\
&\approx \tau\sum\limits_{j=-N}^{N}F(j\tau).
\end{align*}
Here, the second estimate is a direct truncation of the first estimate. The quadrature error can be estimated by the following theorem \cite[Theorem 3.2.1]{MR1226236}.
\begin{lemma}\label{lem:sincQuadrature}
Let $F(s)$ be holomorphic in the strip $\mathcal{D}_{\delta}:=\{z\in\mathbb{C}:\Im z\in (-\delta,\delta)\}$ with $\delta>0$ such that
\begin{align*}
\int_{\partial\mathcal{D}_{\delta}}|F(s)||ds|<\infty, \text{ and for some } c, \beta>0: |F(s)|<c\mathrm{e}^{-\beta |s|} \text{ for all } s\in \mathbb{R}.
\end{align*}
Then taking $\tau:=\Big(\frac{2\pi\delta}{\beta N}\Big)^{1/2}$, we obtain
\begin{align*}
\left|\int_{-\infty}^{\infty}F(s)\mathrm{d}s-\tau\sum\limits_{j=-N}^{N}F(j\tau)\right|\lesssim \mathrm{e}^{-\sqrt{2\pi\delta\beta N}}.
\end{align*}
Here, the hidden constant depends only on $F(s)$, $\delta$ and $\beta$.
\end{lemma}
We will utilize Lemma \ref{lem:sincQuadrature} to derive a new exponential sums approximation to $\frac{1}{t^{1+\alpha}}$. To this end, we first rewrite it as an integral over $\mathbb{R}$ of a analytic, improperly integrable function, which satisfies the conditions listed in Lemma \ref{lem:sincQuadrature}.
\begin{theorem}[Approximation by exponential sums for long time]
\label{lem:expsum2}
For all $\alpha>0$ and $N\in \mathbb{N}_{+}$, there holds
\begin{align}\label{eq:trunc-long}
\left|\frac{1}{t^{1+\alpha}}- \tau\sum\limits_{\ell=-N}^{N}F(\ell \tau;t)\right|\leq
 \left(
 \frac{(t-\alpha/\mathrm{e})^{-1}+1}{1-\mathrm{e}^{-\pi\sqrt{\beta N}}}+\frac{2}{\beta}
 \right)
 \mathrm{e}^{-\pi\sqrt{\beta N}}\text{ for all } t>\alpha.
\end{align}
Given $t>0$, $F(s;t): \mathbb{R}\to \mathbb{R}^{+}$ as a function of $s\in \mathbb{R}$ is
\begin{align}\label{eq:weight1}
F(s;t)&:=\mathrm{e}^{-t\ln(1+\mathrm{e}^s)}(\ln(1+\mathrm{e}^s))^{\alpha}\frac{1}{1+\mathrm{e}^{-s}}.
\end{align}
The step size $\tau$ and the positive parameter $\beta$ are 
\begin{equation}\label{eq:weight2}
\left\{
\begin{aligned}
 \tau&:=\frac{\pi}{\sqrt{\beta N}}\\
\beta&:=\min\{1,t-\alpha\}.
\end{aligned}
\right.
\end{equation}
\end{theorem}
\begin{proof}
For any $t>0$, by the definition of Gamma function, we obtain
\begin{align*}
t^{-(1+\alpha)}&
=\frac{1}{\Gamma(\alpha+1)}\int_{0}^{\infty}\mathrm{e}^{-tx}x^{\alpha}\dx.
\end{align*}
Then taking integration by parts and a change of variables yields,
\begin{align*}
t^{-(1+\alpha)}
&=\frac{1}{\Gamma(\alpha+1)}\int_{-\infty}^{+\infty}\mathrm{e}^{-t\ln(1+\mathrm{e}^s)}
(\ln(1+\mathrm{e}^s))^{\alpha}\frac{1}{1+\mathrm{e}^{-s}}\mathrm{d}s\\
&=\frac{1}{\Gamma(\alpha+1)}\int_{-\infty}^{+\infty}F(s;t)\mathrm{d}s.
\end{align*}
Here, $F(s;t):\mathbb{R}\to \mathbb{R}^{+}$ as a function of $s$ parameterized by $t\in (0,\infty)$ is defined in \eqref{eq:weight1}. In the following, we will verify that $F(s;t)$ satisfies the conditions listed in Lemma \ref{lem:sincQuadrature} when $t>\alpha$.

On the one hand, it holds that
\begin{align}\label{eq:111}
|F(s;t)|\lesssim \mathrm{e}^{-(t-\alpha)|\Re s|}.
\end{align}
Indeed, Definition \eqref{eq:weight1} implies
\[
|F(s;t)|=\left|\mathrm{e}^{-t\ln(1+\mathrm{e}^s)}\left(\ln(1+\mathrm{e}^s)\right)^{\alpha}\frac{1}{1+\mathrm{e}^{-s}}\right|\lesssim
\mathrm{e}^{-t\Re s}(\Re s)^{\alpha} \text{ for all }\Re s>0.
\]
This, together with the fact that
\[
\ln x\leq \frac{x}{\mathrm{e}} \text{ for all } x>0,
\]
leads to
\[
|F(s;t)|\leq\mathrm{e}^{-(t-\alpha)\Re s}\text{ for all }\Re s>0.
\]
One can observe that
\begin{align*}
|F(s;t)|\lesssim \mathrm{e}^{-|\Re s|} \text{ for all }\Re s<0.
\end{align*}
Consequently, inequality \eqref{eq:111} has been proved.

Note that $F(s;t)$ is holomorphic in the strip $\mathcal{D}:=\{z\in\mathbb{C}:\Im z\in (-\pi/2,\pi/2)\}$, and one can prove from inequality \eqref{eq:111},
\begin{align*}
\int_{\partial\mathcal{D}}|F(s;t)||ds|\lesssim \frac{1}{t-\alpha}+1.
\end{align*}
Consequently, all the requirement for Lemma \ref{lem:sincQuadrature} is fulfilled and the assertion follows from Lemma \ref{lem:sincQuadrature}.
\end{proof}
Note that we attempt to use sum of exponentials to approximate the history component $ \Dbeh v(x,t_{n+1})$, and thus we demand an estimate \eqref{eq:trunc-long} for $t>\tau_f$ instead. This, nevertheless, can be derived by a scaling argument.
\begin{corollary}\label{corollary:soe}
For any $\epsilon>0$ and $\tau_f\in (0,1)$, there exist $(\omega_{j},\lambda_{j})_{j=1}^{N_{\text{exp}}}$, s.t.,
\begin{align}\label{eq:trunc-longxx}
\Big|\frac{1}{t^{1+\alpha}}- \sum_{j=1}^{N_{\text{exp}}}\omega_{j}\mathrm{e}^{-\lambda_{j}t}\Big|\lesssim \epsilon\text{ for all } t\in(\tau_f,T),
\end{align}
with
\[
N_{\text{exp}}\approx |\log({\tau_f^{1+\alpha}}{\epsilon})|^2.
\]
\end{corollary}
\begin{proof}
Let $t>\tau_f$, then the new variable $s:=({\alpha}+1)\frac{t}{\tau_f}>{\alpha}+1$. In view of \eqref{eq:trunc-long}, we obtain for $N\in \mathbb{N}$,
\begin{align*}
\Big|\frac{1}{s^{1+\alpha}}- \tau\sum\limits_{\ell=-N}^{N}F(\ell \tau;s)\Big|\leq
 \Big(\frac{2}{1-\mathrm{e}^{-\pi\sqrt{ N}}}+2\Big)\mathrm{e}^{-\pi\sqrt{N}}
\end{align*}
with $\tau=\frac{\pi}{\sqrt{N}}$.

Next, making a change of variable by $t=\frac{s\tau_f}{{\alpha}+1}$ and using \eqref{eq:weight1}, lead to
\begin{align*}
\Big|\frac{1}{t^{1+\alpha}}- \sum_{j=1}^{N_{\text{exp}}}\omega_{j}\mathrm{e}^{-\lambda_{j}t}\Big|\lesssim \tau_f^{-1-\alpha}\Big(\frac{2}{1-\mathrm{e}^{-\pi\sqrt{ \frac{N_{\text{exp}}-1}{2}}}}+2\Big)\mathrm{e}^{-\pi\sqrt{\frac{N_{\text{exp}}-1}{2}}}\text{ for all } t\in(\tau_f,\infty).
\end{align*}
Here,
\begin{align*}
N_{\exp}&=2N+1, \\
\omega_j&:=\tau_f^{-1-\alpha}({\alpha}+1)^{1+\alpha}
\frac{\pi}{\sqrt{N}}\Big(\ln(1+\mathrm{e}^{(j+N+1)\tau})\Big)^{\alpha}\frac{1}{1+\mathrm{e}^{-(j+N+1)\tau}},\\
\lambda_j&:=\tau_f^{-1}({\alpha}+1)\ln(1+\mathrm{e}^{(j+N+1)\tau}).
\end{align*}
Consequently, to guarantee
\begin{align*}
\Big|\frac{1}{t^{1+\alpha}}- \sum_{j=1}^{N_{\text{exp}}}\omega_{j}\mathrm{e}^{-\lambda_{j}t}\Big|\lesssim\epsilon \text{ for all } t\in(\tau_f,\infty),
\end{align*}
one only needs to choose $N_{\text{exp}}$ sufficiently large, s.t.,
\begin{align*}
\tau_f^{-1-\alpha}\Big(\frac{2}{1-\mathrm{e}^{-\pi\sqrt{ \frac{N_{\text{exp}}-1}{2}}}}+2\Big)\mathrm{e}^{-\pi\sqrt{\frac{N_{\text{exp}}-1}{2}}}\lesssim \epsilon,
\end{align*}
i.e.,
\begin{align*}
N_{\text{exp}}\approx |\log({\tau_f^{1+\alpha}}{\epsilon})|^2.
\end{align*}
This completes the proof.
\end{proof}
In another word, a number of $\mathcal{O}(|\log({\tau_f^{1+\alpha}}{\epsilon})|^2)$ terms is required to have an approximation error below a tolerance $\epsilon>0$. Note that this result is independent of the final time. Further, one can observe that $\lambda_j=\mathcal{O}(\tau_f^{-1})$ and $\lambda_j>0$ for all $j=1,\cdots, N_{\text{exp}}$. Denote
\begin{align}\label{not:gamma}
\gamma:=\min_{j=1,\cdots,N_{\text{exp}}}\{\lambda_j\},
\end{align}
then $\gamma\gg 1$ and $\gamma=\mathcal{O}(\tau_f^{-1})$.
\begin{remark}[ Alternative approximation by exponential sums]
Another approximation by exponential sums are provided in \cite[Theorem 2.5]{jiang2017fast},
For any $\epsilon>0$, there exist $(\omega_{1,j},\lambda_{1,j})_{j=1}^{N_{1,\text{exp}}}$, s.t.,
\begin{align}\label{eq:trunc-short}
\Big|\frac{1}{t^{1+\alpha}}- \sum_{j=1}^{N_{1,\text{exp}}}\omega_{1,j}\mathrm{e}^{-\lambda_{1,j}t}\Big|\lesssim \epsilon\text{ for all } t\in(\tau_f,T),
\end{align}
with
\[
N_{1,\text{exp}}\approx |\ln\epsilon|\Bigg(\ln|\ln\epsilon|+\ln \frac{T}{\tau_f}\Bigg)+|\ln\tau_f|\Big(\ln|\ln\epsilon|+|\ln\tau_f|\Big).
\]
Note that this approximation is only suitable for cases with a final time of reasonable size due to the term $\ln \frac{T}{\tau_f}$ in the total number of terms $N_{1,\text{exp}}$.
\end{remark}
\subsection{Local time integrators}\label{subsec:locInt}
We construct in this section local time integrators based upon the approximation by exponential sums developed in Corollary \ref{corollary:soe}.
The accuracy requirement of this exponential sums approximation is derived in \eqref{assum:epsilon}, which implies that $\epsilon=\mathcal{O}(T^{-1-\alpha})$. The fine scale time step $\tau_f$ can be very tiny due to the possible singularity of the solution $u(\cdot,t)$ at the original $t=0$ due to, for example, insufficient  regularity in the source term $f$ or the initial data $u_0$. For this reason, we assume
\begin{align}\label{ass:comp}
\tau_f\ll \epsilon
\end{align}
for given tolerance $\epsilon>0$ throughout this paper.


Applying the summation of exponentials (Corollary \ref{corollary:soe}) to the history component $\Dbeh$ after taking integration by parts, we obtain
\begin{align}
&\bar{\partial}_{\tau_f}^{\alpha} v(x,t_{n+1})\nonumber\\
&\approx\frac{v(x,t_{n+1})-v(x,t_n)}{\tau_f^{\alpha}c_{\alpha}}
+\frac{1}{\Gamma(1-\alpha)}\Bigg( \frac{v(x,t_n)}{\tau_f^{\alpha}}-\frac{v(x,t_0)}{ t_{n+1}^{\alpha}}
-\alpha\sum_{j=1}^{\nexp}\omega_{j}\int_{0}^{t_n}\mathrm{e}^{-\lambda_{j}(t_{n+1}-s)}v(x,s)\mathrm{d}s\Bigg)
\nonumber\\
&=\underbrace{\frac{v(x,t_{n+1})-v(x,t_n)}{\tau_f^{\alpha}c_{\alpha}}}_{=:\Dbelaxx v(x,t_{n+1})}
+\underbrace{\frac{1}{\Gamma(1-\alpha)}\Big( \frac{v(x,t_n)}{\tau_f^{\alpha}}-\frac{v(x,t_0)}{ t_{n+1}^{\alpha}}
-\alpha\sum_{j=1}^{N_{\text{exp}}}\omega_{j}{\psi}_{j}[n]\Big)
}_{=:\Dbehaxx v(x,t_{n+1})}
.\label{eq:form1}
\end{align}
Here, the parameters ${\psi}_{j}[n]$ for $j=1,\cdots,N_{\text{exp}}$ are defined by
\begin{align*}
{\psi}_{j}[n]:=\int_{0}^{t_n}\mathrm{e}^{-\lambda_{j}(t_{n+1}-s)}v(x,s)\mathrm{d}s.
\end{align*}
which can be written as
\begin{align*}
{\psi}_{j}[n]
=\int_{0}^{t_{n-1}}\mathrm{e}^{-\lambda_{j}(t_{n+1}-s)}v(x,s)\mathrm{d}s
+\int_{t_{n-1}}^{t_n}\mathrm{e}^{-\lambda_{j}(t_{n+1}-s)}v(x,s)\mathrm{d}s.
\end{align*}
This implies the following recurrence relation
\begin{align*}
\psi_{j}[n]&=\mathrm{e}^{-\lambda_{j}(t_{n+1}-t_n)}\psi_{j}[n-1]+\int_{t_{n-1}}^{t_n}
\mathrm{e}^{-\lambda_{j}(t_{n+1}-s)}v(x,s)\mathrm{d}s\text{ for }n\geq 1\\
\psi_{j}[0]&=0.
\end{align*}
Note that the integrals in the recurrence relation can be estimated by the L1 scheme, which yields
\begin{align}\label{eq:monsxxxxxx}
\int_{t_{n-1}}^{t_n}\mathrm{e}^{-\lambda_{j}(t_{n+1}-s)}v(x,s)\mathrm{d}s
&\approx\int_{t_{n-1}}^{t_n}\mathrm{e}^{-\lambda_{j}(t_{n+1}-s)}
\Big(\frac{t_{n}-s}{\tau_f}v(x,t_{n-1})+\frac{s -t_{n-1}}{\tau_f}v(x,t_{n})\Big)\mathrm{d}s\\
&=: c_{1,j}^{\tau_f}v(x,t_{n-1})+c_{2,j}^{\tau_f}v(x,t_{n}),\nonumber
\end{align}
with
\begin{align}\label{eq:xxxxx}
c_{1,j}^{\tau_f}:=\frac{\mathrm{e}^{-\lambda_{j}\tau_f}}{\lambda_{j}^2\tau_f}\Big( 1-\mathrm{e}^{-\lambda_{j}\tau_f}
-\lambda_{j}\tau_f\mathrm{e}^{-\lambda_{j}\tau_f}\Big)\text{ and }
c_{2,j}^{\tau_f}:=\frac{\mathrm{e}^{-\lambda_{j}\tau_f}}{\lambda_{j}^2\tau_f}\Big( -1+\mathrm{e}^{-\lambda_{j}\tau_f}
+\lambda_{j}\tau_f\Big).
\end{align}
Note also that the parameters $c_{i,j}^{\tau_f}$ for $i=1,2$ and $j=1,2,\cdots,N_{\text{exp}}$ rely only on $\{\lambda_j\}_{j=1}^{N_{\text{exp}}}$ and the parameter $\tau_f$. Analogously, we can define $c_{i,j}^{\tau_c}$ with a coarse step size $\tau_c$.

We establish an estimate of those parameters in the next lemma that provides a property of the history information.
\begin{lemma}\label{lem:222xxx}
Let $c_{i,j}^{\tau}$ for $i=1,2$ and $j=1,2,\cdots,N_{\text{exp}}$ be defined in \eqref{eq:xxxxx} with $\tau=\tau_f$ or $\tau=\tau_c$. Then it holds
\begin{align*}
\Bigg|\sum_{j=1}^{\nexp}\omega_{j}c_{i,j}^{\tau}\Bigg|\leq \mathrm{e}^{-\gamma \tau}\Big(\tau^{-\alpha}(1-\alpha)^{-1}+\epsilon\tau/2\Big).
\end{align*}
\end{lemma}
\begin{proof}
A combination of \eqref{eq:monsxxxxxx} and \eqref{eq:trunc-short} lead to
\begin{align*}
\Bigg|\sum_{j=1}^{\nexp}\omega_{j}c_{1,j}^{\tau}\Bigg|&=
\Bigg|\sum_{j=1}^{\nexp}\omega_{j}\int_{t_{n-1}}^{t_n}\mathrm{e}^{-\lambda_{j}(t_{n+1}-s)}
\frac{t_{n}-s}{\tau_f}\mathrm{d}s\Bigg|\\
&\leq \mathrm{e}^{-\gamma \tau}
\Bigg|\sum_{j=1}^{\nexp}\omega_{j}\int_{0}^{\tau}\mathrm{e}^{-\lambda_{j}(\tau-s)}
\frac{\tau-s}{\tau}\mathrm{d}s\Bigg|\\
&\leq\mathrm{e}^{-\gamma \tau}\left(
\int_{0}^{\tau}(\tau-s)^{-1-\alpha}
\frac{\tau-s}{\tau}\mathrm{d}s+\epsilon\int_{0}^{\tau}
\frac{\tau-s}{\tau}\mathrm{d}s\right)\\
&=\mathrm{e}^{-\gamma \tau}(\tau^{-\alpha}(1-\alpha)^{-1}+\epsilon\tau/2).
\end{align*}
Recall that $\gamma$ is defined in \eqref{not:gamma}. This has proved the case for $i=1$. The other case can be shown similarly.
\end{proof}
\begin{remark}[Efficiency of the new time stepping scheme \eqref{eq:form1}]
$\Dap$ can be defined similarly by replacing the fine time steps with the coarse time steps.
One can observe from formulas \eqref{eq:form1} that each calculation of $\Dap v(x,T^{n})$ involves $N_{\text{exp}}$ calculation and  $N_{\text{exp}}$ storage for $n=2,\cdots, M_c$. Note from \eqref{eq:trunc-long} and \eqref{ass:comp} that $N_{\text{exp}}\ll M_{f}$ and $N_{\text{exp}}$ is independent of the final time $T$ as $T\to \infty$. Comparing with the L1 scheme \eqref{eq:l1scheme}, the new time stepping scheme is much more efficient.
\end{remark}
For the sake of conciseness, we will rewrite the time integrator as
\begin{equation}\label{eq:form}
\begin{aligned}
\bar{\partial}_{\tau_f}^{\alpha} v(x,t_{n+1})
&=:\frac{v(x,t_{n+1})-v(x,t_n)}{\tau_f^{\alpha}c_{\alpha}}
+\frac{1}{\Gamma(1-\alpha)}\Big( \frac{v(x,t_n)}{\tau_f^{\alpha}}-\frac{v(x,0)}{ t_{n+1}^{\alpha}}
-\alpha\sum_{j=1}^{N_{\text{exp}}}\omega_{j}{\psi}_{j}[n]\Big)\\
\psi_j[n+1]&=\mathcal{H}(\psi_j[n]; \tau_f, v(x,t_n), v(x,t_{n+1}))\text{ for }n\geq 0 \text{ and }
\psi_j[0]=0 \text{ for all } j=1,\cdots,N_{\text{exp}}.
\end{aligned}
\end{equation}
Here, the propagation operator of the history information is
\begin{align}\label{eqn:history-propag}
\mathcal{H}(\psi_j[n]; \tau_f, v(x,t_n), v(x,t_{n+1})):=\mathrm{e}^{-\lambda_j\tau_f}\psi_j[n]+{c}_{1,j}^{\tau_f}v(x,t_{n})+{c}_{2,j}^{\tau_f}v(x,t_{n+1}).
\end{align}
\subsection{Multiscale-SOE-scheme}
In this section, we formulate the numerical solution based on the time integrator \eqref{eq:form} in the  multiscale ansatz space $V_{\text{ms},\ell}$. The multiscale-SOE-scheme aims to seeking $U_{\text{ms},\ell}^{n}\in V_{\text{ms},\ell}$ for $n=1,\cdots,M_{f}$, satisfying
\begin{equation}\label{eq:multiscale-his2}
	\left\{
	\begin{aligned} (U_{\text{ms},\ell}^{n},v_{\text{ms}})_D
&+c_{\alpha}\tau_f^{\alpha}a(U_{\text{ms},\ell}^{n},v_{\text{ms}})=
\alpha(U_{\text{ms},\ell}^{n-1},v_{\text{ms}})_D+\alpha(1-\alpha)\tau_f^{\alpha}\sum_{j=1}^{N_{\text{exp}}}\omega_{j}({\psi}_{j}[n-1],v_{\text{ms}})_D\\
&+\frac{1-\alpha}{(n+1)^{\alpha}}(U_{\text{ms},\ell}^{0},v_{\text{ms}})_D+c_{\alpha}\tau_f^{\alpha}(\mathcal{P}_{\ell}f(\cdot,t_n),v_{\text{ms}})_{D}
	\quad \text{for all} \, \, v_{\text{ms}}\in V_{\text{ms},\ell}\\
\psi_j[n]&=\mathcal{H}(\psi_j[n-1]; \tau_f, U_{\text{ms},\ell}^{n-1}, U_{\text{ms},\ell}^{n})\text{ for }n\geq 0 \text{ and }
\psi_j[0]=0 \text{ for all } j=1,\cdots,N_{\text{exp}}\\
	U_{\text{ms},\ell}^0&=\mathcal{P}_{\ell} u_0.
	\end{aligned}
	\right.
	\end{equation}
For the sake of the convergence analysis, we require a semidiscrete Galerkin scheme in the multiscale ansatz space $V_{\text{ms},\ell}$, defined by seeking $u_{\text{ms},\ell}(t)\in  V_{\text{ms},\ell}$, satisfying
\begin{equation}\label{eq:multiscale-hissemi}
	\left\{
	\begin{aligned}
({\partial}_{t}^{\alpha}u_{\text{ms},\ell}, v_{\text{ms}})_D+a(u_{\text{ms},\ell},v_{\text{ms}})
&=(f,v_{\text{ms}})_D\quad \text{for all} \, \, v_{\text{ms}}\in V_{\text{ms},\ell}\\
	u_{\text{ms}}(x,0)&=\mathcal{P}_{\ell} u_0.
	\end{aligned}
	\right.
	\end{equation}
\section{Wavelet-based Edge Multiscale Parareal Algorithm}\label{sec:WEMP}
We construct in this section the Wavelet-based Edge Multiscale Parareal (WEMP) Algorithm. To this end,
we first recap a few terminologies commonly appeared in parareal algorithm. The one step coarse solver on the coarse subinterval $(T^n,T^{n+1})$ is
\begin{align}
\begin{bmatrix} U^{n+1}_{\text{ms},\ell}	\\ \Phi^{n+1}_{\text{ms},\ell}\end{bmatrix}
&=\mathcal{G}^{{\text{ms},\ell}}(T^n, U^{n}_{\text{ms},\ell};\Phi^{n}_{\text{ms},\ell})
 \label{eq:coarseSolver}
\end{align}
which yields $U^{n+1}_{\text{ms},\ell}\in V_{\text{ms},\ell}$ as a coarse approximation to $u(\cdot,T^{n+1})$, provided with an approximation $U^{n}_{\text{ms},\ell}$ of $u(\cdot,T^n)$. Here, $\Phi^{n}_{\text{ms},\ell}$ serves as the input parameter to calculate $U^{n+1}_{\text{ms},\ell}$, which represents the memory effect from the previous time steps.

In matrix form, it reads
\begin{equation*}
\left\{
\begin{aligned}
U^{n+1}_{\text{ms},\ell}
&=\Big(\frac{\tau_c^{-\alpha}}{\Gamma(2-\alpha)}M_{\text{ms},\ell}+ A_{\text{ms},\ell}\Big)
^{-1}M_{\text{ms},\ell}
\Bigg(\tau_c^{-\alpha}\frac{\alpha}{\Gamma(2-\alpha)}U^{n}_{\text{ms},\ell}+\\
&\frac{\alpha}{\Gamma(1-\alpha)}\Theta^{T}\Phi^{n}_{\text{ms},\ell}+\frac{1}{\Gamma(1-\alpha)T^{n+1}} U^0_{\text{ms},\ell}+F^{n+1}_h
\Bigg) \\
\Phi^{n+1}_{\text{ms},\ell}&=D^{\tau_c}\Phi^{n}_{\text{ms},\ell}+\mathbf{C_1^{\tau_c}}U^{n}_{\text{ms},\ell}
+\mathbf{C_2^{\tau_c}}U^{n+1}_{\text{ms},\ell}.
\end{aligned}
\right.
\end{equation*}
Here, we use the following notations,
\begin{equation*}
\left\{
\begin{aligned}
\Phi^n_{\text{ms},\ell}&:=[\psi_1[n],\psi_2[n],\cdots,\psi_{\nexp}[n]]&&\\
\Theta&=\begin{bmatrix} \omega_1,\omega_2,\cdots,\omega_{N_{\text{exp}}}\end{bmatrix}^{T}&&\\
\mathbf{C_{i}^{\tau}}&=\begin{bmatrix} c_{i,1}^{\tau},c_{i,2}^{\tau},\cdots,c_{i,N_{\text{exp}}}^{\tau}\end{bmatrix}^{T}&&\text{ for }i=1,2\\
D^{\tau}&= \text{diag}\Big(\mathrm{e}^{-\lambda_1\tau}, \mathrm{e}^{-\lambda_2\tau},\cdots, \mathrm{e}^{-\lambda_{N_{\text{exp}}}\tau}\Big)&&\\
M_{\text{ms},\ell}&=\mathrm{\Psi}^{T}_{\text{ms},\ell}M_h\mathrm{\Psi}_{\text{ms},\ell}\text{ and }A_{\text{ms},\ell}=\mathrm{\Psi}^{T}_{\text{ms},\ell}A_h\mathrm{\Psi}_{\text{ms},\ell},&&
\end{aligned}
\right.
\end{equation*}
where $M_h$ and $A_h$ are the mass matrix and stiffness matrix corresponding to the discretization of the elliptic operator $-\nabla\cdot(\kappa\nabla \cdot)$ in the finite element space $V_h:=\text{span}\{\phi_1,\cdots,\phi_{\text{dof}_h}\}$, $\text{dof}_h$ denotes the dimension of $V_h$. $(F^{n+1}_h)_i:=\int_{D}f(\cdot,t_{n+1})\phi_i\dx$ for all $i=1,\cdots,\text{dof}_h$, $\mathrm{\Psi}_{\text{ms},\ell}$ denotes a matrix with columns composed of the coefficients of multiscale basis functions in $V_{\text{ms},\ell}$ in the finite element space $V_h$. Let $\text{dof}_{\ell}$ be the dimension of the multiscale space $V_{\text{ms},\ell}$. Note that $\Phi^{n}_{\text{ms},\ell}$ is vector-valued function with $N_{\text{exp}}$ components, and the length of each component is the dimension of the multiscale ansatz space $V_{\text{ms},\ell}$.

The fine solver on the coarse subinterval $(T^n,T^{n+1})$ is
\begin{align}
\begin{bmatrix} u_{\text{ms},\ell}^{n+1}\\\Phi_{\text{ms},\ell}^{n+1}\end{bmatrix}
=\mathcal{F}^{{\text{ms},\ell}}_f(T^n,U^{n}_{\text{ms},\ell};\Phi^{n}_{\text{ms},\ell}), \label{eq:fineSolver}
\end{align}
which is defined in analogous to the one step coarse solver \eqref{eq:coarseSolver}, but with a uniform discrete time step $\tau_f$.

Analogously, we can define the semidiscrete Galerkin approximation in $V_{\text{ms},\ell}$
	\begin{align}
\begin{bmatrix} v_{\text{ms},\ell}\\\Phi_1\end{bmatrix}
&=\mathcal{F}^{{\text{ms},\ell}}(T^n,U^{n}_{\text{ms},\ell};\Phi^{n}_{\text{ms},\ell}). \label{eq:exactSolver}
	\end{align}
Time stepping schemes with $\bar{\cdot}$ on top refers to the same scheme with a vanishing source term, for example, $\bar{\mathcal{G}}^{{\text{ms},\ell}}(T^n, U^n_{\text{ms},\ell};\Phi_0)$ denotes the one step coarse solver with $f=0$. $\bar{\mathcal{F}}^{{\text{ms},\ell}}(T^n,U^{n}_{\text{ms},\ell};\Phi^{n}_{\text{ms},\ell})$ and $\bar{\mathcal{F}}^{{\text{ms},\ell}}_{f}(T^n,U^{n}_{\text{ms},\ell};\Phi^{n}_{\text{ms},\ell})$ are defined analogously.
	
In the parareal algorithm \ref{algorithm:parareal}, the cheap coarse time stepping scheme $\mathcal{G}^{\text{ms},\ell}$ is sequentially utilized over the global time interval $I$ to provide a rough approximation to $u(\cdot,T^{n+1})$, while the expensive fine time stepping scheme $\mathcal{F}_{f}^{\text{ms},\ell}$ is applied in each subinterval $[T^n,T^{n+1}]$ for $n=0,1,\cdots, M_{c}-1$ in paralell to provide a more accurate solution. They do not agree with each other, and the discrepancy between them is denoted as the correction operator,
\begin{align}\label{eq:correct}
\mathcal{S}(T^{n},U^{n}_{\text{ms},\ell};\Phi^{n}_{\text{ms},\ell}):=\mathcal{F}^{{\text{ms},\ell}}_f(T^{n},
U^{n}_{\text{ms},\ell};\Phi^{n}_{\text{ms},\ell})_1-\mathcal{G}^{{\text{ms},\ell}}(T^n, U^{n}_{\text{ms},\ell};\Phi^{n}_{\text{ms},\ell})_1.
\end{align}
Here, ${\cdot}_1$ denotes taking the first component.

In comparison with recently proposed parareal algorithm for time-fractional diffusion problems \cite{WU2018135}, our new proposed algorithm corrects only the solution $u(\cdot,T^{n+1})$. Its associated history component $\Phi^{n+1}_{\text{ms},\ell}$ will be updated by \eqref{eqn:history-propag} using the corrected solution $u(\cdot,T^{n+1})$. Extensive numerical tests shows that our proposed algorithm converges to the fine-scale solution within three iterations.

Next, we are ready to present our main algorithm, i.e., Algorithm \ref{algorithm:wavelet+parareal}. Its main objective is to obtain a good approximation to the exact solution $u$ of Problem \eqref{eqn:pde} at the discrete coarse time step $T^n$ for $n=1,2,\cdots,M_c$, albeit that the time stepping scheme is built upon a coarse descretization. This relies on a parallel correction operator defined by \eqref{eq:correct}.

We can solve \eqref{eq:multiscale-his2} using the coarse-scale time stepping scheme as the initial guess $\begin{bmatrix}U^{n}_0\\\Phi^{n}_0\end{bmatrix}$ for $n=0,\cdots, M_{c}$. Given the iteration parameter $k$, we apply the local coarse solver \eqref{eq:coarseSolver} in Step 4 to obtain $U_k^{n+1}$. Using the coarse solution $U_k^{n}$ as the initial condition, the fine solver \eqref{eq:fineSolver} subsequently is used to calculate the fine solution $u_k^{n+1}$ in paralell on each local time subinterval $[T^n, T^{n+1}]$. Then we calculate the discrepancy between the coarse solution and the fine solution in Step 6 on each discrete coarse time point $T^n$ for $n=1,2,\cdots,M_{c}$, and denote it as $\mathcal{S}(T^{n-1},U^{n-1}_{k};\Phi^{n-1}_{\text{ms},\ell})$. Subsequently, this jump term is utilized in Step 7 to correct the coarse solution calculated by the global coarse solver \eqref{eq:coarseSolver}. This process will be performed iteratively until certain tolerance on the jump terms is satisfied.	
\section{Error analysis}\label{sec:convergence}
This section is concerned with the convergence analysis for the multiscale-SOE-scheme \eqref{eq:multiscale-his2} and the WEMP Algorithm \ref{algorithm:wavelet+parareal}.
\subsection{Error analysis for the multiscale-SOE-scheme \eqref{eq:multiscale-his2}}\label{subsec:xxxx}
To derive the error estimate for Problem \eqref{eq:multiscale-his2}, we start with a regularity property of the semidiscrete Galerkin approximation solution $u_{\text{ms},\ell}$ in Lemma \ref{assump} that is standard in the numerical analysis, see, e.g., \cite{MR3601002}. We then show in Lemma \ref{eq:approx-numerical-frac} that the local time stepping scheme based on summation of exponentials $\bar{\partial}_{t}^{\alpha}$ is a small perturbation of the L1 scheme $\tilde{\partial}_{t}^{\alpha}$.
Based upon this, we establish the stability estimate of Problem \eqref{eq:multiscale-his2}. Finally, the error estimate is presented in Theorem \ref{thm:conv11}.

	\begin{algorithm}[H]
		\caption{Wavelet-based Edge Multiscale Parareal (WEMP) Algorithm}
		\label{algorithm:wavelet+parareal}
	\KwData{Multiscale space $V_{\text{ms},\ell} $; initial data $u_0$; source term $f$; tolerance $\delta$ and initial solution $\begin{bmatrix}U^{n}_0\\\Phi^{n}_0\end{bmatrix}$ for $n=0,\cdots, M_{c}$.}
\KwResult{$U^n_k$ for $n=0,\cdots, M_{c}$.}
		\begin{algorithmic}[1]
			\State $k=1$.
			\Do
\State Set the initial value: $\begin{bmatrix}U^{0}_k\\\Phi^{0}_k\end{bmatrix}=\begin{bmatrix}\mathcal{P}_{\ell}u_{0}\\0\end{bmatrix}$.
			\State Parallel compute $U^{n}_k$ on each local time subinterval [$T^{n-1}$,$T^{n}$] for $n=1,\cdots, M_{c}$,
			\begin{align*}
			U^{n}_k&=\mathcal{G}^{\text{ms},\ell}(T^n, U^{n-1}_{k-1};\Phi^{n-1}_{k-1})_1.
			\end{align*}
			\State Parallel compute $ u_k^{n}$ on each local time subinterval [$T^{n-1}$,$T^{n}$] for $n=1,\cdots, M_{c}$,
	\begin{align*}
	u_k^{n}&=\mathcal{F}^{\text{ms},\ell}_f(T^{n-1},U_{k-1}^{n-1};\Phi_{k-1}^{n-1})_1.
\end{align*}
			\State Parallel compute the solution jumps for $n=1,\cdots, M_{c}$,
			\[
			\mathcal{S}(T^{n-1},U^{n-1}_{k-1};\Phi^{n-1}_{k-1}):= u_k^{n}- U_k^{n}.
			\]
			\State Sequentially compute the corrected coarse solutions $\begin{bmatrix}U^{n}_{k}\\\Phi^{n}_{k}\end{bmatrix}$
			for $n=1,\cdots, M_{c}$,
			\begin{align*}
			U^{n}_{k}&=\mathcal{S}(T^{n-1}, U_{k-1}^{n-1};\Phi_{k-1}^{n-1})
+\mathcal{G}^{\text{ms},\ell}(T^{n-1},  U_{k}^{n-1};\Phi_{k}^{n-1})_1\\
\Phi_{k}^{n}
&=\mathcal{H}(\Phi^{n-1}_{k};\tau_c, U^{n-1}_{k}, U^{n}_{k}).
			\end{align*}
			\State Estimate the error,
			\[{\rm err}:=1/{M_c}\sum\limits_{n=1}^{M_{c}}\|U_{k}^{n}-U_{k-1}^{n}\|_{\ell_2}.\]
			\State $k\gets k+1$.
\doWhile{err$>\delta$}
\end{algorithmic}
\end{algorithm}
We begin with establishing a certain regularity result of the semidiscrete Galerkin approximation solution $u_{\text{ms},\ell}$ with respect to the temporal variable $t$, which is similar to \cite[Assumption 3.7]{MR3601002}.
\begin{lemma}\label{assump}
Let $u_{\text{ms},\ell}$ be the solution to \eqref{eq:multiscale-hissemi}, then there holds
\begin{align*}
\normL{u_{\text{ms},\ell}(\cdot,t)}{D}&\leq \Const{0}(\|f\|_{W^{2,\infty}(I,L^2(D))}+\|u_0\|_{D})&&\\
\normL{\partial^{m}_{t}u_{\text{ms},\ell}(\cdot,t)}{D}&\leq\Const{0} t^{\sigma\alpha-m}\left(\|f\|_{W^{2,\infty}(I,L^2(D))}+|u_0|_{2\sigma}\right)\text{ for } m\geq 1&&.
\end{align*}
Here, and throughout this paper, $\Const{0}$ denotes a positive constant that is independent of $H$, $h$, $\tau_c$ and $\tau_f$ with varying value from context to context.
\end{lemma}
Next we present the truncation error estimate of our proposed time integrator $\bar{\partial}_{\tau_f}^{\alpha}$:
\begin{lemma}\label{eq:approx-numerical-frac}
Let $u_{\text{ms},\ell}$ be the solution to \eqref{eq:multiscale-hissemi}, then there holds
\begin{align*}
\normL{\partial_{t}^{\alpha}u_{\text{ms},\ell}(\cdot,t_n)
-\bar{\partial}_{\tau_f}^{\alpha}u_{\text{ms},\ell}(\cdot,t_{n})}{D}
\leq \Const{0}
&\left(\min\{(n-1)^{-\alpha},1\}\tau_f^{-(1-\sigma)\alpha}+\epsilon \sqrt{t_{n-1}}\right)\\
&\times\left(|u_0|_{2\sigma}+\|f\|_{W^{2,\infty}(I,L^2(D))}\right).
\end{align*}
\end{lemma}
\begin{proof}
An application of the triangle inequality implies
\begin{align*}
\normL{\partial_{t}^{\alpha}u_{\text{ms},\ell}(\cdot,t_n)-
\bar{\partial}_{\tau_f}^{\alpha}u_{\text{ms},\ell}(\cdot,t_n)}{D}
&\leq\normL{\partial_{t}^{\alpha}u_{\text{ms},\ell}(\cdot,t_n)
-\tilde{\partial}_{\tau_f}^{\alpha}u_{\text{ms},\ell}(\cdot,t_n)}{D}\\
&+\normL{\tilde{\partial}_{\tau_f}^{\alpha}u_{\text{ms},\ell}(\cdot,t_n)
-\bar{\partial}_{\tau_f}^{\alpha}u_{\text{ms},\ell}(\cdot,t_n)}{D}.
\end{align*}
The first term can be estimated by \cite[Lemma 3.9]{MR3601002}, and a direct calculation leads to the estimate of the second term.
\end{proof}	
\begin{lemma}[Stability of Problem \eqref{eq:multiscale-his2}.]\label{lemma:stability}
For any given positive parameter $\eta\in (0,1)$, let $\epsilon$ be sufficiently small, s.t.,
\begin{align}\label{assum:epsilon}
\epsilon\leq \min\Big\{\frac{\alpha M_{f}^{\alpha}}{T^{1+\alpha}}\eta, \frac{1}{2T^{1+\alpha}}\Big\}.
\end{align}
Then it holds
\begin{align*}
\tau_f^{\alpha/2} |U_{\text{ms},\ell}^{n}|_1+\normL{U_{\text{ms},\ell}^{n}}{D}\lesssim
\normL{\mathcal{P}_{\ell} u_0}{D}+(1+\eta)c_{\alpha}\tau_f^{\alpha}\sum\limits_{k=0}^{n-1}\normL{\mathcal{P}_{\ell}f(\cdot,t_{k+1})}{D}.
\end{align*}
Moreover, there holds for any $\sigma\in (0,1/2)$ that
\begin{align*}
\tau_f^{\sigma\alpha} |U_{\text{ms},\ell}^{n}|_{2\sigma}
\lesssim
\normL{\mathcal{P}_{\ell} u_0}{D}+(1+\eta)c_{\alpha}\tau_f^{\alpha}\sum\limits_{k=0}^{n-1}\normL{\mathcal{P}_{\ell}f(\cdot,t_{k+1})}{D}.
\end{align*}
\end{lemma}
\begin{proof}
On the one hand, we can derive by \eqref{eqn:cont-decom} and \eqref{eq:form1},
\begin{align*}
 \bar{\partial}_{\tau_f}^{\alpha} U_{\text{ms},\ell}^{n+1}= \tilde{\partial}_{\tau_f}^{\alpha}U_{\text{ms},\ell}^{n+1}
+\frac{1}{\tau_f^{\alpha} c_{\alpha}}\sum_{j=1}^{n}\Big(  a_{n+1,j}U_{\text{ms},\ell}^{j-1}+b_{n+1,j}U_{\text{ms},\ell}^{j} \Big)
\end{align*}
with
\[
|a_{n+1,j}|, |b_{n+1,j}|\leq \frac{1}{2} \alpha(1-\alpha)\tau_{f}^{1+\alpha}\epsilon.
\]
Together with \eqref{eq:l1scheme}, this yields
\begin{align*}
\bar{\partial}_{\tau_f}^{\alpha} U_{\text{ms},\ell}^{n+1}=\frac{1}{\tau_f^{\alpha} c_{\alpha}}\Bigg( v_{n+1}-
\sum_{j=1}^{n}\Big(  b_{j-1}-b_{j}-\xi_j \Big)U_{\text{ms},\ell}^{n+1-j}\Bigg)
\end{align*}
with
\[
|\xi_j|\leq \alpha(1-\alpha)\tau_{f}^{1+\alpha}\epsilon.
\]
On the other hand, \cite[Lemma 3.4]{MR3601002} implies
\begin{align*}
\sum\limits_{j=1}^{n}(b_{j-1}-b_j)(n+1-j)^{\alpha-1}\leq (n+1)^{\alpha-1},
\end{align*}
 which, combining with \eqref{assum:epsilon}, yields
\begin{align*}
\sum\limits_{j=1}^{n}(b_{j-1}-b_j-\xi_j)(n+1-j)^{\alpha-1}
&\leq (n+1)^{\alpha-1}
+\tau_{f}^{1+\alpha}\epsilon \sum\limits_{j=1}^{n}(n+1-j)^{\alpha-1}\\
&\leq (n+1)^{\alpha-1}
+\tau_{f}^{1+\alpha}\epsilon \int_{0}^{n}(n+1-x)^{\alpha-1}\mathrm{d}x\\
&= (n+1)^{\alpha-1}
+\tau_{f}^{1+\alpha}\epsilon \alpha^{-1}(n+1)^{\alpha}\\
&\leq (1+\eta)(n+1)^{\alpha-1}.
\end{align*}
Then the desired result follows from an energy argument by mathematical induction analogous to the proof to \cite[Theorem 3.5]{MR3601002}.
\end{proof}
Note that the stability of Problem \eqref{eq:multiscale-his2} using L2 time stepping scheme and summation of exponential approximation was derived in \cite{zhu2019fast}, which imposed a more restrictive constraint on $\epsilon$ than our proposed requirement \eqref{assum:epsilon}.

Then a combination of Lemma \ref{eq:approx-numerical-frac} and the approximation properties \eqref{eq:glo-energy1} and \eqref{eq:glo-energy2} implies
\begin{theorem}[Convergence rate of Problem \eqref{eq:multiscale-his2}.]
\label{thm:conv11}
Assume that \eqref{assum:epsilon} holds. Let $f\in W^{2,\infty}(I;L^2(D))$ and $u_0\in \dot{H}^{2\sigma}(D)$ with $0<\sigma\leq 1$, then there holds
\begin{align*}
\normL{u(\cdot,t_n)-U_{\text{ms},\ell}^{n}}{D}
&\lesssim \Big(\tau_f^{\sigma\alpha}+\eta(H,\ell)|\log \eta(H,\ell)|\Big)\|f\|_{W^{2,\infty}(I;L^2(D))}\\
&+\Big(\tau_f^{\sigma\alpha}+\eta(H,\ell)|\log \eta(H,\ell)| t_n^{-\alpha(1-\sigma)}\Big)|u_0|_{2\sigma}.
\end{align*}
\end{theorem}
\begin{proof}
The proof is analogous to \cite[Theorems 3.10 and 3.11]{MR3601002}.
\end{proof}
\subsection{Convergence analysis for WEMP Algorithm \ref{algorithm:wavelet+parareal}}
This section is concerned with the convergence of  our main algorithm (Algorithm \ref{algorithm:wavelet+parareal}). We will first derive the approximation properties of the one step coarse time stepping scheme $\mathcal{G}^{{\rm{ms},\ell}}$ in the first component and the jump operator $\mathcal{S}$ in Lemma \ref{lemma:parareal}. The convergence of Algorithm \ref{algorithm:wavelet+parareal} is derived by the mathematical induction on the solution itself, based upon which we can derive the convergence rate of the history information using \eqref{eqn:history-propag}. This will lead to the inequality \eqref{eq:monster1111111}, and then the proof follows from \cite[Theorem 4.1]{li2020wavelet}.

To start with, we first introduce a new norm $\vertiii{\cdot}$ to measure the effect of the history component.
\begin{definition}\label{defn:tnorm}
Let $\Psi:=\begin{bmatrix}\psi_1,\cdots,\psi_{N_{\text{exp}}} \end{bmatrix}$ with each component $\psi_j\in L^2(D)$. Then we define its associated history contribution by
$v_{\rm{his}}:=\sum_{j=1}^{N_{\text{exp}}}\omega_{j}\psi_{j}$.
This vector $\Psi$ is measured by the following norm
\[
\vertiii{\Psi}:=\tau_c^{\alpha}\normL{v_{\rm{his}}}{D}.
\]
\end{definition}
We next prove the boundedness and Lipschitz continuity properties of the first component of the coarse solver $\mathcal{G}^{{\text{ms},\ell}}$ and the jump operator $\mathcal{S}$ in the multiscale space $V_{\text{ms},\ell}$:
\begin{lemma}\label{lemma:parareal}
Let $\begin{bmatrix}U^{0}_{i}\\\Phi^{0}_{i}\end{bmatrix}=
\begin{bmatrix}v_0\\0\end{bmatrix}$ for $i=1,2$ with $v_0\in V_{\rm{ms},\ell}$ being the initial data, and let $v_i,\,\phi_{i,j}\in V_{\rm{ms},\ell}$ with $i=1,2$ and $j=1,2,\cdots,N_{\rm{exp}}$. Denote $\Phi_i:=[\phi_{i,1},\cdots,\phi_{i,N_{\text{exp}}}]$ for $i=1,2$ that represents the history information at the previous iteration. Then for all $n\in \{1,\cdots,M_{c}-1\}$, the following properties hold.
\begin{itemize}
\item[1.]The one step coarse solver $\mathcal{G}^{{\rm{ms},\ell}}$ is Lipschitz in $V_{\rm{ms},\ell}$. There holds
\begin{align*}
\quad\normL{\mathcal{G}^{{\rm{ms},\ell}}(T^n,v_1;\Phi_{1})_1-\mathcal{G}^{{\rm{ms},\ell}}(T^n, v_2;\Phi_{2})_1}{D}
&\leq \alpha\normL{v_1-v_2}{D}+\alpha(1-\alpha)\vertiii{\Phi_1-\Phi_2}.
\end{align*}
\item[2.] The jump operator $\mathcal{S}$ has the following approximation property.
For any $\sigma\in (0,1)$, it holds
\begin{align}\label{eq:app-jump}
\quad&\normL{\mathcal{S}(T^n,v_1;\Phi_1)
-\mathcal{S}(T^n, v_2;\Phi_2)}{D}\leq \Const{0}\tau_c^{\sigma\alpha}\Big({|v_1-v_2|}_{2\sigma}+\alpha\vertiii{\Phi_1-\Phi_2}\Big).
\end{align}
\end{itemize}
\end{lemma}
\begin{proof}
1. Let $e_{{\text{ms}}}^{n+1}:=\mathcal{G}^{{\rm{ms},\ell}}(T^n,
v_1;\Phi_{1})_1-\mathcal{G}^{{\rm{ms},\ell}}(T^n,v_2;\Phi_{2})_1$ and let $v_{i,\rm{his}}:=\sum_{j=1}^{N_{\text{exp}}}\omega_j\phi_{i,j}$ for $i=1,2$, then it holds for all multiscale test function $w_{\text{ms}}\in V_{\text{ms},\ell}$,
\begin{align*}
\int_{D}e_{\text{ms}}^{n+1} w_{\text{ms}}\dx
+c_{\alpha}\tau_c^{\alpha}\int_{D}\kappa \nabla e_{\text{ms}}^{n+1}\cdot \nabla w_{\text{ms}}\dx&=\alpha\int_{D}(v_1-v_2) w_{\text{ms}}\dx\\
&+\alpha(1-\alpha) \tau_c^{\alpha}\int_{D}(v_{1,\rm{his}}-v_{2,\rm{his}}) w_{\text{ms}}\dx.
\end{align*}
Choosing $w_{\text{ms}}:=e_{{\text{ms}}}^{n+1}$ leads to
\begin{align*}
\normL{e_{{\text{ms}}}^{n+1}}{D}^2+c_{\alpha}\tau_c^{\alpha}\normE{e_{{\text{ms}}}^{n+1}}{D}^2&=
\alpha\int_{D}(v_1-v_2) e_{{\text{ms}}}^{n+1}\dx
\\&+\alpha(1-\alpha) \tau_c^{\alpha}\int_{D}e_{{\text{ms}}}^{n+1}(v_{1,\rm{his}}-v_{2,\rm{his}}) \dx.
\end{align*}
Then an application of Young's inequality implies
\begin{align*}
\normL{e_{{\text{ms}}}^{n+1}}{D}^2\leq
\alpha \normL{v_1-v_2}{D}\normL{e_{{\text{ms}}}^{n+1}}{D}+\alpha(1-\alpha)
 \tau_c^{\alpha}\normL{v_{1,\rm{his}}-v_{2,\rm{his}}}{D}\normL{e_{{\text{ms}}}^{n+1}}{D}.
\end{align*}
The a direct calculation proves the first assertion.
		
\noindent 2. To prove the second assertion, let
\begin{align}
e^{n+1}&:=\mathcal{S}(T^n,v_1;\Phi_{1})
-\mathcal{S}(T^n, v_2;\Phi_{2})\nonumber\\
&=\Big(\mathcal{F}^{{\text{ms},\ell}}_{f }(T^{n}, v_1;\Phi_{1})-\mathcal{F}^{{\text{ms},\ell}}_{f}(T^{n}, v_2;\Phi_{2})\Big)-\Big(\mathcal{G}^{{\text{ms},\ell}}(T^n,  v_1;\Phi_{1})-\mathcal{G}^{{\text{ms},\ell}}(T^n,  v_2;\Phi_{2})\Big)\nonumber\\
&=\bar{\mathcal{F}}^{{\text{ms},\ell}}_{f}(T^{n}, v_1-v_2;\Phi_{1}-\Phi_2)
-\bar{\mathcal{G}}^{{\text{ms},\ell}}(T^n, v_1-v_2;\Phi_{1}-\Phi_2)\nonumber\\
&=\left(\bar{\mathcal{F}}^{{\text{ms},\ell}}_{f}(T^{n}, v_1-v_2;\Phi_{1}-\Phi_2)-\bar{\mathcal{F}}^{{\text{ms},\ell}}(T^{n}, v_1-v_2;\Phi_{1}-\Phi_2)\right)\nonumber\\
&-\left(\bar{\mathcal{G}}^{{\text{ms},\ell}}(T^n, v_1-v_2;\Phi_{1}-\Phi_2)
-\bar{\mathcal{F}}^{{\text{ms},\ell}}(T^{n}, v_1-v_2;\Phi_{1}-\Phi_2)\right)
\nonumber\\
&=:e^{n+1}_{f}-e^{n+1}_{c}.\nonumber
\end{align}
To estimate $e^{n+1}$, we only need to derive the estimate for $e_{f}^{n+1}$ and $e_{c}^{n+1}$, separately.

Firstly, we introduce the following semidiscrete Galerkin approximation problem in the spatial domain $D$: let $v:=\bar{\mathcal{F}}^{{\text{ms},\ell}}(T^{n}, v_1-v_2;\Phi_{1}-\Phi_2)_1\in L^2([T^n,T^{n+1}];V_{\text{ms},\ell})$ satisfy
\begin{align}
\forall w_{\text{ms}}\in V_{\text{ms},\ell}: \int_{D}\Dbel v\; w_{\text{ms}}\dx
&+\int_{D}\kappa \nabla v\cdot \nabla w_{\text{ms}}\dx=-\frac{1}{\tau_c^{\alpha}\Gamma(1-\alpha)}\int_{D}(v_1-v_2) w_{\text{ms}}\dx\nonumber\\
&+\frac{\alpha}{\Gamma(1-\alpha)}\int_{D}(v_{1,\rm{his}}-v_{2,\rm{his}}) w_{\text{ms}}\dx\label{eq:yyy}\\
&\qquad\qquad \qquad v(\cdot,T^n)=v_1-v_2.\nonumber
\end{align}
Note that the error equation for $e^{n+1}_{c}$ derived from a fully discretization problem with L1 scheme in the temporal domain is
\begin{align}\label{eq:zzz}
\forall w_{\text{ms}}\in V_{\text{ms},\ell}:\int_{D}\bar{\partial}_{\tau_c,\mathcal{N}}^{\alpha}e_{c}^{n+1}\; w_{\text{ms}}\dx+\int_{D}\kappa \nabla e_{c}^{n+1}\cdot \nabla w_{\text{ms}}\dx
&=\int_{D}w_0\cdot w_{\text{ms}}\dx\\
e_{c}^{n}&=0.\nonumber
\end{align}
with $w_{0}$ being the approximation error from L1 scheme,
\begin{align*}
w_{0}&:=-\Dbel v(\cdot,T^{n+1})+\Dbela{v(\cdot,T^{n+1})},
\end{align*}
which can be estimated by Lemma \ref{eq:approx-numerical-frac},
\[
\normL{w_{0}}{D}\leq \Const{0}\tau_c^{-(1-\sigma)\alpha}\Big({|v_1-v_2|_{2\sigma}}+\alpha\vertiii{\Phi_1-\Phi_2}\Big).
\]
The stability estimate \cite[Theorem 3.5]{MR3601002} implies
\begin{align*}
\normL{e_{c}^{n+1}}{D}\leq \Const{0}\tau_c^{\sigma\alpha}\Big(|v_1-v_2|_{2\sigma}+\alpha\vertiii{\Phi_1-\Phi_2}\Big).
\end{align*}
Analogously, we can obtain the estimate for $e_{f}^{n+1}$, which reads
\begin{align*}
\normL{e_{{f}}^{n+1}}{D}\leq\Const{0} \tau_f^{\alpha}\tau_c^{\sigma\alpha-\alpha}\Big(|{v_1-v_2}|_{2\sigma}+\alpha\tau_f^{\alpha}\tau_c^{-\alpha}\vertiii{\Phi_1-\Phi_2}\Big).
\end{align*}
Note that $\tau_f\ll \tau_c$, then a combination of the two estimates above with the triangle inequality, shows the second assertion.		
\end{proof}
\begin{remark}
The key idea to prove the approximation property of the jump operator $\mathcal{S}$ lies in the auxiliary problem \eqref{eq:yyy}, which is a time-fractional problem defined in the subdomain $[T^n,T^{n+1}]$. Together with the history information on $[0,T^n]$ from the terms $\Phi_1$ and $\Phi_2$, this results in a time-fraction problem on  $[T^n,T^{n+1}]$, with both $\bar{\mathcal{F}}^{{\text{ms},\ell}}_{f}(T^{n}, v_1-v_2;\Phi_{1}-\Phi_2)$ and $\bar{\mathcal{G}}^{{\text{ms},\ell}}(T^n, v_1-v_2;\Phi_{1}-\Phi_2)$ as its temporal discretization under time step size of $\tau_f$ and $\tau_c$, respectively.
\end{remark}
We present in the following a smoothing property of the SOE-scheme \eqref{eq:form} that plays a crucial role in the convergence analysis with rough initial data. This property has been utilized without proof in \cite[Assumption 4.1]{li2020wavelet}. This property is proved by first establishing the global fully discretized scheme for $U_{k}^{n}$, and then derive the global error equation between the multiscale SOE scheme \eqref{eq:multiscale-his2} and Algorithm \ref{algorithm:wavelet+parareal} to eliminate the source term $f$.
\begin{lemma}[Smoothing property of the SOE-scheme \eqref{eq:form}]\label{lemma:smoothing}
Let Assumption \ref{ass:coeff} hold. Assume that the source term $f\in W^{2,\infty}(I; L^2(D))$ and initial data $u_0\in \dot{H}^{2\sigma}(D)$ for some $0<\sigma\leq 1$. Let $\ell\in \mathbb{N}_{+}$ be the level parameter. The coarse time step size and fine time step size are $\tau_c$ and $\tau_f$. Let $U_{\text{ms},\ell}^m\in V_{\text{ms},\ell}$ be the solution to Problem \eqref{eq:multiscale-his2} and let $U_k^n$ be the solution from Algorithm \ref{algorithm:wavelet+parareal} with iteration $k\in\mathbb{N}_{+}$.
 Then for any $\sigma\in (0,1)$, it holds
\begin{equation}\label{eq:globalProp}
\begin{aligned}
|{U_{\text{ms},\ell}^m-U_{k}^{n}}|_{2\sigma}
&\lesssim (T^n)^{-\alpha\sigma}\normL{u_{\text{ms},\ell}^{m'}-U_{k}^{{n-1}}}{D}.
\end{aligned}
\end{equation}
\end{lemma}
\begin{proof}
Note that Algorithm \ref{algorithm:wavelet+parareal} is equivalent to seeking $U_{k}^{n}\in V_{\text{ms},\ell}$ for $n=1,\cdots,M_{c}$, satisfying
\begin{equation}\label{eq:parareal-global}
\left\{
\begin{aligned}
\bar{\partial}_{\tau_c}^{\alpha} (U_k^n,v_{\text{ms}})_D&+a(U_k^n,v_{\text{ms}})=(\mathcal{P}_{\ell}f(\cdot,T^n),v_{\text{ms}})_D+\bar{\partial}_{\tau_c}^{\alpha} (S_{k-1}^n,v_{\text{ms}})_D+a(S_{k-1}^n,v_{\text{ms}})
\quad \text{for all} \, \, v_{\text{ms}}\in V_{\text{ms},\ell}\\
\Phi_{k,j}^n&=\mathcal{H}(\Phi_{k,j}^{n-1}; \tau_c, U_{k}^{n-1}, U_{k}^{n})\text{ for }n\geq 0 \text{ and }
\Phi_{k,j}^0=0 \text{ for all } j=1,\cdots,N_{\text{exp}}\\
U_{k}^0&=\mathcal{P}_{\ell} u_0.
\end{aligned}
\right.
\end{equation}
Here, the local correction $S_{k-1}^n:=\mathcal{S}(T^{n-1},U^{n-1}_{k-1};\Phi^{n-1}_{k-1})$. Note that for all $v_{\text{ms}}\in V_{\text{ms},\ell}$, it holds
\begin{align*}
(\bar{\partial}_{\tau_c}^{\alpha} S_{k-1}^n,v_{\text{ms}})_D+a(S_{k-1}^n,v_{\text{ms}})
=(\bar{\partial}_{\tau_c}^{\alpha}-\bar{\partial}_{\tau_f}^{\alpha})\Big(\mathcal{F}^{\text{ms},\ell}_f(T^{n-1},U_{k-1}^{n-1};\Phi_{k-1}^{n-1})_1, v_{\text{ms}}\Big)_D.
\end{align*}
Let $e_k^n:=U_{k}^{n}-U_{\text{ms},\ell}^m$, then it satisfies
\begin{equation}\label{eq:parareal-global}
\left\{
\begin{aligned}
\bar{\partial}_{\tau_c}^{\alpha} (e_k^n,v_{\text{ms}})_D&+a(e_k^n,v_{\text{ms}})=(R_n,v_{\text{ms}})_D
\quad \text{for all} \, \, v_{\text{ms}}\in V_{\text{ms},\ell}\\
f_{j,k}[n]&=\mathcal{H}(f_{j,k}[n-1]; \tau_c, e_{k}^{n-1}, e_{k}^{n})\text{ for }n\geq 0 \text{ and }
f_{j,k}[0]=0 \text{ for all } j=1,\cdots,N_{\text{exp}}\\
e_{k}^0&=0
\end{aligned}
\right.
\end{equation}
with the source term $R_n$ being
\begin{align*}
R_n&:=-(\bar{\partial}_{\tau_c}^{\alpha}-\bar{\partial}_{\tau_f}^{\alpha}) U_{\text{ms},\ell}^m+(\bar{\partial}_{\tau_c}^{\alpha}-\bar{\partial}_{\tau_f}^{\alpha})\mathcal{F}^{\text{ms},\ell}_f(T^{n-1},U_{k-1}^{n-1};\Phi_{k-1}^{n-1})_1\\
&=-(\bar{\partial}_{\tau_c}^{\alpha}-\bar{\partial}_{\tau_f}^{\alpha}) \Big(U_{\text{ms},\ell}^{m}-\mathcal{F}^{\text{ms},\ell}_f(T^{n-1},U_{k-1}^{n-1};\Phi_{k-1}^{n-1})_1\Big)\\
&=-(\bar{\partial}_{\tau_c}^{\alpha}-\bar{\partial}_{\tau_f}^{\alpha}) \mathcal{F}^{\text{ms},\ell}_f(T^{n-1},U_{\text{ms},\ell}^{m'}-U_{k-1}^{n-1};\Phi_{\text{ms},\ell}^{m'}-\Phi_{k-1}^{n-1})_1.
\end{align*}
Note that $e_{k-1}^{n-1}=U_{\text{ms},\ell}^{m'}-U_{k-1}^{n-1}$, Lemma \ref{eq:approx-numerical-frac} implies
\begin{align*}
\|R_n\|_D\leq C_0 (n-1)^{-\alpha}\tau_c^{-\alpha}\Big( \|e_{k-1}^{n-1}\|_D+\alpha\vertiii{\Phi_{\text{ms},\ell}^{m'}-\Phi_{k-1}^{n-1}}\Big).
\end{align*}
By an application of Lemma \ref{lemma:stability} to the error equation \eqref{eq:parareal-global}, we obtain
\begin{align*}
\tau_c^{\alpha\sigma}|e_k^n|_{2\sigma}&\leq C_{\alpha}\tau_c^{\alpha}\sum_{m=1}^{n}\|R_n\|_D\\
&\leq C_{0}n^{1-\alpha} \left(\frac{n}{n-1}\right)^{\alpha}\Big( \|e_{k-1}^{n-1}\|_D+\alpha\vertiii{\Phi_{\text{ms},\ell}^{m'}-\Phi_{k-1}^{n-1}}\Big)\\
&\leq C_{0}n^{1-\alpha}\Big( \|e_{k-1}^{n-1}\|_D+\alpha\vertiii{\Phi_{\text{ms},\ell}^{m'}-\Phi_{k-1}^{n-1}}\Big).
\end{align*}
Consequently, we obtain
\begin{align*}
|e_k^n|_{2\sigma}\leq C_{0}n^{1-\alpha}\times \tau_c^{-\alpha\sigma} \Big( \|e_{k-1}^{n-1}\|_D+\alpha\vertiii{\Phi_{\text{ms},\ell}^{m'}-\Phi_{k-1}^{n-1}}\Big).
\end{align*}

Then one can follow the proof in \cite[Lemma 7.3]{thomee1984galerkin} and \cite[Theorem 2.2]{Sakamoto-Yamamoto11} the desired result, under the condition that $\tau_c$ is of reasonable size.
\end{proof}
We present in the next theorem the convergence rate of Algorithm \ref{algorithm:wavelet+parareal} to Problems \eqref{eqn:pde} in pointwise-in-time in $L^2(D)$-norm. To derive it, we first decompose the error from Algorithm \ref{algorithm:wavelet+parareal} as a summation of the error from  Problem \eqref{eq:multiscale-his2} and the error from parareal algorithm. The former part is estimated by Theorem \ref{thm:conv11}.
\begin{theorem}
[Error estimate of Algorithm \ref{algorithm:wavelet+parareal} to Problem \eqref{eqn:pde} in $L^2(D)$-norm]\label{prop:wavelet-based}
Let Assumption \ref{ass:coeff} hold. Assume that the source term $f\in W^{2,\infty}(I; L^2(D))$ and initial data $u_0\in L^2(D)$. Let $\ell\in \mathbb{N}_{+}$ be the level parameter. The coarse time step size and fine time step size are $\tau_c$ and $\tau_f$. Assume that $\tau_f$ is sufficiently small, s.t.,
\begin{align}\label{assm:gamma}
\mathrm{e}^{-\tau_c\gamma}((1-\alpha)^{-1}+\epsilon\tau_c^{1+\alpha}/2)
<\eta(\alpha,\theta):=\frac{\alpha}{1+\alpha}(\theta-\alpha) \text{  for given  } \theta\in [\alpha,1).
\end{align}
Let $u(\cdot,T^n)\in V$ be the solution to Problem \eqref{eqn:pde} and let $U_k^n$ be the solution from Algorithm \ref{algorithm:wavelet+parareal} with iteration $k\in\mathbb{N}_{+}$. Assume that the iteration number $k$ is not large.
Then for any $\sigma\in (0,1)$, it holds
\begin{equation}\label{eq:waveletErr}
\begin{aligned}
\normL{u(\cdot,T^n)-U_{k}^{n}}{D}&\lesssim \left({\tau_f}^{\sigma\alpha}+|\eta(H,\ell)\log \eta(H,\ell)|^2\right)\|f\|_{W^{2,\infty}(0,T;L^2(D))}\\
&+\left({\tau_f}^{\sigma\alpha}+\eta(H,\ell)|\log \eta(H,\ell)|(T^n)^{-\alpha}+\Pi_{j=0}^{k}\Big(\frac{\tau_c}{T^{n-j}}\Big)^{\alpha\sigma} \right)|u_0|_{2\sigma}.
\end{aligned}
\end{equation}
\end{theorem}
\begin{proof}
We first split the error as the summation of $\normL{u(\cdot,T^n)-U_{\text{ms},\ell}^{m}}{D}$ and $\normL{U_{\text{ms},\ell}^{m}-U_{k}^{n}}{D}$ for $m:=\tau_c/{\tau_f} \times n$ with $T^n=t_m$. Similarly, let $m\rq{}:=\tau_c/{\tau_f} \times (n-1)$, then it holds $t_{m\rq{}}=T^{n-1}$. Here,  $U_{\text{ms},\ell}^{m}$ is the solution to Problem \eqref{eq:multiscale-his2} and the error estimate is presented in Theorem \ref{thm:conv11}.

We next estimate the error induced by parareal algorithm in the multiscale method, i.e., the second term $\normL{U_{\text{ms},\ell}^{m}-U_{k}^{n}}{D}$. We will prove:
\begin{align}\label{eq:parareal-est}
e_k^n:=\normL{U_{\text{ms},\ell}^{m}-U_{k}^{n}}{D}
\leq\Const{0} \Pi_{j=0}^{k}\Big(\frac{\tau_c}{T^{n-j}}\Big)^{1-\alpha} |u_0|_{2\sigma}.
\end{align}
Obviously, the inequality \eqref{eq:parareal-est} holds when $k=0$. Assume that it holds for iteration $k$ for some $k\in \mathbb{N}_{+}$. We will show that it holds for the next iteration $k+1$. Before that, we first derive the error estimate for the history information. Let $\theta\in [\alpha,1)$ be a positive parameter, then Lemma \ref{lem:222xxx} implies
\begin{align}
d_{k}^n:=\vertiii{\Phi^{m}_{\text{ms},\ell}-\Phi^{n}_{k}}
&\leq \mathrm{e}^{-\tau_c\gamma}\vertiii{\Phi^{m'}_{\text{ms},\ell}-\Phi^{n-1}_{k}}\\&+
\mathrm{e}^{-\tau_c\gamma}((1-\alpha)^{-1}+\epsilon\tau_c^{1+\alpha}/2)
\Big(\normL{u_{\text{ms},\ell}^{m'}-U_{k}^{n-1}}{D}+\normL{u_{\text{ms},\ell}^{m}-U_{k}^{n}}{D}\Big)\nonumber\\
&\leq \eta(\alpha, \theta)\left(d_{k}^{n-1}+e_{k}^{n-1}+e_k^n\right).\label{eq:yyy1}
\end{align}
Here, $\gamma$ is defined in \eqref{not:gamma} and we have used \eqref{assm:gamma} in the last inequality.

Apply this inequality iteratively, we obtain
\begin{align}\label{eq:xxx1}
d_{k}^n\leq
(1+\eta(\alpha,\theta))\sum_{i=1}^{n-k}\eta(\alpha,\theta)^{i-1}e_{k}^{n-i}
+\eta(\alpha,\theta)e_{k}^{n}.
\end{align}
By \eqref{eq:parareal-est}, we derive
\begin{align}\label{eq:zzz1}
d_{k}^n\leq\frac{2\Const{0}}{1-\eta(\alpha,\theta)} \Pi_{j=0}^{k}\Big(\frac{\tau_c}{T^{n-j}}\Big)^{\alpha\sigma}|u_0|_{2\sigma}.
\end{align}
Next we estimate $e^n_{k+1}$. We can obtain from Algorithm \ref{algorithm:wavelet+parareal}:
\begin{align*}	
U_{\text{ms},\ell}^{m}-U_{k+1}^{n}&=\mathcal{S}(T^{n-1},U_{\text{ms},\ell}^{m\rq{}};\Phi^{m'}_{\text{ms},\ell})
-\mathcal{S}(T^{n-1},U_{k}^{n-1};\Phi^{n-1}_{k})\\
&+\mathcal{G}^{{\text{ms},\ell}}(T^{n-1},U_{\text{ms},\ell}^{m\rq{}};\Phi^{m'}_{\text{ms},\ell})-\mathcal{G}^{{\text{ms},\ell}}(T^{n-1},U_{k+1}^{n-1};\Phi^{n-1}_{k+1}).
\end{align*}
Consequently, an application of Lemma \ref{lemma:parareal} and Lemma \ref{lemma:smoothing} leads to
\begin{align*}
&e_{k+1}^n=\normL{U_{\text{ms},\ell}^{m}-U_{k+1}^{n}}{D}\\
&\leq\normL{\mathcal{S}(T^{n-1},U_{\text{ms},\ell}^{m\rq{}};\Phi^{m'}_{\text{ms},\ell})
-\mathcal{S}(T^{n-1},U_{k}^{n-1};\Phi^{n-1}_{k})}{D}+\normL{\mathcal{G}^{{\text{ms},\ell}}(T^{n-1},u_{\text{ms},\ell}^{m\rq{}};\Phi^{m'}_{\text{ms},\ell})-\mathcal{G}^{{\text{ms},\ell}}(T^{n-1},U_{k+1}^{n-1};\Phi^{n-1}_{k+1})}{D}\\
&\leq \Const{0} \Big(\frac{1}{T^{n}}\Big)^{\alpha\sigma}\tau_c^{\alpha\sigma}\Big(\normL{U_{\text{ms},\ell}^{m\rq{}}-U_{k}^{n-1}}{D}
+\alpha\vertiii{\Phi^{m'}_{\text{ms},\ell}-\Phi^{n-1}_{k}}\Big)
+\alpha\Big(\normL{U_{\text{ms},\ell}^{m\rq{}}-U_{k+1}^{n-1}}{D}
+(1-\alpha)\vertiii{\Phi^{m'}_{\text{ms},\ell}-\Phi^{n-1}_{k+1}}\Big)
\end{align*}
for any $\sigma\in (0,1)$. That is,
\begin{align*}
e_{k+1}^n\leq \Const{0}  \Big(\frac{\tau_c}{T^{n}}\Big)^{\alpha\sigma}\Big(e_k^{n-1}+\alpha d_k^{n-1}\Big)
+\alpha\Big(e_{k+1}^{n-1}+(1-\alpha)d_{k+1}^{n-1}\Big).
\end{align*}
Let $\beta:=\frac{1}{1+\alpha}$ be a positive parameter, then together with the inequality \eqref{eq:yyy1}, we obtain
\begin{align*}
(1-\beta)e_{k+1}^n+\beta d_{k+1}^n
&\leq \Const{0}(1-\beta+\beta\eta(\alpha,\theta)(1+\alpha))  \Big(\frac{\tau_c}{T^{n}}\Big)^{\alpha\sigma}\Big(e_k^{n-1}+\alpha d_k^{n-1}\Big)\\&
+\left((1-\beta)\alpha+\beta\eta(\alpha,\theta)(1+\alpha)\right)e_{k+1}^{n-1}\\&
+\left((1-\beta)\alpha(1-\alpha)+\beta\eta(\alpha,\theta)(1+\alpha-\alpha^2)\right)d_{k+1}^{n-1}.
\end{align*}
Notice that
\begin{align*}
(1-\beta)\alpha+\beta\eta(\alpha,\theta)(1+\alpha)=\theta(1-\beta),
\end{align*}
and
\begin{align*}
(1-\beta)\alpha(1-\alpha)+\beta\eta(\alpha,\theta)(1+\alpha-\alpha^2)\leq \theta\beta.
\end{align*}
These, together with inequalities \eqref{eq:parareal-est} and \eqref{eq:zzz1}, lead to
\begin{align*}
(1-\beta)e_{k+1}^n+\beta d_{k+1}^n
&\leq \Const{0}^2\Bigg(1+\frac{2\alpha}{1-\eta(\alpha,\theta)} \Bigg)\Pi_{j=0}^{k+1}\Big(\frac{\tau_c}{T^{n-j}}\Big)^{\alpha\sigma}|u_0|_{2\sigma}\\&
+\theta\Bigg((1-\beta)e_{k+1}^{n-1}+\beta d_{k+1}^{n-1}\Bigg).
\end{align*}
Consequently, we derive
\begin{align}\label{eq:monster1111111}
(1-\beta)e_{k+1}^n+\beta d_{k+1}^n
\leq \Const{0}^2\Bigg(1+\frac{2\alpha}{1-\eta(\alpha,\theta)} \Bigg)\Pi_{j=0}^{k+1}\Big(\frac{\tau_c}{T^{n-j}}\Big)^{\alpha\sigma}|u_0|_{2\sigma}
+\theta\Bigg((1-\beta)e_{k+1}^{n-1}+\beta d_{k+1}^{n-1}\Bigg).
\end{align}
Repeat this recurrence relation, we obtain
\begin{align*}
(1-\beta)e_{k+1}^n+\beta d_{k+1}^n
\leq \frac{1}{1-\theta}\Const{0}^2\Bigg(1+\frac{2\alpha}{1-\eta(\alpha,\theta)} \Bigg)\Pi_{j=0}^{k+1}\Big(\frac{\tau_c}{T^{n-j}}\Big)^{\alpha\sigma}|u_0|_{2\sigma}\times
A(n,k),
\end{align*}
where
\begin{align*}
A(n,k)&:=\sum\limits_{j=1}^{n-(k+1)}\theta^{j-1} \Big(\Pi_{i=0}^{k+1}\frac{T^{n-i}}{T^{n-j-i}}\Big)^{\alpha\sigma}\\
&= \sum\limits_{m=k+1}^{n-1}\theta^{n-1-m} \Big(\Pi_{i=0}^{k+1}\frac{n-i}{m-i}\Big)^{\alpha\sigma}\\
&\leq \sum\limits_{m=k+1}^{n-1}\theta^{n-1-m} \Big(\Pi_{i=0}^{k+1}\frac{n-i}{m-i}\Big).
\end{align*}
The last inequality above holds because $\alpha\sigma<1$. Notice that we have proved in \cite[Theorem 4.1]{li2020wavelet} that $A(n,k)$ is uniformly bounded with respect to $n$ when $k$ is small. Consequently, this proves the case for $k+1$, and completes the proof.
\end{proof}
\section{Numerical experiment}\label{sec:num}
In this section, we perform a series of numerical experiments to demonstrate the performance of the proposed WEMP Algorithm. Let the spatial domain be $D=[0,1]^2$. The coarse-scale and fine-scale mesh sizes are $H=\frac{1}{10}$ and $h=\frac{1}{200}$. The coarse-scale and fine-scale time step sizes are $\tau_c=0.1$ and $\tau_f=10^{-4}$. For all the numerical tests, we take the level parameter $\ell=2$. A larger $\ell$ can significantly reduce the spatial error. We take a smooth initial data,
\begin{align*}
u_0(x,y)=x(1-x)y(1-y).
\end{align*}
The error is measured in $L^2(D)$-relative error and $H^1(D)$-relative error, defined by
\begin{align*}
\frac{\|u_h^n-\hat{u}_h^n\|_{D}}{\|u_h^n\|_{D}} \times 100 \qquad \frac{\|u_h^n-\hat{u}_h^n\|_{H^1_{\kappa}(D)}}{\|u_h^n\|_{H^1_{\kappa}(D)}}\times 100, 
\end{align*}
where $u_h^n$ and $\hat{u}_h^n$ denote the reference solution and its approximation at $t_n$, respectively. 
\subsection{The influence of $\epsilon$ and $\alpha$}
We study in this section the influence of $\epsilon$ and $\alpha$ to the accuracy of \eqref{eqn:weakform_h} when we replace the L1 scheme $\tilde{\partial}_{\tau_f}^{\alpha}$ with the summation of exponential approximation $\bar{\partial}_{\tau_f}^{\alpha}$, i.e., we solve for $U^{n}_h\in V_{h}$ for $n=1,\cdots,M_{f}$ satisfying
\begin{equation*}
\left\{
\begin{aligned}
\Big(\frac{1}{\tau_f^{\alpha}c_{\alpha}}+ A_h\Big)U^{n}_h
&=\frac{\alpha}{\tau_f^{\alpha}c_{\alpha}}U^{n-1}_h+
\frac{\alpha}{\Gamma(1-\alpha)}
\sum_{j=1}^{N_{\text{exp}}}\omega_{j}\psi^{n-1}_{h,j}
+\frac{1}{\Gamma(1-\alpha)t_{n}^{\alpha}} U^0_h+F^{n}_h \\
\psi^{n}_{h,j}&=\mathcal{H}(\psi^{n-1}_{h,j}; \tau_f, U_h^{n-1}, U_h^n)\text{ for }n\geq 0\text{ and } \psi^{0}_{h,j}=0\text{ for all } j=1,\cdots,N_{\text{exp}}\\
U_h^0&=I_h u_0.
\end{aligned}
\right.
\end{equation*}
Let $T:=1$, then \eqref{assum:epsilon} implies the minimum requirement of $\epsilon$ is $\epsilon\leq \epsilon^*:=0.5$.
\subsubsection{Smooth source term}
The first series of numerical tests are presented in Figures \ref{fig:test11} and \ref{fig:test12}, with a smooth source data, different numbers of terms in the SOE approximation and different values of $\alpha=0.1$ and $\alpha=0.9$. The source term is
\begin{align*}
f(x,y,t)=xyt.
\end{align*}
One can observe from both cases that a more economic numerical scheme $\bar{\partial}_{\tau_f}^{\alpha}$ yields almost the same accuracy even with a much larger tolerance $\epsilon\gg \epsilon^*$. For example, using a number of $N_{\text{exp}}=19$ summations with $\epsilon =24823$, both $L^2(D)$-relative error and $H^1(D)$-relative error are below 0.15\% for $\alpha=0.9$ in Figure \ref{fig:test12}. The performance for a smaller $\alpha=0.1$ is better as it is shown in Figure \ref{fig:test11} that both $L^2(D)$-relative error and $H^1(D)$-relative error are below 0.0003\% in this case.
\begin{figure} [H]
\centering
\begin{subfigure}[b]{0.3\textwidth}
\centering
\includegraphics[width=1.3\textwidth,trim={9cm 0.5cm 5cm 0.8cm},clip]{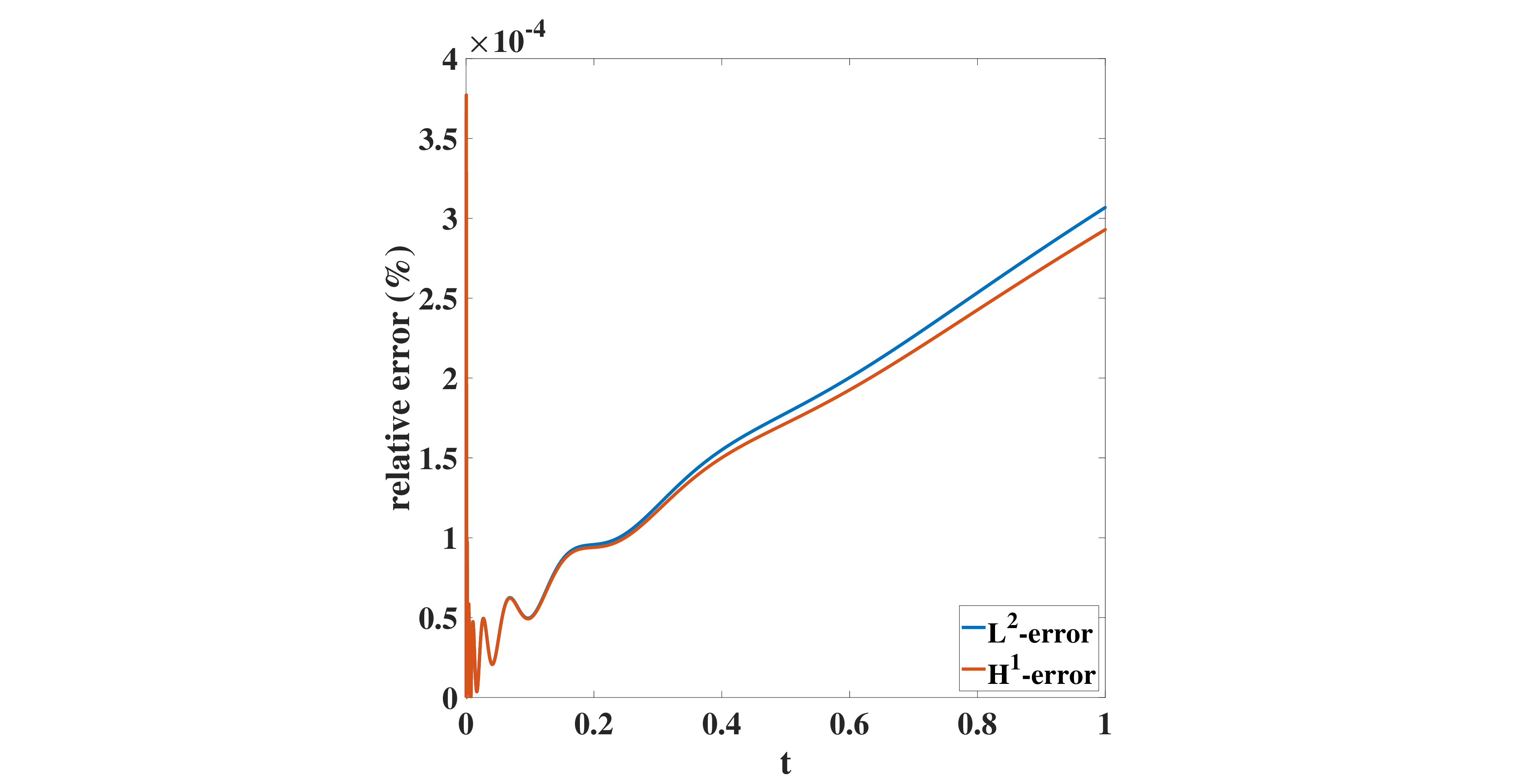}
\caption{$(\epsilon, N_{\text{exp}})=(15.662,19)$}
\end{subfigure}
\begin{subfigure}[b]{0.3\textwidth}
\centering
\includegraphics[width=1.3\textwidth,trim={9cm 0.5cm 5cm 0.8cm},clip]{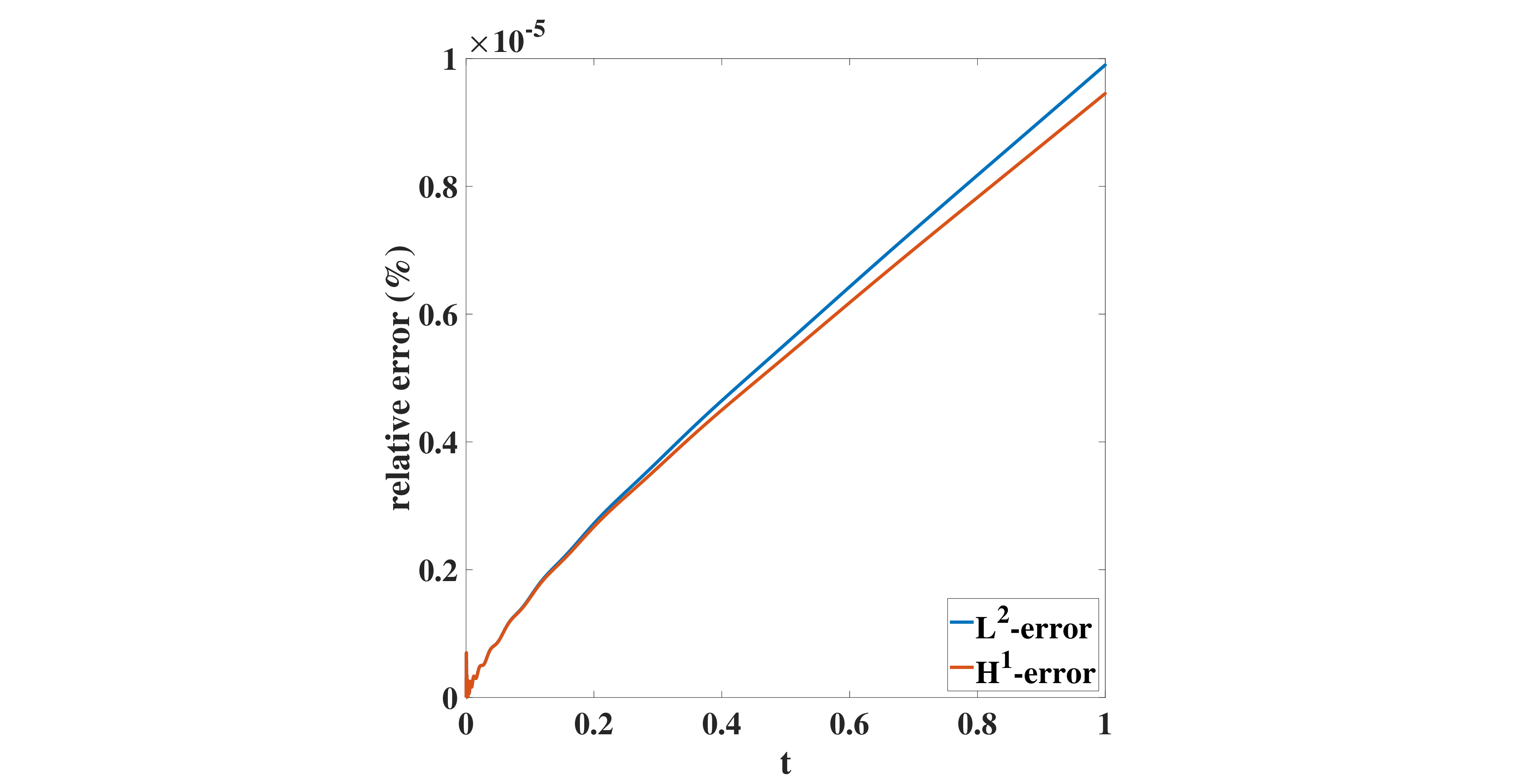}
\caption{$(\epsilon, N_{\text{exp}})=(2.4384e-1,30)$ }
\end{subfigure}
\begin{subfigure}[b]{0.3\textwidth}
\centering
\includegraphics[width=1.3\textwidth,trim={9cm 0.5cm 5cm 0.8cm},clip]{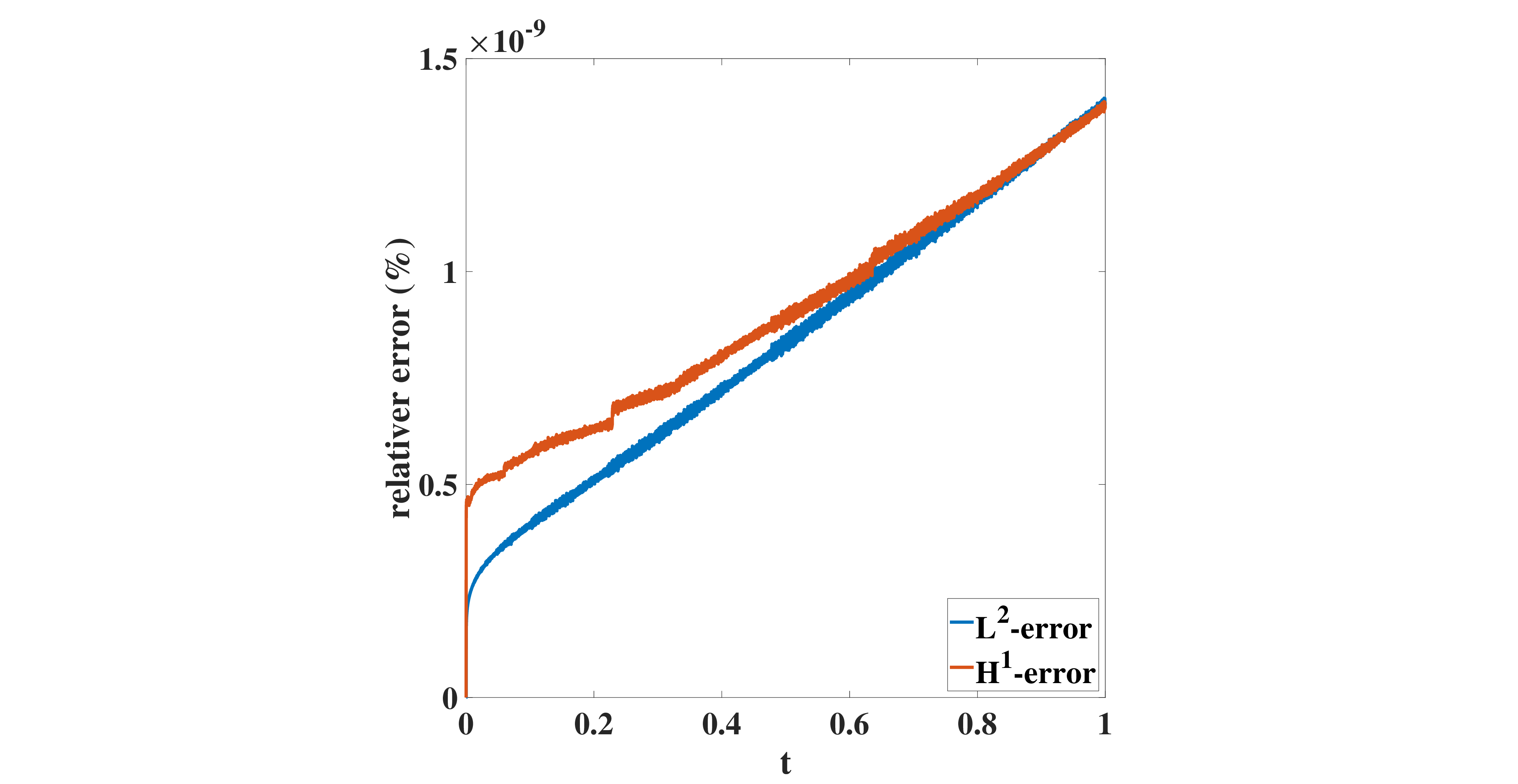}
\caption{$(\epsilon, N_{\text{exp}})=(6.8884e-5,52)$.}
\end{subfigure}
\caption{$\alpha=0.1$, $\epsilon^*=0.5$.}
\label{fig:test11}
\end{figure}

\begin{figure} [H]
\centering
\begin{subfigure}[b]{0.3\textwidth}
\centering
\includegraphics[width=1.3\textwidth,trim={9cm 0.5cm 5cm 0.8cm},clip]{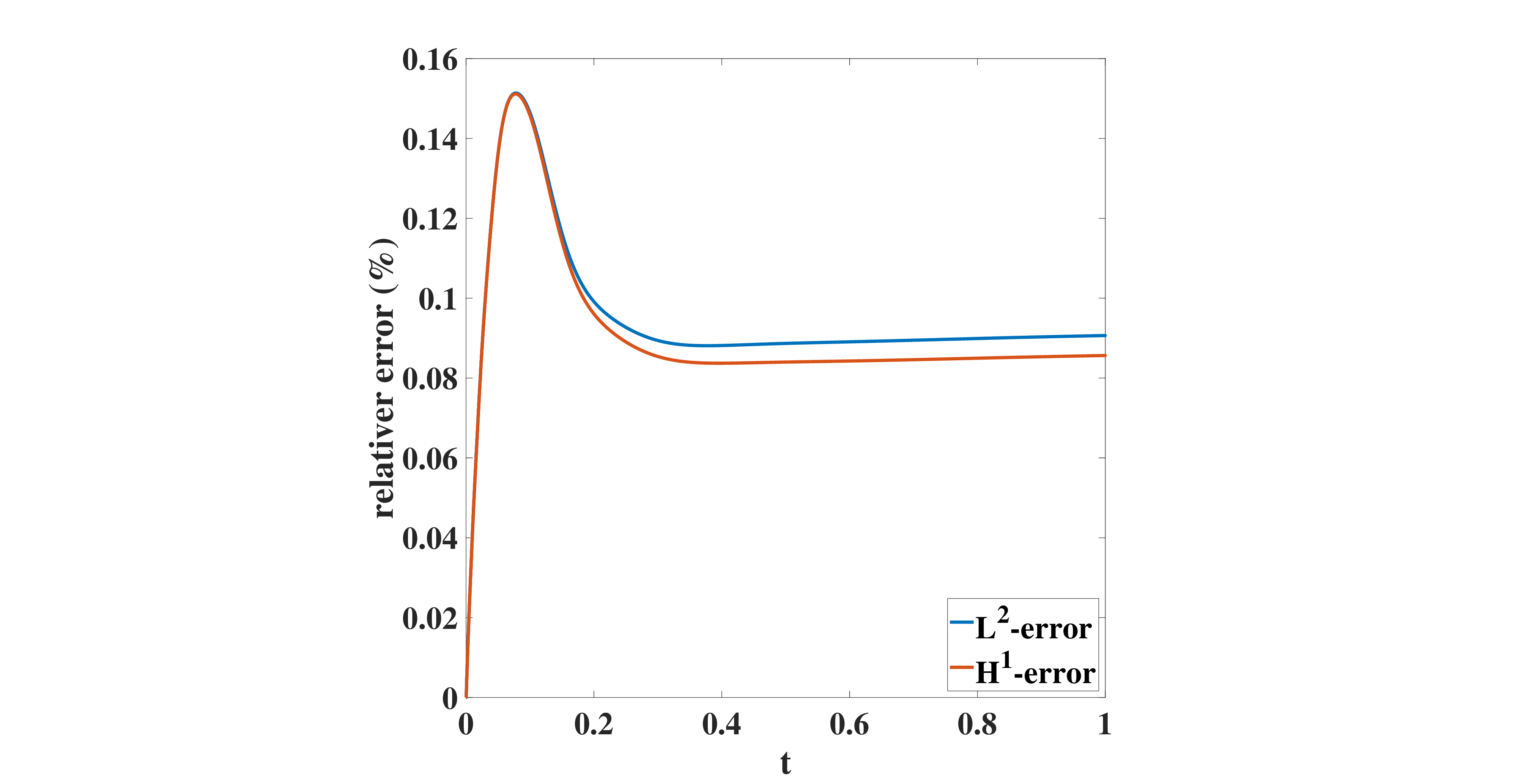}
\caption{$(\epsilon, N_{\text{exp}})=(2.4823e4,19)$}
\end{subfigure}
\begin{subfigure}[b]{0.3\textwidth}
\centering
\includegraphics[width=1.3\textwidth,trim={9cm 0.5cm 5cm 0.8cm},clip]{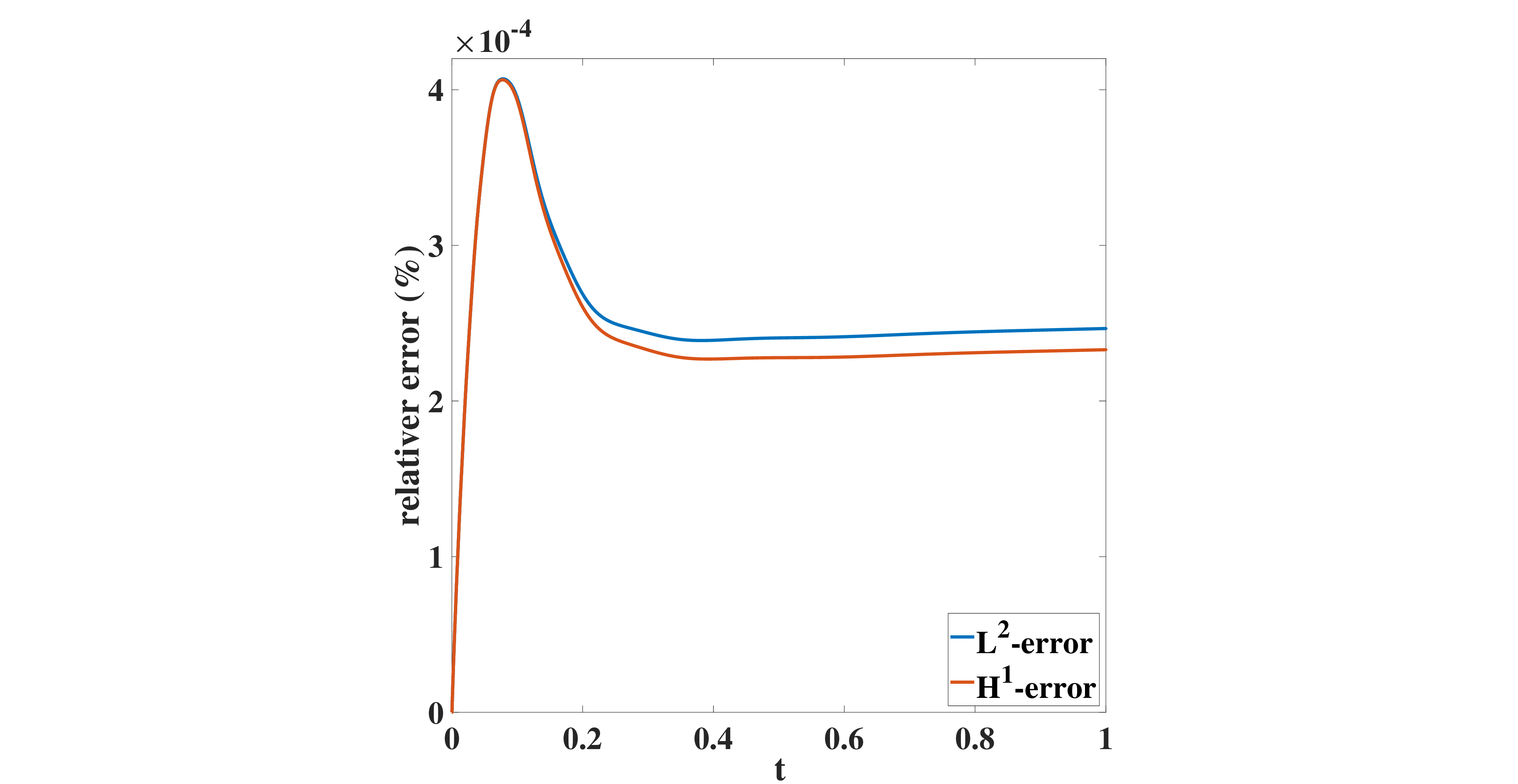}
\caption{$(\epsilon, N_{\text{exp}})=(386.4684,30)$}
\end{subfigure}
\begin{subfigure}[b]{0.3\textwidth}
\centering
\includegraphics[width=1.3\textwidth,trim={9cm 0.5cm 5cm 0.8cm},clip]{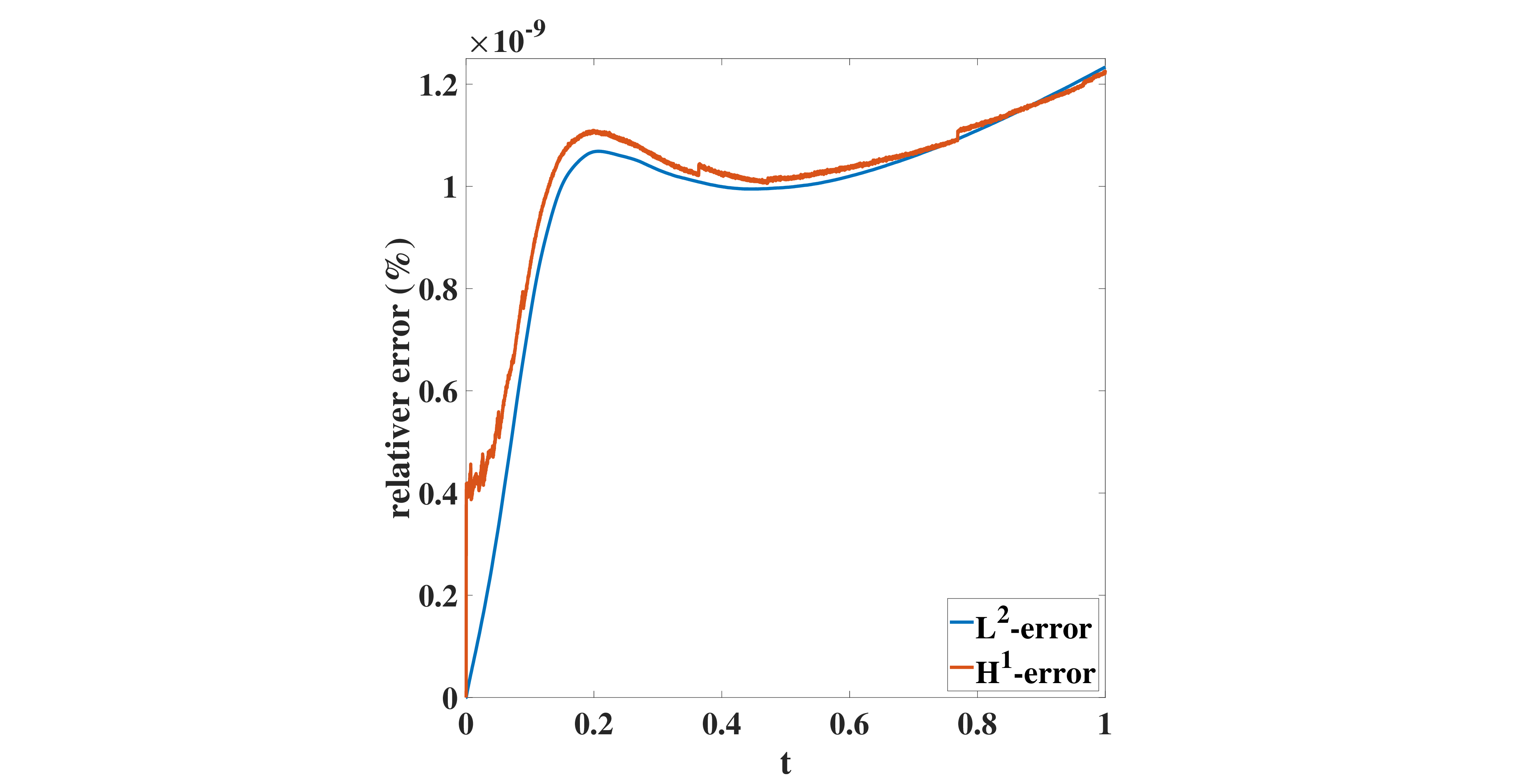}
\caption{$(\epsilon, N_{\text{exp}})=(0.2239,52)$.}
\end{subfigure}
\caption{$\alpha=0.9$, $\epsilon^*=0.5$.}
\label{fig:test12}
\end{figure}

\subsubsection{Rough source term}
The next series of numerical tests are presented in Figure \ref{fig:test11rough}, with a rough source term and different numbers of terms in the SOE approximation. The rough source term has jump along the temporal variable, defined by
\begin{align*}
f(x,y,t)=sgn(cos(2\pi t))xy.
\end{align*}
Compared with Figure \ref{fig:test12}, the influence of rough source term is negligible. We observe similar behavior for varying $\alpha$. {{For the sake of conciseness, we only present the results for $\alpha=0.9$ in Figure \ref{fig:test11rough}.}}
\begin{figure} [H]
\centering
\begin{subfigure}[b]{0.3\textwidth}
\centering
\includegraphics[width=1.3\textwidth,trim={9cm 0.5cm 5cm 1cm},clip]{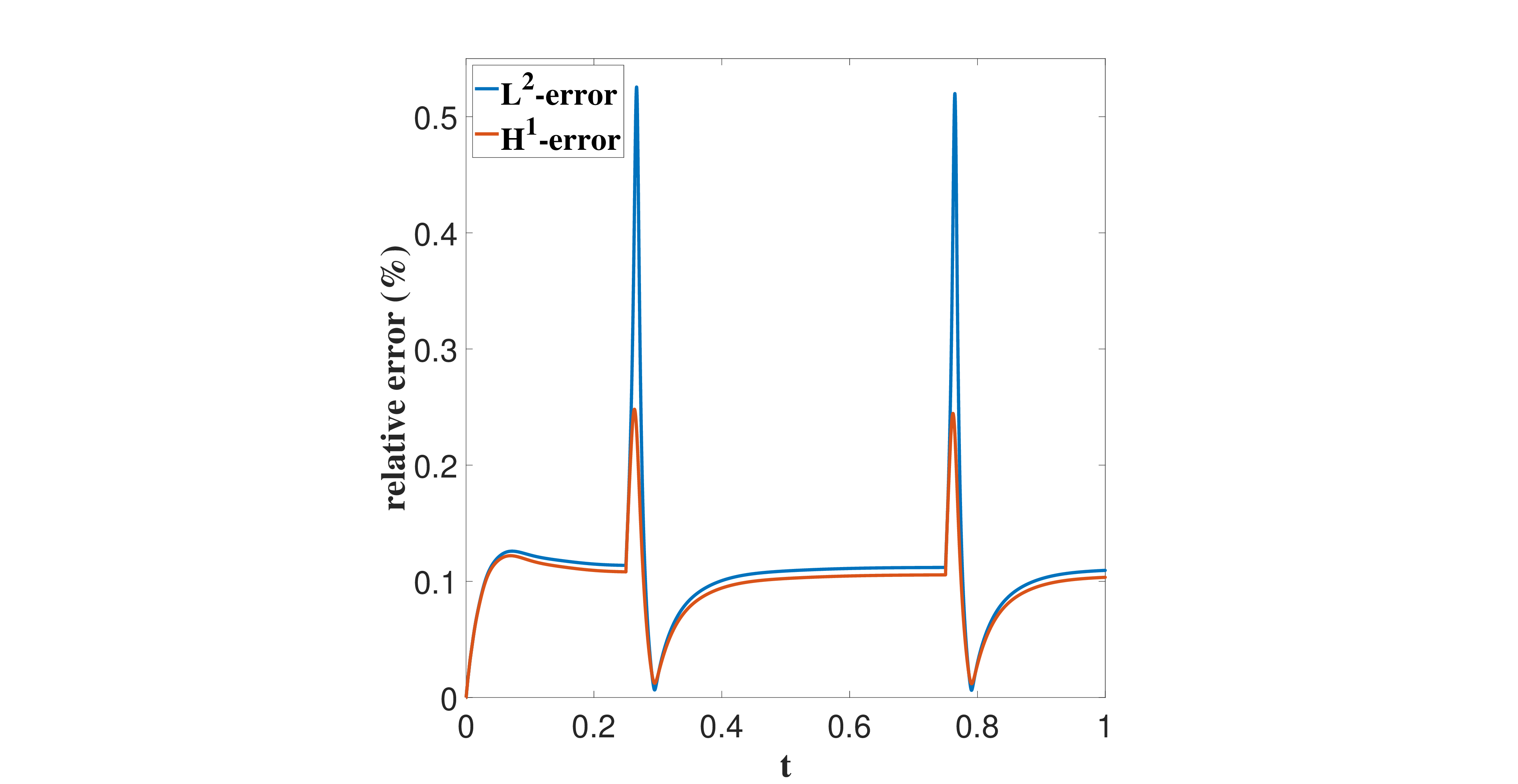}
\caption{$(\epsilon, N_{\text{exp}})=(7.7843e4,19)$}
\end{subfigure}
\begin{subfigure}[b]{0.3\textwidth}
\centering
\includegraphics[width=1.3\textwidth,trim={9cm 0.5cm 5cm 1cm},clip]{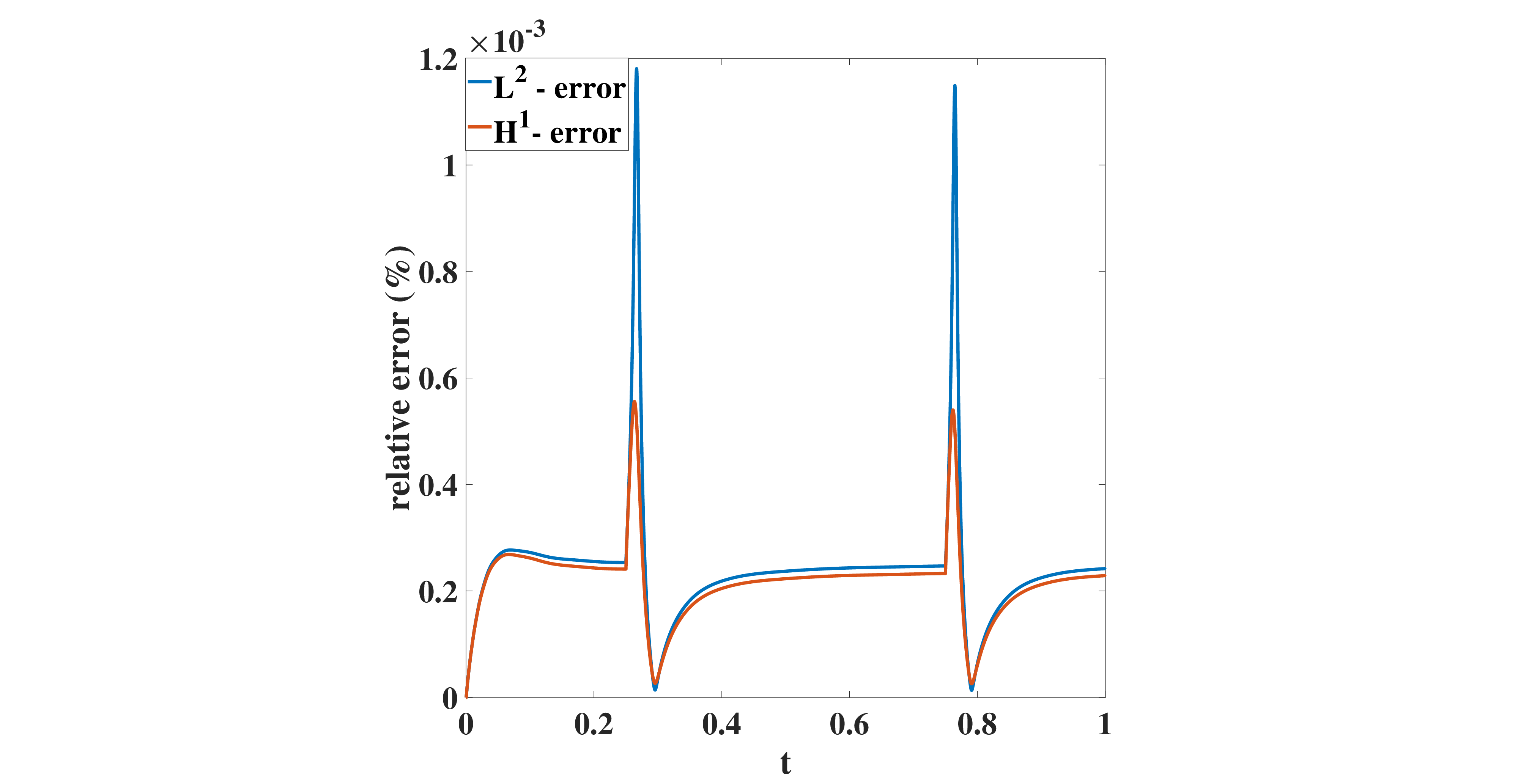}
\caption{$(\epsilon, N_{\text{exp}})=(970.7445,30)$ }
\end{subfigure}
\begin{subfigure}[b]{0.3\textwidth}
\centering
\includegraphics[width=1.3\textwidth,trim={9cm 0.5cm 5cm 1cm},clip]{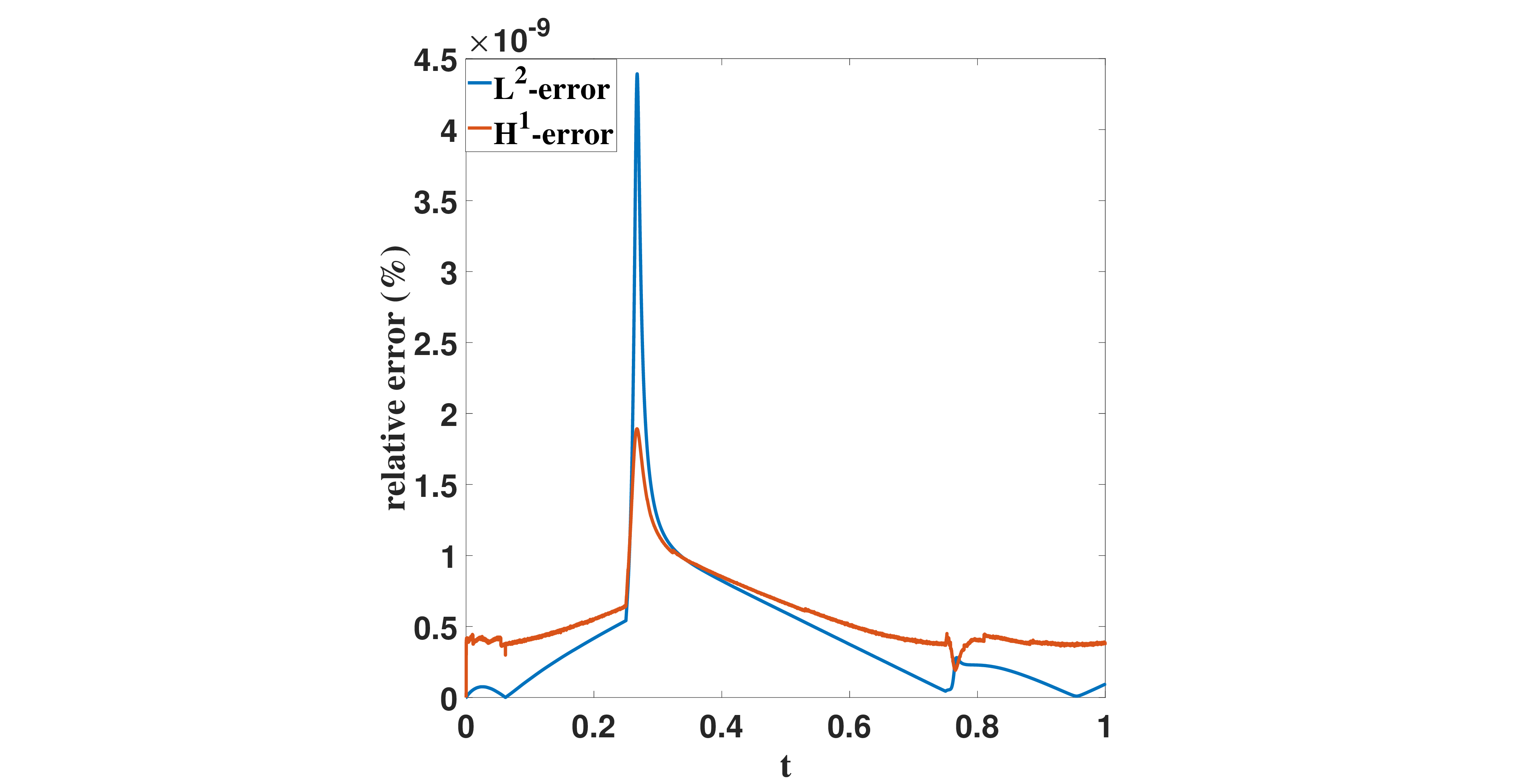}
\caption{$(\epsilon, N_{\text{exp}})=(0.2742,52)$.}
\end{subfigure}
\caption{$\alpha=0.9$, $\epsilon^*=0.5$.}\label{fig:test11rough}
\end{figure}

\subsection{The performance of Algorithm \ref{algorithm:wavelet+parareal}}
Next, we present the performance of Algorithm \ref{algorithm:wavelet+parareal} with different $\alpha=0.1, 0.5, 0.9$, $\epsilon$ and iteration number $k=1, 2, 3$. The convergence history is presented in Figure \ref{fig:0.1}, Figure \ref{fig:0.5} and Figure \ref{fig:0.9}. For all tests, Algorithm \ref{algorithm:wavelet+parareal} converges within 3 iterations{{, this explains the fact that most lines in these figures overlap.}} We observe that the relative $L^2(D)$-error is less than $5\%$ in the 3 iteration. We can not reduce further this error due to the coarse scale mesh size $H=1/10$. To reduce this error further, we need to decrease $H$. Slightly more iteration numbers are needed for larger parameter $\alpha$. 
\begin{figure} [H]
\centering
\begin{subfigure}[b]{0.3\textwidth}
\centering
\includegraphics[width=1.3\textwidth,trim={9cm 0.2cm 5cm .5cm},clip]{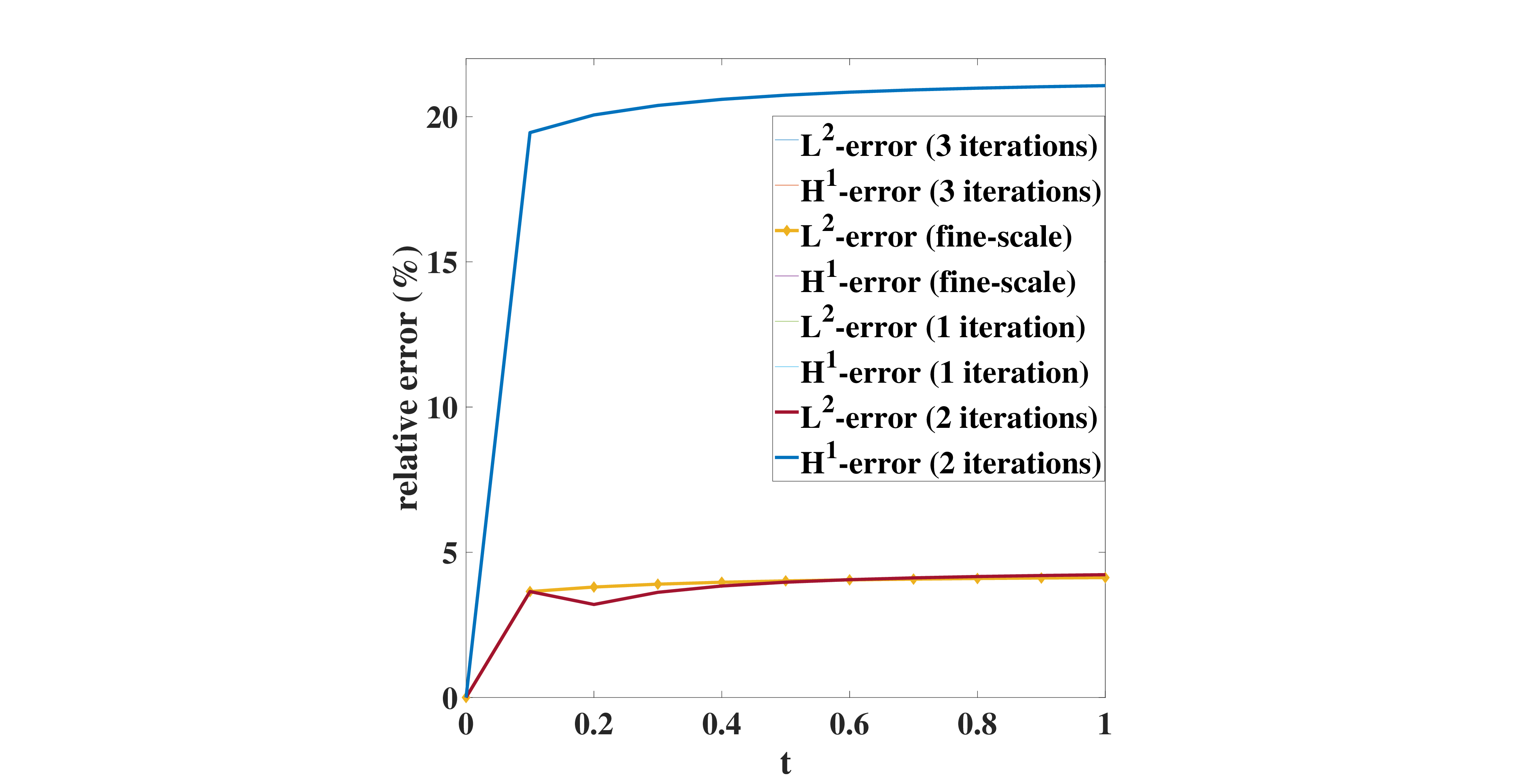}
\caption{$(\epsilon, N_{\text{exp}})=(0.0014,44)$}
\end{subfigure}
\begin{subfigure}[b]{0.3\textwidth}
\centering
\includegraphics[width=1.3\textwidth,trim={9cm 0.5cm 5cm 1cm},clip]{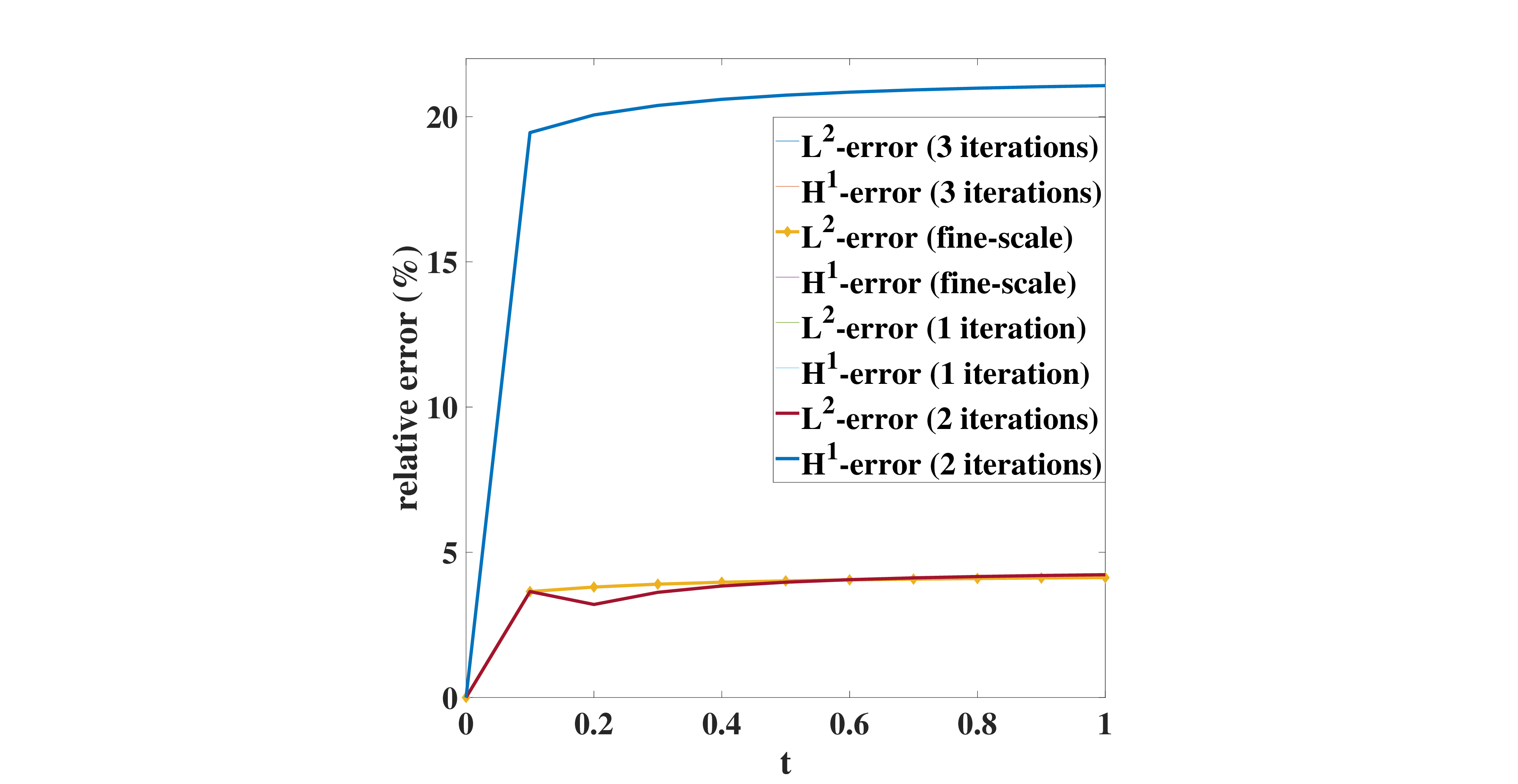}
\caption{$(\epsilon, N_{\text{exp}})=(1.3332,28)$}
\end{subfigure}
\begin{subfigure}[b]{0.3\textwidth}
\centering
\includegraphics[width=1.3\textwidth,trim={9cm 0.5cm 5cm 1cm},clip]{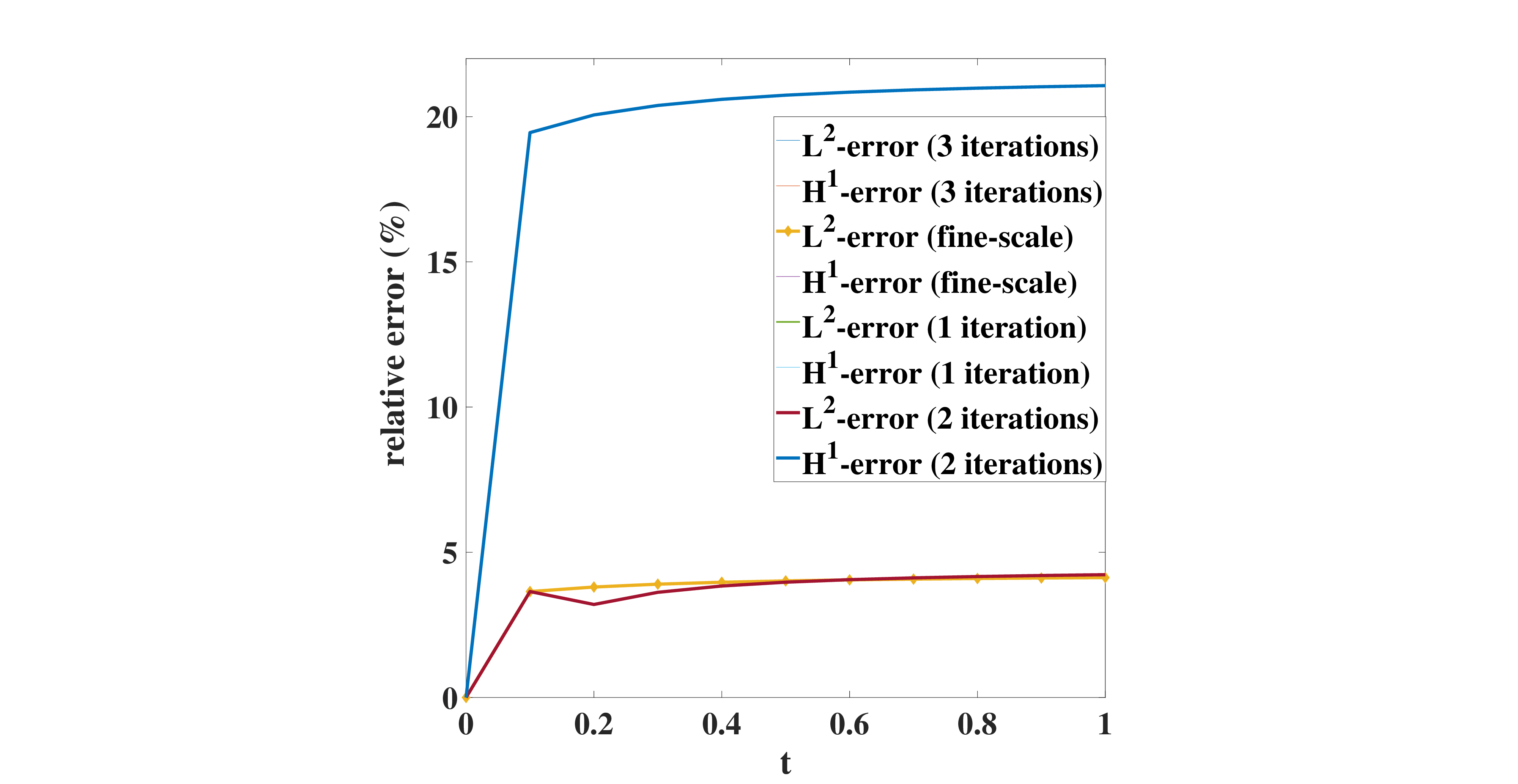}
\caption{$(\epsilon, N_{\text{exp}})=(18.4317,22)$.}
\end{subfigure}
\caption{$\alpha=0.1$, $\epsilon^*=0.5$.}
\label{fig:0.1}
\end{figure}

\subsection{The performance of Algorithm \ref{algorithm:wavelet+parareal} for long time}
We test our proposed algorithm for longer time in this section with $T=10$. The parameters of $\epsilon$ and $\nexp$ are determined by Lemma \ref{lemma:stability} and Corollary \ref{corollary:soe}. 
Similar convergence behavior is observed as in the previous section, which is shown in Figure \ref{fig:largetime}. For example, the $L^2(D)$-relative error is below 5\% for $\alpha=0.1$ and $\alpha=0.5$ after 3 iteration. However, the performance for $\alpha=0.9$ is slightly worse with $L^2(D)$-relative error below 10\% after 3 iteration. In this case, the performance of multiscale-SOE-scheme has a slightly larger error as observed from Figure \ref{fig:largetime3}.

\begin{figure} [H]
\centering
\begin{subfigure}[b]{0.3\textwidth}
\centering
\includegraphics[width=1.3\textwidth,trim={9cm 0.2cm 5cm .5cm},clip]{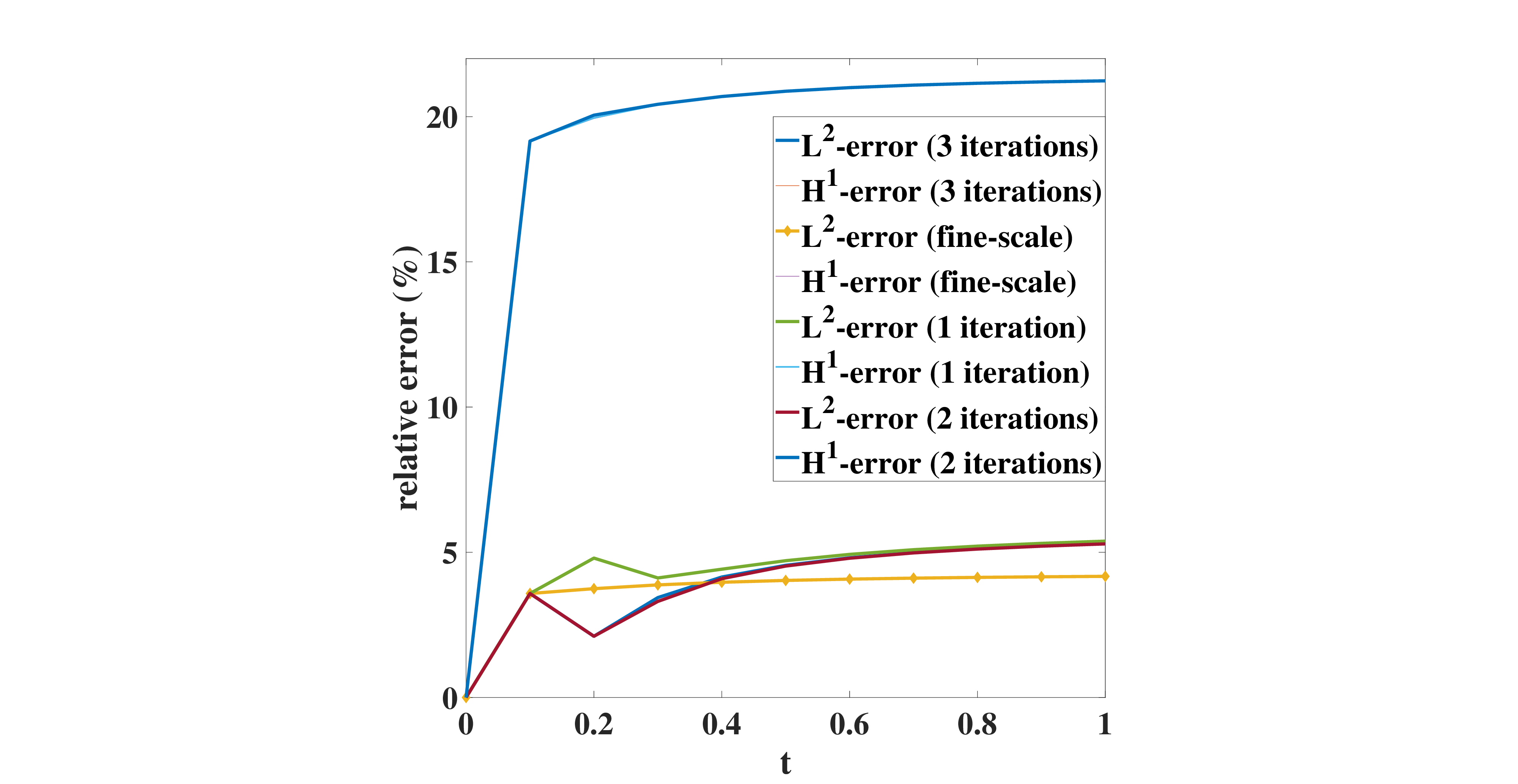}
\caption{$(\epsilon, N_{\text{exp}})=(0.0014,55)$}
\end{subfigure}
\begin{subfigure}[b]{0.3\textwidth}
\centering
\includegraphics[width=1.3\textwidth,trim={9cm 0.5cm 5cm 1cm},clip]{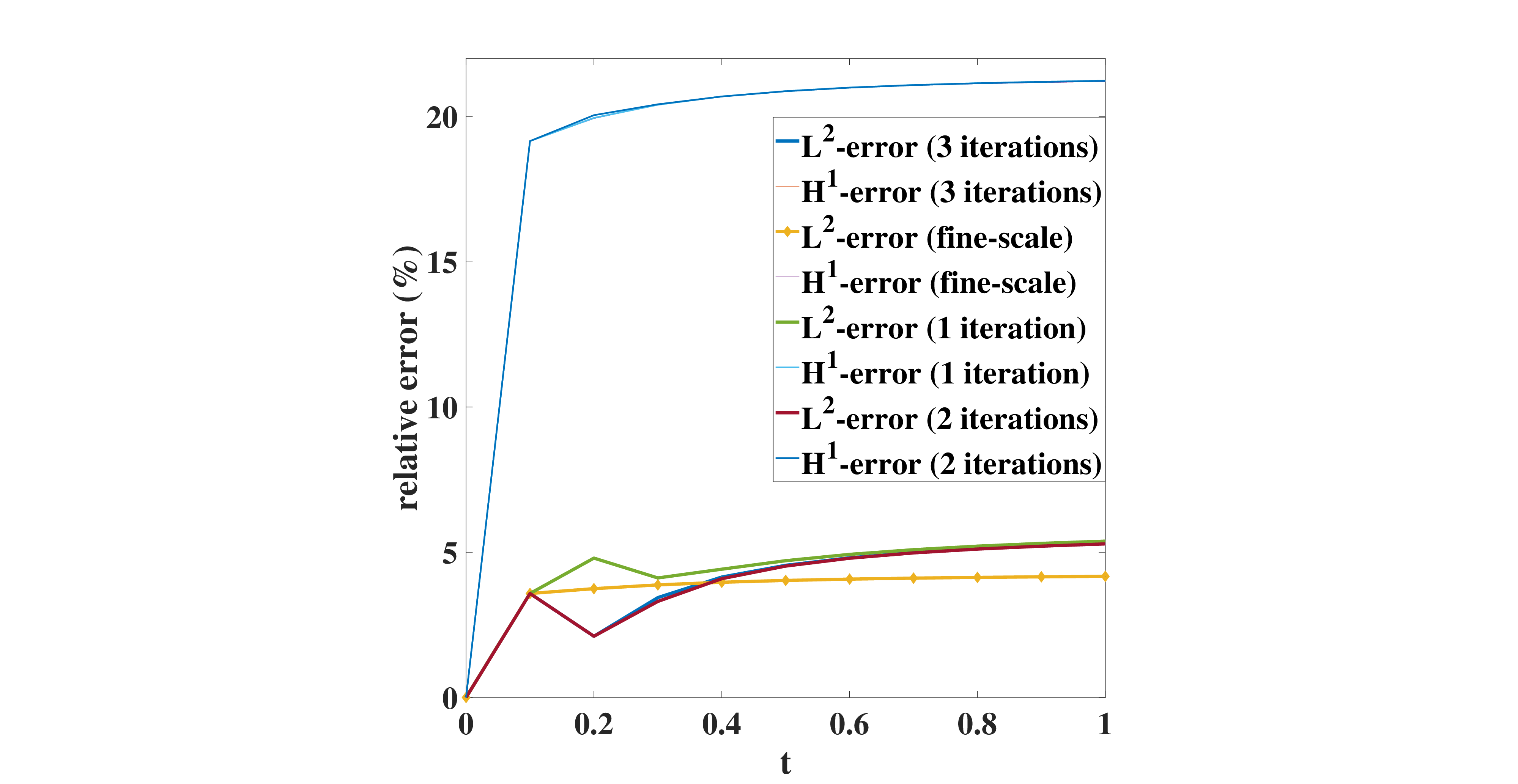}
\caption{$(\epsilon, N_{\text{exp}})=(1.6801,38)$}
\end{subfigure}
\begin{subfigure}[b]{0.3\textwidth}
\centering
\includegraphics[width=1.3\textwidth,trim={9cm 0.5cm 5cm 1cm},clip]{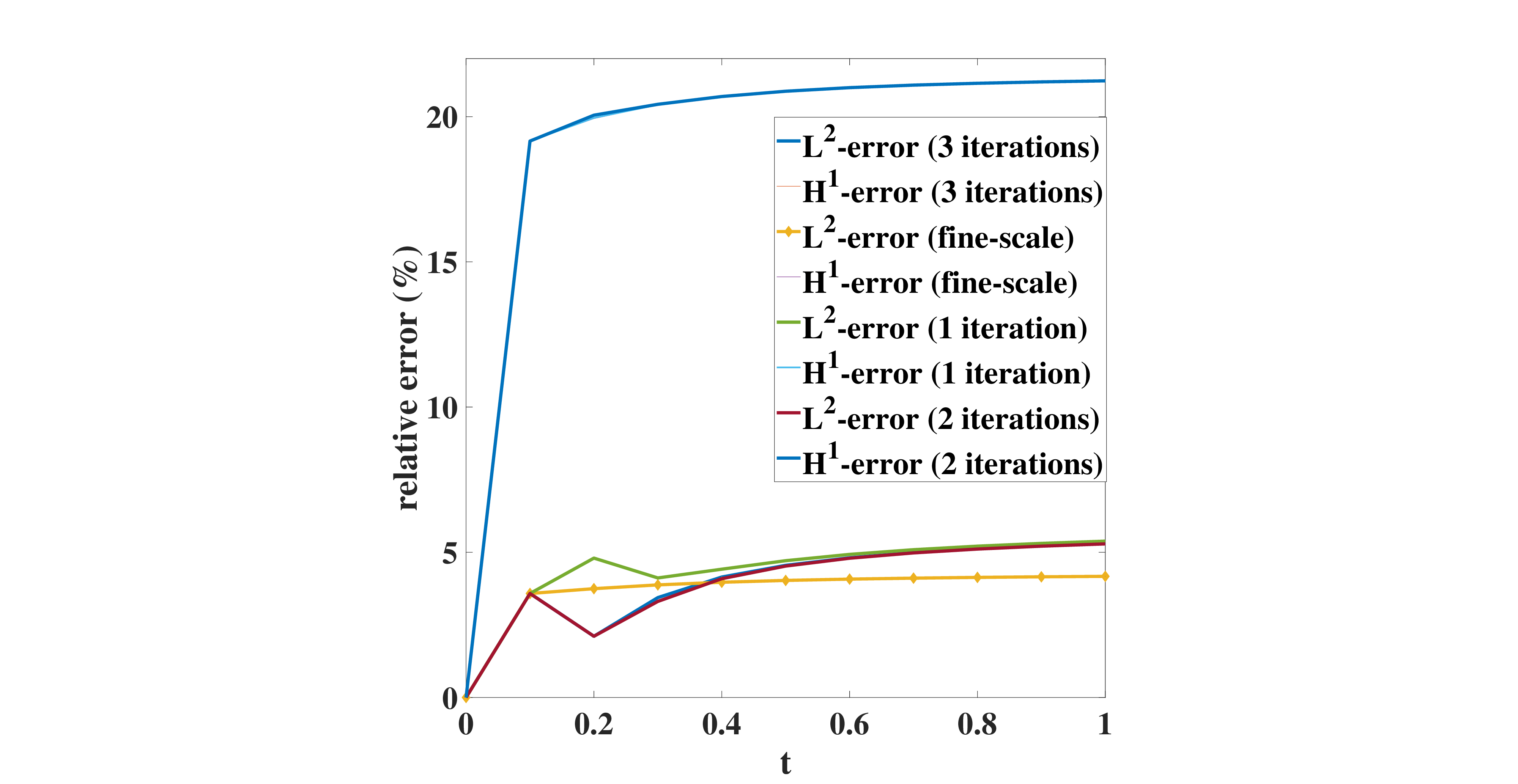}
\caption{$(\epsilon, N_{\text{exp}})=(19.629,31)$.}
\end{subfigure}
\caption{$\alpha=0.5$, $\epsilon^*=0.5$.}
\label{fig:0.5}
\end{figure}

\begin{figure} [H]
\centering
\begin{subfigure}[b]{0.3\textwidth}
\centering
\includegraphics[width=1.3\textwidth,trim={9cm 0.2cm 5cm .5cm},clip]{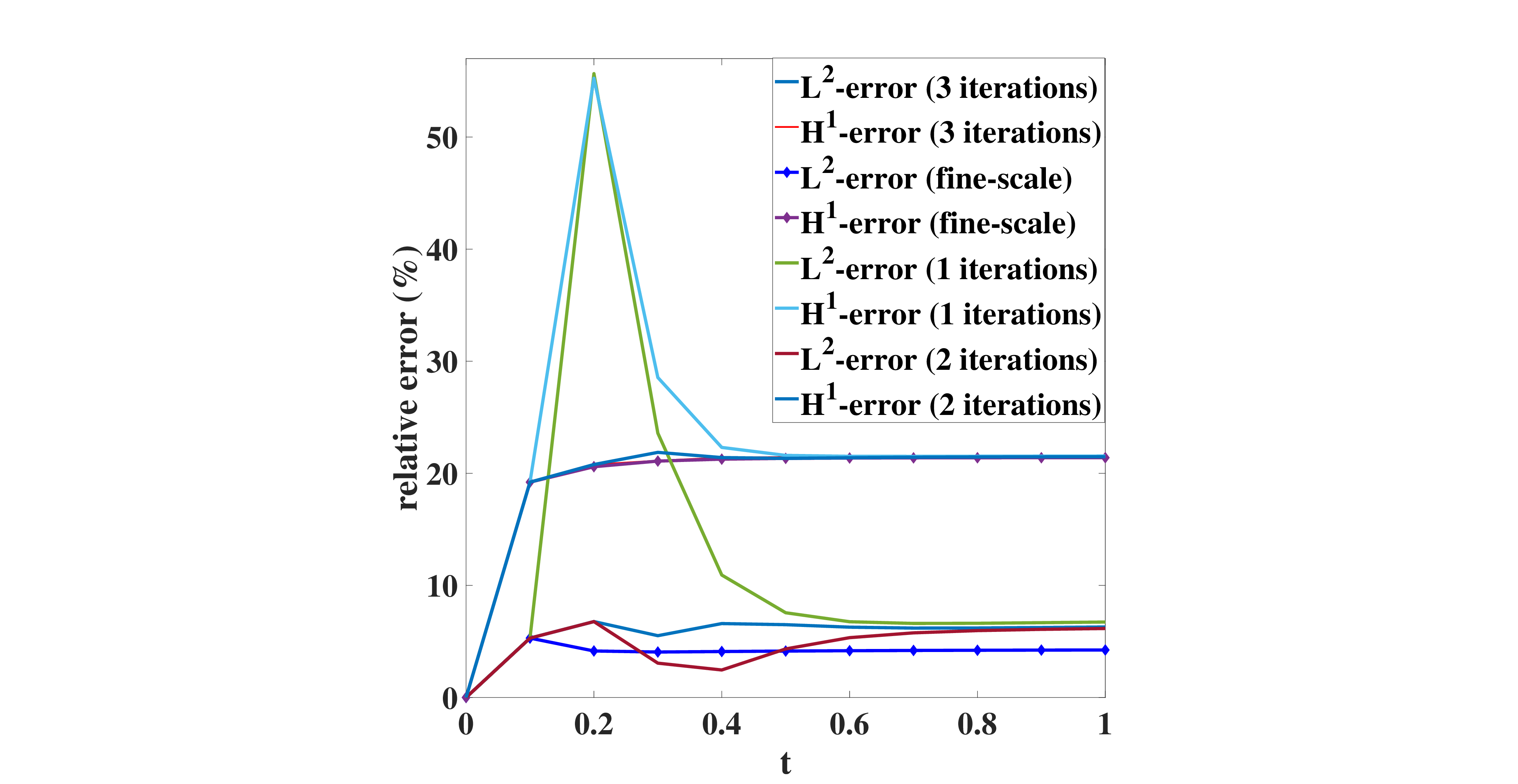}
\caption{$(\epsilon, N_{\text{exp}})=(0.0025,64)$}
\end{subfigure}
\begin{subfigure}[b]{0.3\textwidth}
\centering
\includegraphics[width=1.3\textwidth,trim={9cm 0.5cm 5cm 1cm},clip]{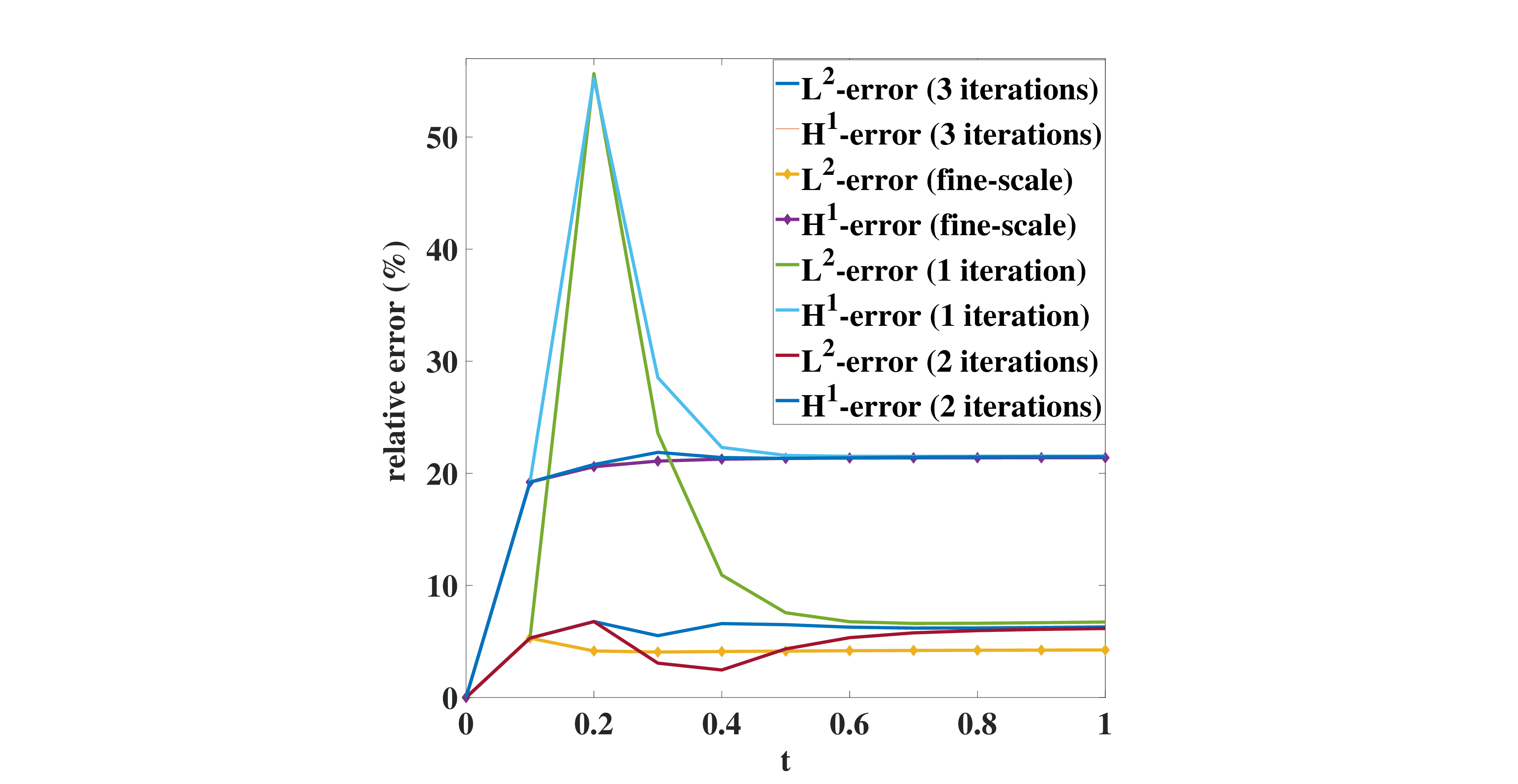}
\caption{$(\epsilon, N_{\text{exp}})=(2.6387,47)$ }
\end{subfigure}
\begin{subfigure}[b]{0.3\textwidth}
\centering
\includegraphics[width=1.3\textwidth,trim={9cm 0.5cm 5cm 1cm},clip]{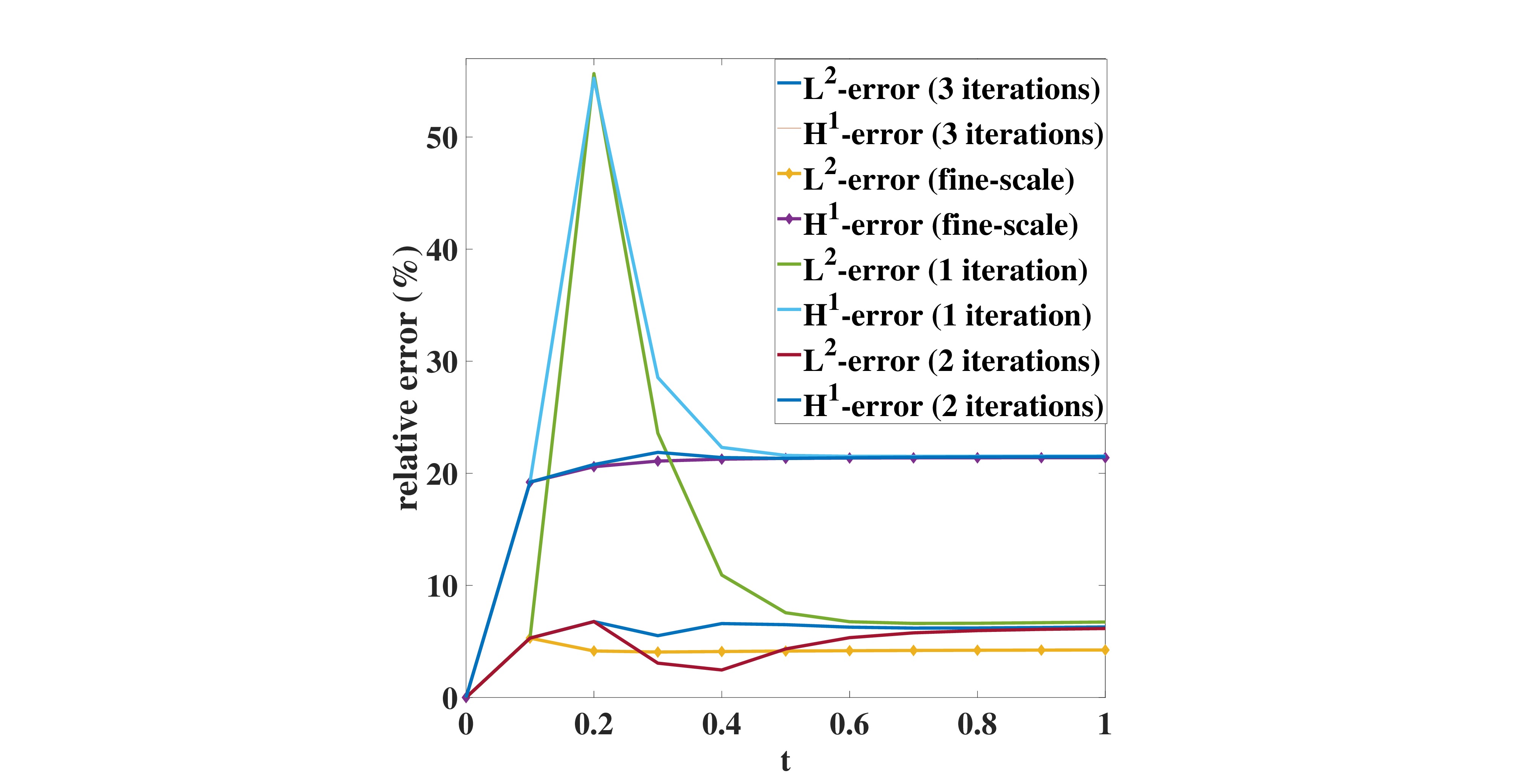}
\caption{$(\epsilon, N_{\text{exp}})=(41.9406,38)$.}
\end{subfigure}
\caption{$\alpha=0.9$, $\epsilon^*=0.5$.}
\label{fig:0.9}
\end{figure}

\begin{figure} [H]
\centering
\begin{subfigure}[b]{0.3\textwidth}
\centering
\includegraphics[width=1.3\textwidth,height=0.81\textwidth, trim={9cm 0.2cm 5cm .5cm},clip]{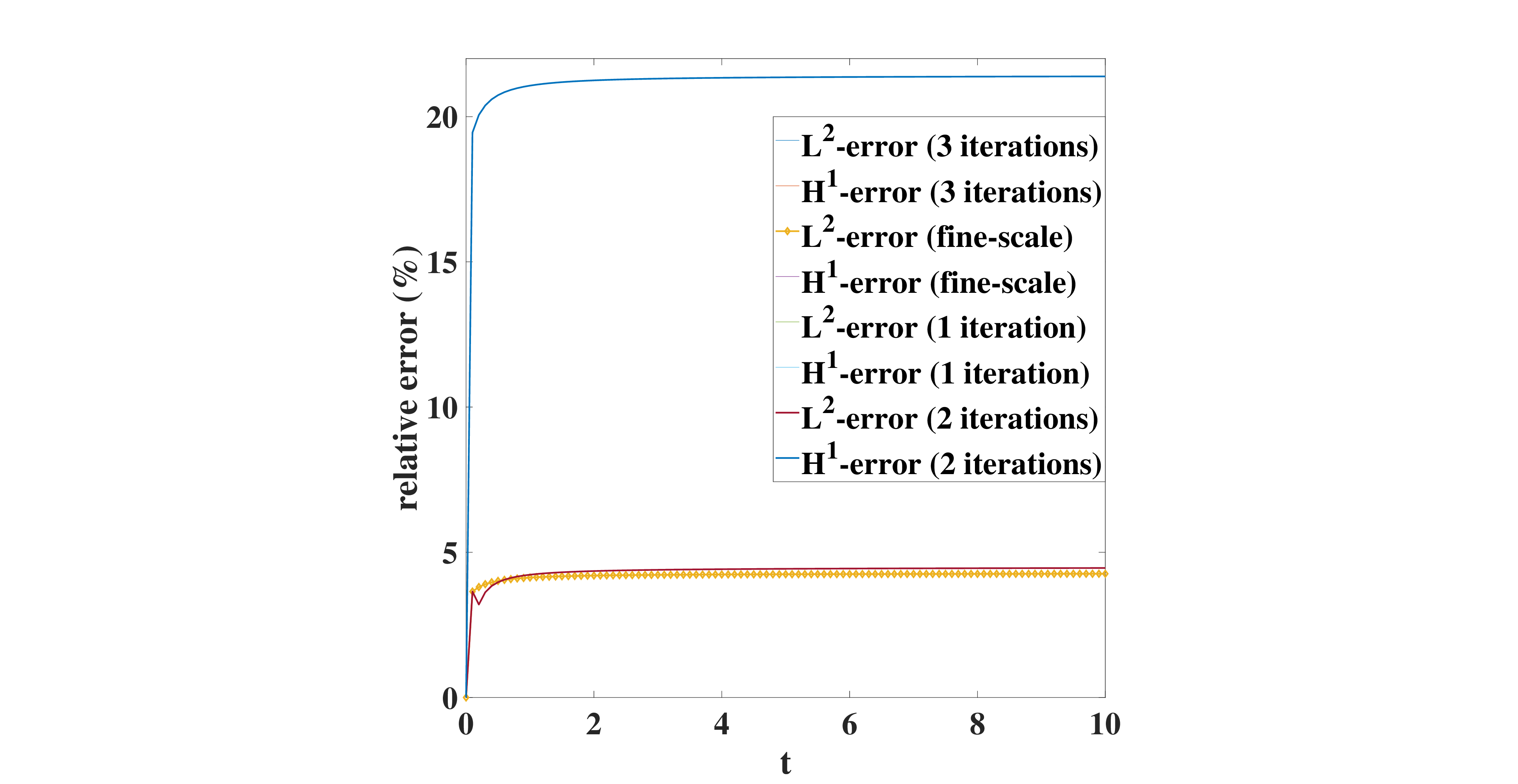}
\caption{$\alpha=0.1$}
\label{fig:largetime1}
\end{subfigure}
\begin{subfigure}[b]{0.3\textwidth}
\centering
\includegraphics[width=1.3\textwidth,trim={9cm 0.5cm 5cm 1cm},clip]{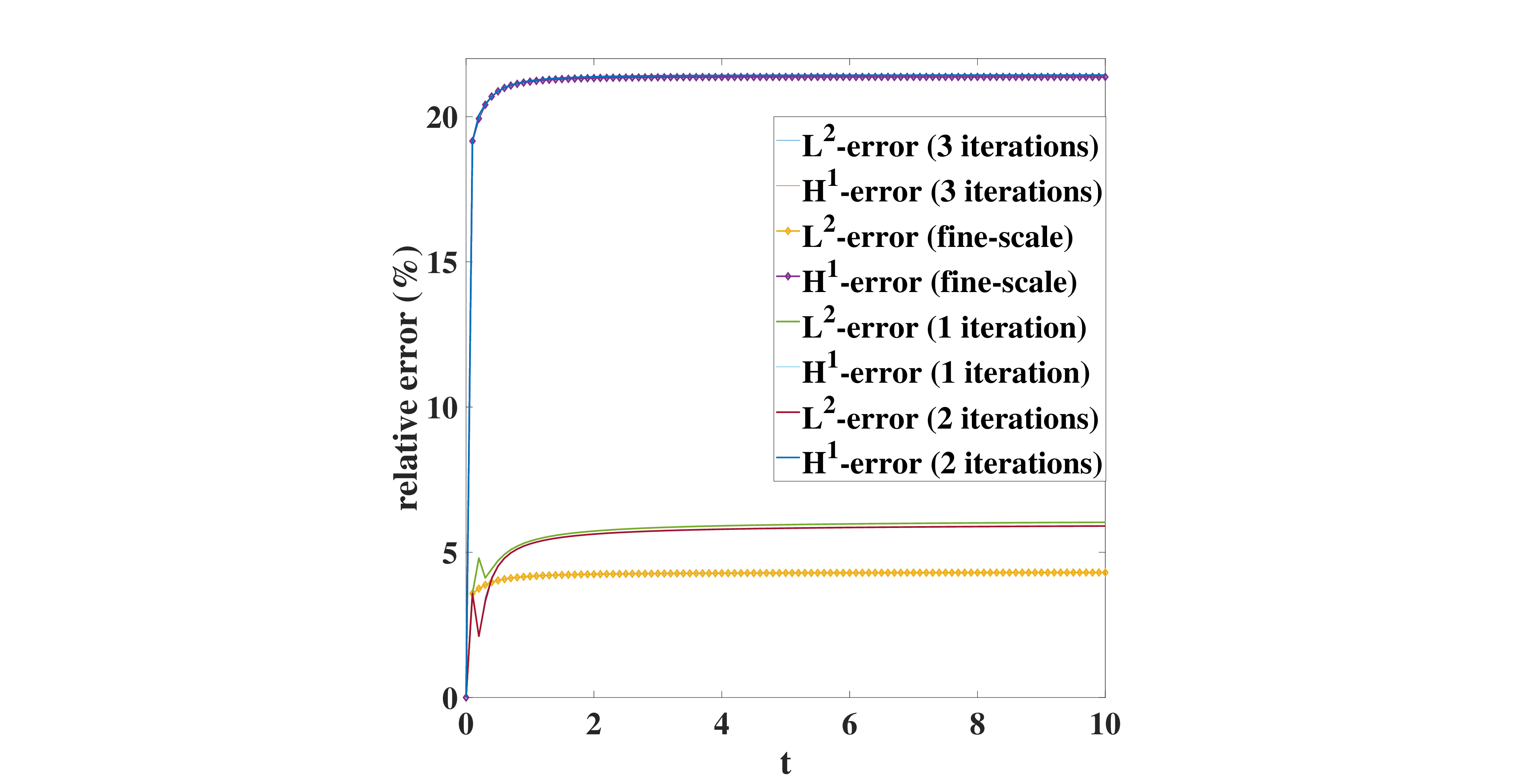}
\caption{$\alpha=0.5$} 
\label{fig:largetime2}
\end{subfigure}
\begin{subfigure}[b]{0.3\textwidth}
\centering
\includegraphics[width=1.3\textwidth,height=0.81\textwidth,trim={9cm 0.5cm 5cm 1cm},clip]{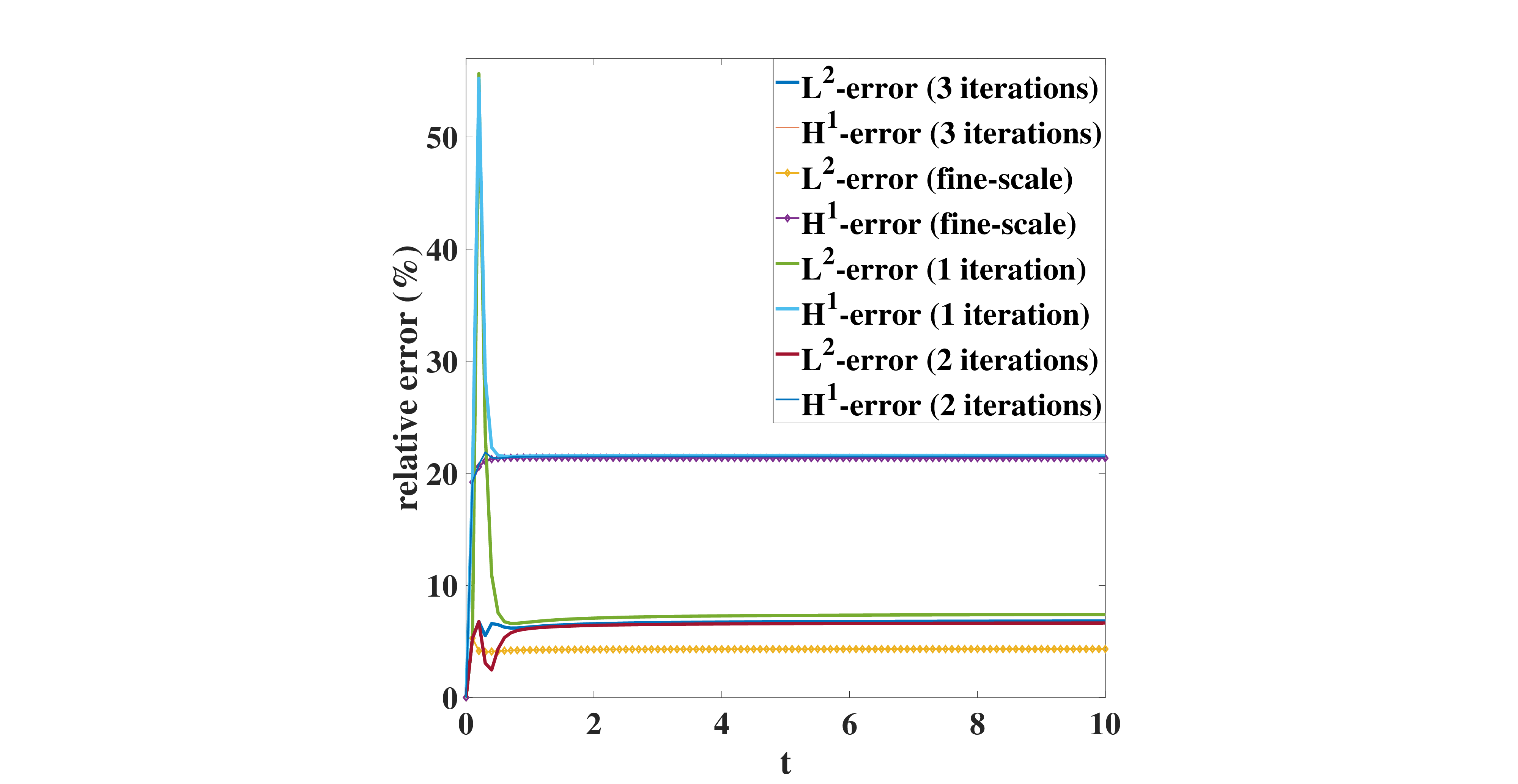}
\caption{$\alpha=0.9$} 
\label{fig:largetime3}
\end{subfigure}
\caption{$T=10$.}
\label{fig:largetime}
\end{figure}

\section{Conclusion}\label{sec:conclusion}
We propose in this paper a new efficient algorithm for time-fractional diffusion equations with heterogeneous coefficients. This algorithm is named as the Wavelet-based Edge Multiscale Parareal (WEMP) Algorithm, which can reduce computational cost and storage cost significantly. Moreover, it is suitable for long time simulation when multiple processors are available. We want to remark that the total number of terms in the exponential sums can be reduced hugely for a given accuracy requirement \eqref{eq:trunc-long} if an adaptive representation instead of a uniform representation is utilized. Consequently, this will lead to a more efficient algorithm for more challenging applications, for example, for much longer time simulation.
		
\bibliographystyle{abbrv}
\bibliography{reference}

\begin{thebibliography}{10}

\bibitem{altmann2021numerical}
R.~Altmann, P.~Henning, and D.~Peterseim.
\newblock Numerical homogenization beyond scale separation.
\newblock {\em Acta Numerica}, 30:1--86, 2021.

\bibitem{2016PararealMultiscale}
G.~Ariel, S.~J. Kim, and R.~Tsai.
\newblock Parareal multiscale methods for highly oscillatory dynamical systems.
\newblock {\em SIAM J. Sci. Comput.}, 38(6):A3540--A3564, 2016.

\bibitem{bal2005convergence}
G.~Bal.
\newblock On the convergence and the stability of the parareal algorithm to
  solve partial differential equations.
\newblock In {\em Domain decomposition methods in science and engineering},
  pages 425--432. Springer, 2005.

\bibitem{Bal_Mayday_AmericanPut}
G.~Bal and Y.~Maday.
\newblock A ``parareal'' time discretization for non-linear {PDE}'s with
  application to the pricing of an {A}merican put.
\newblock In {\em Recent developments in domain decomposition methods
  ({Z}\"{u}rich, 2001)}, volume~23 of {\em Lect. Notes Comput. Sci. Eng.},
  pages 189--202. Springer, Berlin, 2002.

\bibitem{MR2721592}
L.~Berlyand and H.~Owhadi.
\newblock Flux norm approach to finite dimensional homogenization
  approximations with non-separated scales and high contrast.
\newblock {\em Arch. Ration. Mech. Anal.}, 198(2):677--721, 2010.

\bibitem{parareal_ChemicalKinetics2010}
A.~Blouza, L.~Boudin, and S.~M. Kaber.
\newblock Parallel in time algorithms with reduction methods for solving
  chemical kinetics.
\newblock {\em Commun. Appl. Math. Comput. Sci.}, 5(2):241--263, 2010.

\bibitem{braess2005approximation}
D.~Braess and W.~Hackbusch.
\newblock Approximation of $1/x$ by exponential sums in {$[1,\infty)$}.
\newblock {\em IMA journal of numerical analysis}, 25(4):685--697, 2005.

\bibitem{parareal_integrator_2010}
A.~Christlieb, C.~Macdonald, and B.~Ong.
\newblock Parallel high-order integrators.
\newblock {\em SIAM J. Sci. Comput.}, 32(2):818--835, 2010.

\bibitem{MR1979846}
W.~E and B.~Engquist.
\newblock The heterogeneous multiscale methods.
\newblock {\em Commun. Math. Sci.}, 1(1):87--132, 2003.

\bibitem{egh12}
Y.~Efendiev, J.~Galvis, and T.~Hou.
\newblock Generalized multiscale finite element methods.
\newblock {\em J. Comput. Phys.}, 251:116--135, 2013.

\bibitem{parareal_multigrid_2014}
R.~Falgout, S.~Friedhoff, T.~Kolev, S.~MacLachlan, and J.~Schroder.
\newblock Parallel time integration with multigrid.
\newblock {\em SIAM J. Sci. Comput.}, 36(6):C635--C661, 2014.

\bibitem{time_decomposed_parallel2003}
C.~Farhat and M.~Chandesris.
\newblock Time-decomposed parallel time-integrators: theory and feasibility
  studies for fluid, structure, and fluid-structure applications.
\newblock {\em Internat. J. Numer. Methods Engrg.}, 58(9):1397--1434, 2003.

\bibitem{time_parallel2006}
C.~Farhat, J.~Cortial, C.~Dastillung, and H.~Bavestrello.
\newblock Time-parallel implicit integrators for the near-real-time prediction
  of linear structural dynamic responses.
\newblock {\em Internat. J. Numer. Methods Engrg.}, 67(5):697--724, 2006.

\bibitem{parareal_NS_mayday}
P.~Fischer, F.~Hecht, and Y.~Maday.
\newblock A parareal in time semi-implicit approximation of the
  {N}avier-{S}tokes equations.
\newblock In {\em Domain decomposition methods in science and engineering},
  volume~40 of {\em Lect. Notes Comput. Sci. Eng.}, pages 433--440. Springer,
  Berlin, 2005.

\bibitem{fu2018}
S.~Fu, E.~Chung, and G.~Li.
\newblock Edge multiscale methods for elliptic problems with heterogeneous
  coefficients.
\newblock {\em J. Comput. Phys}, 369(1):228--242, 2019.

\bibitem{fu2019wavelet}
S.~Fu, G.~Li, R.~Craster, and S.~Guenneau.
\newblock Wavelet-based edge multiscale finite element method for helmholtz
  problems in perforated domains.
\newblock {\em To appear in Multiscale Model. Simul.}, 2021.

\bibitem{MR1455261}
T.~Hou and X.-H. Wu.
\newblock A multiscale finite element method for elliptic problems in composite
  materials and porous media.
\newblock {\em J. Comput. Phys.}, 134(1):169--189, 1997.

\bibitem{MR1660141}
T.~Hughes, G.~Feij\'oo, L.~Mazzei, and J.-B. Quincy.
\newblock The variational multiscale method---a paradigm for computational
  mechanics.
\newblock {\em Comput. Methods Appl. Mech. Engrg.}, 166(1-2):3--24, 1998.

\bibitem{jiang2017fast}
S.~Jiang, J.~Zhang, Q.~Zhang, and Z.~Zhang.
\newblock Fast evaluation of the caputo fractional derivative and its
  applications to fractional diffusion equations.
\newblock {\em Communications in Computational Physics}, 21(3):650--678, 2017.

\bibitem{MR3601002}
B.~Jin and Z.~Zhou.
\newblock An analysis of {G}alerkin proper orthogonal decomposition for
  subdiffusion.
\newblock {\em ESAIM Math. Model. Numer. Anal.}, 51(1):89--113, 2017.

\bibitem{KilbasSrivastavaTrujillo:2006}
A.~Kilbas, H.~Srivastava, and J.~Trujillo.
\newblock {\em Theory and Applications of Fractional Differential Equations},
  volume 204 of {\em North-Holland Mathematics Studies}.
\newblock Elsevier Science B.V., Amsterdam, 2006.

\bibitem{2013Micro_Macro_parareal}
F.~Legoll, T.~Leli\`evre, and G.~Samaey.
\newblock A micro-macro parareal algorithm: application to singularly perturbed
  ordinary differential equations.
\newblock {\em SIAM J. Sci. Comput.}, 35(4):A1951--A1986, 2013.

\bibitem{li2019convergence}
G.~Li.
\newblock On the convergence rates of gmsfems for heterogeneous elliptic
  problems without oversampling techniques.
\newblock {\em Multiscale Model. Simul}, 17(2):593--619, 2019.

\bibitem{li2020wavelet}
G.~Li and J.~Hu.
\newblock Wavelet-based edge multiscale parareal algorithm for parabolic
  equations with heterogeneous coefficients.
\newblock {\em J. Comput. Phys}, 444, 2021.

\bibitem{li2017error}
G.~Li, D.~Peterseim, and M.~Schedensack.
\newblock Error analysis of a variational multiscale stabilization for
  convection-dominated diffusion equations in two dimensions.
\newblock {\em IMA J. Numer. Anal.}, 38(3):1229--1253, 2018.

\bibitem{LIN20071533}
Y.~Lin and C.~Xu.
\newblock Finite difference/spectral approximations for the time-fractional
  diffusion equation.
\newblock {\em Journal of Computational Physics}, 225(2):1533--1552, 2007.

\bibitem{lions_mayday_2001_parareal}
J.-L. Lions, Y.~Maday, and G.~Turinici.
\newblock R\'{e}solution d'{EDP} par un sch\'{e}ma en temps ``parar\'{e}el''.
\newblock {\em C. R. Acad. Sci. Paris S\'{e}r. I Math.}, 332(7):661--668, 2001.

\bibitem{parareal_ocean_model2008}
Y.~Liu and J.~Hu.
\newblock Modified propagators of parareal in time algorithm and application to
  {P}rinceton ocean model.
\newblock {\em Internat. J. Numer. Methods Fluids}, 57(12):1793--1804, 2008.

\bibitem{parareal_ModelReduction_mayday2007}
Y.~Maday.
\newblock Parareal in time algorithm for kinetic systems based on model
  reduction.
\newblock In {\em High-dimensional partial differential equations in science
  and engineering}, volume~41 of {\em CRM Proc. Lecture Notes}, pages 183--194.
  Amer. Math. Soc., Providence, RI, 2007.

\bibitem{parareal_QuantumSystem_mayday2007}
Y.~Maday, J.~Salomon, and G.~Turinici.
\newblock Monotonic parareal control for quantum systems.
\newblock {\em SIAM J. Numer. Anal.}, 45(6):2468--2482, 2007.

\bibitem{MR3246801}
A.~M{\aa}lqvist and D.~Peterseim.
\newblock Localization of elliptic multiscale problems.
\newblock {\em Math. Comp.}, 83(290):2583--2603, 2014.

\bibitem{Podlubnybook}
I.~Podlubny.
\newblock {\em Fractional Differential Equations}, volume 198 of {\em
  Mathematics in Science and Engineering}.
\newblock Academic Press, Inc., San Diego, CA, 1999.
\newblock An introduction to fractional derivatives, fractional differential
  equations, to methods of their solution and some of their applications.

\bibitem{Sakamoto-Yamamoto11}
K.~Sakamoto and M.~Yamamoto.
\newblock Initial value/boundary value problems for fractional diffusion-wave
  equations and applications to some inverse problems.
\newblock {\em J. Math. Anal. Appl.}, 382(1):426--447, 2011.

\bibitem{MR1226236}
F.~Stenger.
\newblock {\em Numerical methods based on sinc and analytic functions},
  volume~20 of {\em Springer Series in Computational Mathematics}.
\newblock Springer-Verlag, New York, 1993.

\bibitem{Babuska2}
T.~Strouboulis, I.~Babu{\v{s}}ka, and K.~Copps.
\newblock The design and analysis of the generalized finite element method.
\newblock {\em Comput. Methods Appl. Mech. Engrg.}, 181:43--69, 2000.

\bibitem{SUN2006193}
Z.~Sun and X.~Wu.
\newblock A fully discrete difference scheme for a diffusion-wave system.
\newblock {\em Applied Numerical Mathematics}, 56(2):193--209, 2006.

\bibitem{thomee1984galerkin}
V.~Thom{\'e}e.
\newblock {\em Galerkin finite element methods for parabolic problems}, volume
  1054.
\newblock Springer, 2006.

\bibitem{WU2018135}
S.-L. Wu and T.~Zhou.
\newblock Parareal algorithms with local time-integrators for time fractional
  differential equations.
\newblock {\em Journal of Computational Physics}, 358:135--149, 2018.

\bibitem{zhu2019fast}
H.~Zhu and C.~Xu.
\newblock A fast high order method for the time-fractional diffusion equation.
\newblock {\em SIAM Journal on Numerical Analysis}, 57(6):2829--2849, 2019.

\end{thebibliography}
\end{document}